\documentclass[10pt]{article}
\usepackage{amsmath,amssymb,amsfonts,amsthm}
\usepackage{graphicx}
\usepackage{pstricks}
\usepackage[colorlinks,citecolor=blue]{hyperref}
\usepackage[utf8]{inputenc}
\date{\today}
\usepackage{enumitem}
\usepackage{textcomp}
\usepackage{caption}
\usepackage{pstricks-add}
\usepackage{xspace}
\setenumerate{label*=\arabic*., leftmargin=0cm,widest=0.}%\parindent}
\setenumerate[1]{ref=\arabic*,leftmargin=1em, itemindent=0.7\labelsep}
\setenumerate[2]{ref=\theenumi.\arabic*, itemindent= 2.1\labelsep}
\setenumerate[3]{ref=\theenumii.\arabic*, itemindent=3.6\labelsep}
\setenumerate[4]{ref=\theenumiii.\arabic*, itemindent=5\labelsep}
\setitemize{noitemsep, leftmargin=2em}%\parindent}

\newtheorem{theorem}{Theorem}[section]
\newtheorem{lemma}[theorem]{Lemma}
\newtheorem{definition}[theorem]{Definition}
\newtheorem{corollary}[theorem]{Corollary}
\newtheorem{proposition}[theorem]{Proposition}
\newtheorem*{fact}{Fact}
\newtheorem*{fait}{Fact}

\newcommand{\Z}{\mathbb Z }

\newcommand{\N}{\mathbb N}

\newcommand{\GGBS}{$vGBS$\xspace}
\newcommand{\psa}{p.s.a.\xspace}
\newcommand{\psad}{p.s.a.d.\xspace}
\newcommand{\tor}{toric\xspace}
\newcommand{\pa}{pre-ascending\xspace}
\newcommand{\pde}{pre-descending\xspace}
\newcommand{\pbe}{potentially bi-elliptic\xspace}
\newcommand{\bea}{bearing\xspace}
\newcommand{\vGBS}{$vGBS$\xspace}
\newcommand{\GBS}{$GBS$\xspace}
\newcommand{\GBSn}[1]{$GBS_{#1}$\xspace}
\newcommand{\gr}[1]{\mathfrak{#1}}
\title{Compatibility JSJ decomposition of graphs of free abelian groups}
\author{Benjamin Beeker}
\begin{document}
\maketitle

\begin{abstract}
A group $G$ is a \vGBS group if it admits a decomposition as a finite graph of groups with all edge and vertex groups finitely generated and free abelian. We describe the compatibility JSJ decomposition over abelian groups. We prove that in general this decomposition is not algorithmically computable. 
\end{abstract}

\section{Introduction}
The theory of JSJ splittings starts with the work of Jaco-Shalen and Johansson on orientable irreducible closed 3-manifolds giving a canonical family of 2-dimensional tori. Kropholler first introduced the notion into group theory giving a JSJ decomposition for some Poincar\'e duality groups \cite{Krop}. Then Sela gave a construction for torsion-free hyperbolic groups \cite{Sela}. This notion has been more generally developed by Rips and Sela \cite{RiSe}, Dunwoody and Sageev \cite{DunSa2}, Fujiwara and Papasoglu \cite{FuPa} for various classes of groups. 

However, in general, given a group $G$ and a class of subgroups, there is not a unique JSJ splitting of $G$ over these subgroups.  Guirardel and Levitt define in \cite{Gl3a} the JSJ deformation spacse which generalize the previous notions. They also introduce in \cite{Gl3b} a new object called the compatibility JSJ splitting of $G$ over these subgroups. Unlike the usual JSJ splittings, the compatibility JSJ splitting is unique, and so is invariant under is automorphisms. However, it is often harder to construct it, the purpose of this article is to give some positive and negative results about the constructibility of this object.

In this paper we focus on the construction of the compatibility JSJ tree over abelian groups for the following class of groups. Let $G$ be a group acting on a simplicial tree $T$ with free abelian vertex (and edge) stabilizers. We call such a group a Generalized Baumslag-Solitar group of variable rank, or \vGBS group, and such a tree $T$ a \vGBS ($G$-)tree. We call \GBSn{n} the collection of groups which admits an action on a simplicial tree with vertex stabilizers $\Z^n$ with $n$ a fixed integer. This two classes of groups generalize the one of \GBSn{1} groups (or \GBS groups) introduced by Forester in \cite{for03} as examples of groups for which we do not have a canonical (usual) JSJ tree. 

Recall that an element of $G$ is \textit{elliptic} if it fixes a vertex of $T$, and \textit{hyperbolic} otherwise. Given a hyperbolic element $g\in G$, it acts by translation on a line $\mathcal A_g$ of $T$ called the \textit{axis} of $g$ or its \textit{characteristic space}. The characteristic space of an elliptic element is the set of its fixed points.
A subgroup of $G$ is \textit{elliptic}, if it is included in the stabilizer of a vertex. A subgroup of $G$ is \textit{universally elliptic} if it is elliptic in every $G$-tree.
With no restriction on the $G$-trees, almost any tree can be universally elliptic, we thus consider $G$-trees whose edge groups are included in a set $\mathcal A$ of subgroups of $G$. We then talk about $G$-trees over $\mathcal A$, and a subgroup is $\mathcal A$-universally elliptic if it is elliptic in every  $G$-tree over $\mathcal A$.

Given two $G$-trees $T$ and $T'$, the tree $T$ \textit{refines} $T'$ if $T'$ may be obtained from $T$ by equivariantly collapsing edges. The tree $T$  \textit{dominates} $T'$ if every elliptic group of $T$ is elliptic in $T'$. In particular, if $T$ refines $T'$ then it also dominates it.
We say that $T$ and $T'$ are  \textit{compatible} if there exists a $G$-tree $T''$ which refines both $T$ and $T'$.
A $G$-tree is  \textit{$\mathcal A$-universally compatible} if it is compatible with every $G$-tree (over $\mathcal A$).

A $G$-tree over $\mathcal A$ is a \textit{JSJ tree} over  $\mathcal A$ if it is $\mathcal A$-universally elliptic and dominates every universally elliptic $G$-tree over $\mathcal A$. A $G$-tree over $\mathcal A$ is a \textit{compatibility JSJ tree} over  $\mathcal A$ if it is $\mathcal A$-universally compatible and dominates every universally compatible $G$-tree over $\mathcal A$.

The present paper splits into two parts. We first describe the compatibility JSJ tree  over $\Z^n$ groups of the \GBSn{n} groups. In the second part, we describe the compatibility JSJ tree over abelian groups of the \vGBS groups.

Let $G$ be a \GBSn{n} group. We propose to describe the compatibility JSJ tree over $\Z^n$ groups in the following sense. Starting from a JSJ tree $T$ of $G$ over $\Z^n$ groups (except for some degenerated cases, any \GBSn{n} tree is a JSJ tree \cite{Moi1}), we explicit a set of edges and a set of vertices such that the compatibility JSJ tree is obtained by expanding (in a precise way) these vertices and collapsing these edges in $T$.

The compatibility relation is very restrictive. For example, two $G$-trees with an isomorphic quotient graph of groups may be not compatible. Take the \GBS group $\Z\ast_{2\Z}\Z\ast_{4\Z}\Z$. Its JSJ deformation space contains infinitely many reduced trees and any reduced JSJ tree of this deformation space have same isomorphic quotient graphs of groups. However, each of these reduced JSJ tree is compatible with exactly two others reduced JSJ trees.

Here are the two main examples of edges to collapse. The labelled graphs represent \GBS groups, and must be understood in the following manner: all vertices and edges carry the infinite cyclic group $\Z$, the number $n$ at the end of each edge indicates that the injection of the edge group into the group of its endpoint is $k\mapsto nk$.

In Figure \ref{exempleslip}, the two graphs of groups have isomorphic fundamental group, and are related by a slide along the edge denoted by $\gr e$. In this case, the edge $\gr e$ is called \textit{slippery} (see Section \ref{prelim} for the complete definition) and is collapsed in the compatibility JSJ. In fact, after collapsing the edge $\gr e$, we exactly obtain a compatibility JSJ over group $\Z$.
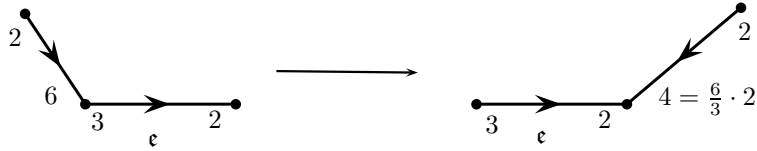
\begin{figure}[!ht]
\begin{center}
% Generated with LaTeXDraw 2.0.5
% Sun Jun 19 16:41:39 CEST 2011
% \usepackage[usenames,dvipsnames]{pstricks}
% \usepackage{epsfig}
% \usepackage{pst-grad} % For gradients
% \usepackage{pst-plot} % For axes
\scalebox{1} % Change this value to rescale the drawing.
{
\begin{pspicture}(0,-1.0459619)(10.242812,1.0259618)
\definecolor{color54b}{rgb}{0.00392156862745098,0.00392156862745098,0.00392156862745098}
\psline[ArrowInside=->, ArrowInsidePos=0.5,arrowsize=0.25291667cm,linewidth=0.04cm](0.36504862,0.8518507)(1.1650486,-0.3481493)
\psline[ArrowInside=->, ArrowInsidePos=0.5,arrowsize=0.25291667cm,linewidth=0.04cm](1.1650486,-0.3481493)(3.1650486,-0.3481493)
\psline[ArrowInside=->, ArrowInsidePos=0.5,arrowsize=0.25291667cm,linewidth=0.04cm](9.880938,0.9359618)(8.365048,-0.3481493)
\psline[ArrowInside=->, ArrowInsidePos=0.5,arrowsize=0.25291667cm,linewidth=0.04cm](6.3650484,-0.3481493)(8.365048,-0.3481493)
\psdots[dotsize=0.14](8.365048,-0.3481493)
\psdots[dotsize=0.14](6.3650484,-0.3481493)
\psdots[dotsize=0.14](3.1650486,-0.3481493)
\psdots[dotsize=0.14](1.1650486,-0.3481493)
\psdots[dotsize=0.14](0.36504862,0.8518507)
\usefont{T1}{ptm}{m}{n}
\rput(2.067861,-0.8231493){$\gr e$}
\usefont{T1}{ptm}{m}{n}
\rput(7.227861,-0.7831493){$\gr e$}
\psline[linewidth=0.03cm, arrowsize=0.05291667cm 2.0,arrowlength=1.4,arrowinset=0.4]{->}(3.7050486,0.0918507)(5.5850487,0.071850695)
\psdots[dotsize=0.14](9.880938,0.9359618)
\usefont{T1}{ptm}{m}{n}
\rput(1.3323437,-0.5790382){$3$}
\usefont{T1}{ptm}{m}{n}
\rput(6.572344,-0.6190382){$3$}
\usefont{T1}{ptm}{m}{n}
\rput(2.8923438,-0.5590382){$2$}
\usefont{T1}{ptm}{m}{n}
\rput(8.072344,-0.5590382){$2$}
\usefont{T1}{ptm}{m}{n}
\rput(0.71234375,-0.2390382){$6$}
\usefont{T1}{ptm}{m}{n}
\rput(9.432344,-0.2790382){$4=\frac{6}{3}\cdot 2$}
\usefont{T1}{ptm}{m}{n}
\rput(9.9323435,0.6209618){$2$}
\usefont{T1}{ptm}{m}{n}
\rput(0.23234375,0.5409618){$2$}
\end{pspicture} 
}
\end{center}
\caption{The two graphs of groups are related by a slide.}
\label{exempleslip}
\end{figure}

In Figure \ref{exemplepa}, the two graphs have isomorphic fundamental groups, and are related by a deformation called \textit{induction along $\gr e$} (see definition in Section  \ref{prelim}). In this case, the edge $\gr e$ is called \textit{strictly ascending}. Collapsing $\gr e$, we obtain a compatibility JSJ.

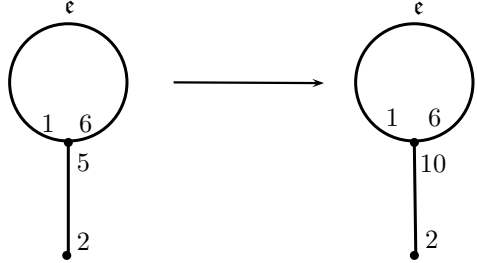
\begin{figure}[!ht]
\begin{center}
% Generated with LaTeXDraw 2.0.5
% Sun Jun 19 18:18:16 CEST 2011
% \usepackage[usenames,dvipsnames]{pstricks}
% \usepackage{epsfig}
% \usepackage{pst-grad} % For gradients
% \usepackage{pst-plot} % For axes
\scalebox{1} % Change this value to rescale the drawing.
{
\begin{pspicture}(0,-1.7800001)(6.301875,1.8000001)
\pscircle[linewidth=0.04,dimen=outer](0.8,0.60156256){0.8}
\psdots[dotsize=0.12](0.8,-0.19843745)
\psline[linewidth=0.03cm,arrowsize=0.05291667cm 2.0,arrowlength=1.4,arrowinset=0.4]{->}(2.2,0.60156256)(4.2,0.5918751)
\pscircle[linewidth=0.04,dimen=outer](5.4,0.60156256){0.8}
\psdots[dotsize=0.12](5.4,-0.19843745)
\usefont{T1}{ptm}{m}{n}
\rput(0.82281256,1.6065625){$\gr e$}
\usefont{T1}{ptm}{m}{n}
\rput(5.4628124,1.5865625){$\gr e$}
\usefont{T1}{ptm}{m}{n}
\rput(0.5314062,0.04499995){$1$}
\usefont{T1}{ptm}{m}{n}
\rput(1.0314063,0.04499995){$6$}
\usefont{T1}{ptm}{m}{n}
\rput(5.1114063,0.12499995){$1$}
\usefont{T1}{ptm}{m}{n}
\rput(5.6714063,0.12499995){$6$}
\psline[linewidth=0.04cm](0.8,-1.7)(0.8,-0.20000005)
\psline[linewidth=0.04cm](5.42,-1.72)(5.4,-0.22000004)
\psdots[dotsize=0.12](5.4,-1.7)
\psdots[dotsize=0.12](0.78,-1.7)
\usefont{T1}{ptm}{m}{n}
\rput(1.0014062,-0.470500004){$5$}
\usefont{T1}{ptm}{m}{n}
\rput(5.6514063,-0.47500006){$10$}
\usefont{T1}{ptm}{m}{n}
\rput(1.0014063,-1.5150001){$2$}
\usefont{T1}{ptm}{m}{n}
\rput(5.631406,-1.495){$2$}
\end{pspicture} 
}
\end{center}
\caption{The two graphs of groups are related by an induction.}
\label{exemplepa}
\end{figure}

Given a usual abelian JSJ tree of a \GBSn{n} group, we obtain a universally compatible tree by collapsing four types of edges. The definitions of these edges is technical and described in Section \ref{prelim}. .

\begin{proposition}
Let $G$ be a \GBSn{n} group.
 Let $T$ be a reduced abelian JSJ tree of $G$. Let $\mathcal E$ be the set of edges containing the vanishing edges of $T$, the potentially strictly ascending edges of $T$, the non-ascending slippery edges of $T$, the \tor $2$-slippery edges of $T$. 

\begin{enumerate}[noitemsep, topsep=0cm]
\item The tree $T'$obtained from $T$ by collapsing the edges of $\mathcal E$ is compatible with every $G$-tree over $\Z^n$ groups.
\item Let $T''$ be a collapse of $T$ in which an edge of $\mathcal E$ is not collapsed. Then $T''$ is not compatible with every $G$-tree over $\Z^n$ groups.
\end{enumerate}
\end{proposition}

However the tree $T'$ is not always the compatibility JSJ tree: some vertices could have to be expanded. Roughly speaking, some vertices could act as dead end: no edge "arriving" at this vertex by a slide may continue to slide further. These vertices must be expanded in the compatibility JSJ tree in a precise way called a \textit{blow up} (see Section \ref{prelim}). The vertex $\gr v$ of Figure \ref{exempleblowup} is an example. The top edge may slide along $\gr e$ or along $\gr f$, but after performing one of these slides, no new slide is allowed (except the converse one). The compatibility JSJ is obtained by collapsing $\gr e$ and $\gr f$ (which are slippery) in the graph of groups on the right. The complete description of \textit{dead ends} is given in Section \ref{prelim}. 

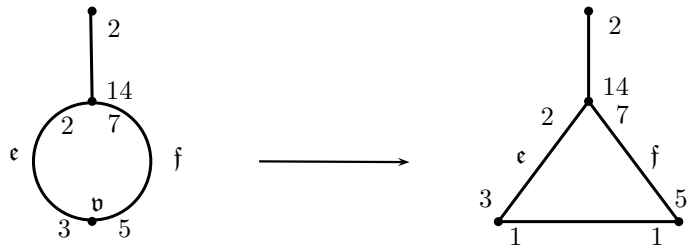
\begin{figure}[!ht]
\begin{center}
% Generated with LaTeXDraw 2.0.5
% Sun Jun 19 18:56:07 CEST 2011
% \usepackage[usenames,dvipsnames]{pstricks}
% \usepackage{epsfig}
% \usepackage{pst-grad} % For gradients
% \usepackage{pst-plot} % For axes
\scalebox{1} % Change this value to rescale the drawing.
{
\begin{pspicture}(0,-1.63375)(9.642813,1.6078612)
\pscircle[linewidth=0.04,dimen=outer](1.4995313,-0.4646874){0.8}
\psdots[dotsize=0.12](1.4995313,-1.2646874)
\psline[linewidth=0.03cm,arrowsize=0.05291667cm 2.0,arrowlength=1.4,arrowinset=0.4]{->}(3.7195313,-0.4646874)(5.719531,-0.4743749)
\psdots[dotsize=0.12](1.4995313,0.3353126)
\psline[linewidth=0.04cm](1.4795313,1.53375)(1.4995313,0.35375)
\psline[linewidth=0.04cm](8.099531,1.53375)(8.099531,0.33375)
\psline[linewidth=0.04cm](8.099531,0.33375)(6.8995314,-1.26625)
\psline[linewidth=0.04cm](8.099531,0.33375)(9.299531,-1.26625)
\psline[linewidth=0.04cm](6.8995314,-1.26625)(9.299531,-1.26625)
\psdots[dotsize=0.12](8.099531,0.33375)
\psdots[dotsize=0.12](6.8995314,-1.26625)
\psdots[dotsize=0.12](9.299531,-1.26625)
\psdots[dotsize=0.12](8.099531,1.53375)
\psdots[dotsize=0.12](1.4995313,1.53375)
\usefont{T1}{ptm}{m}{n}
\rput(1.1323438,-1.36125){$3$}
\usefont{T1}{ptm}{m}{n}
\rput(6.7323437,-0.96125){$3$}
\usefont{T1}{ptm}{m}{n}
\rput(1.9323437,-1.36125){$5$}
\usefont{T1}{ptm}{m}{n}
\rput(9.332344,-0.96125){$5$}
\usefont{T1}{ptm}{m}{n}
\rput(7.132344,-1.46125){$1$}
\usefont{T1}{ptm}{m}{n}
\rput(9.012343,-1.46125){$1$}
\usefont{T1}{ptm}{m}{n}
\rput(1.1723437,0.01875){$2$}
\usefont{T1}{ptm}{m}{n}
\rput(7.552344,0.13875){$2$}
\usefont{T1}{ptm}{m}{n}
\rput(1.7923437,0.03875){$7$}
\usefont{T1}{ptm}{m}{n}
\rput(8.552343,0.15875){$7$}
\usefont{T1}{ptm}{m}{n}
\rput(1.8623438,0.47875){$14$}
\usefont{T1}{ptm}{m}{n}
\rput(8.462344,0.53875){$14$}
\usefont{T1}{ptm}{m}{n}
\rput(1.7923437,1.29875){$2$}
\usefont{T1}{ptm}{m}{n}
\rput(8.452344,1.33875){$2$}
\usefont{T1}{ptm}{m}{n}
\rput(0.46234375,-0.36125){$\gr e$}
\usefont{T1}{ptm}{m}{n}
\rput(2.6396875,-0.42125){$\gr f$}
\usefont{T1}{ptm}{m}{n}
\rput(7.202344,-0.40125){$\gr e$}
\usefont{T1}{ptm}{m}{n}
\rput(8.982344,-0.40125){$\gr f$}
\usefont{T1}{ptm}{m}{n}
\rput(1.5714062,-1.027139){$\gr v$}
\usefont{T1}{ptm}{m}{n}
%\rput(1.2514062,0.5328611){$\gr v$}
\usefont{T1}{ptm}{m}{n}
%\rput(7.771406,0.5928611){$\gr v$}
\end{pspicture} 
}
\end{center}
\caption{The edges $\gr e$ and $\gr f$ are slippery and $\gr v$ is a dead end.}
\label{exempleblowup}
\end{figure}

\begin{theorem}\label{compintro}
Let $G$ be a \GBSn{n} group.
 Let $T$ be a reduced abelian JSJ tree of $G$.
  Call $T_{comp}$ the $G$-tree obtained from $T$ by blowing up up the dead ends and collapsing the  vanishing edges of $T$,  the \psa edges of $T$, the non-ascending slippery edges of $T$, the \tor $2$-slippery edges of $T$.

Then $T_{comp}$ is a compatibility JSJ tree over $\Z^n$ groups.
\end{theorem}

This construction is algorithmic whenever no vertex group is conjugated to one of its proper subgroup. In the general case, the decidability of the construction is unknown.

The \GBSn{n} groups also admit splittings over $\Z^{n+1}$ groups. If for the usual JSJ tree, this does not have any incidence - the abelian JSJ tree and the JSJ tree over $\Z^n$ are isomorphic -, it has one for the compatibility JSJ tree. For example, applying Theorem \ref{compintro}, we may easily see that the compatibility JSJ tree over $\Z^n$ groups of $BS(2,2)=\langle a, t | ta^2t^{-1} =a ^2 \rangle$ is the one associated to the HNN extension, but this tree is not compatible with the tree associated to the amalgamated product
$$BS(2,2) =\langle a,s | sa^2s^{-1}=a^2\rangle \ast_{a^2=b,s=t^2} \langle b, t | bt=tb\rangle.$$ The incompatibility comes from the fact that  the element $t$ is hyperbolic in the first splitting but stabilizes an edge in the second. Such an element is said to be \textit{\pbe.}

In the case of \vGBS groups, to obtain universally compatible tree over abelian groups, it suffices to collapse in any JSJ tree the four types of edges described previously, and every edge in the axis of a \pbe element. Again, the obtained tree is not always the abelian compatibility JSJ tree, even if we blow up the dead ends. Some other edges have to be expanded:

In some very specific cases (see description in Section \ref{sectioninert}), the axes of \pbe elements must be separated from the rest of the tree, and then the axes must be collapsed (see Figure \ref{exempleinert}).
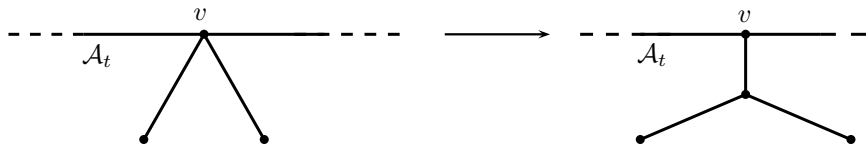
\begin{figure}[!ht]
\begin{center}
% Generated with LaTeXDraw 2.0.5
% Mon Jun 20 11:15:29 CEST 2011
% \usepackage[usenames,dvipsnames]{pstricks}
% \usepackage{epsfig}
% \usepackage{pst-grad} % For gradients
% \usepackage{pst-plot} % For axes
\scalebox{1} % Change this value to rescale the drawing.
{
\begin{pspicture}(0,-0.95921874)(11.42,0.9792187)
\psline[linewidth=0.04cm](1.0,0.5207813)(4.2,0.5207813)
\psdots[dotsize=0.12](2.6,0.5207813)
\psline[linewidth=0.04cm](2.6,0.5207813)(1.8,-0.87921876)
\psline[linewidth=0.04cm](2.6,0.5207813)(3.4,-0.87921876)
\psline[linewidth=0.03cm,arrowsize=0.05291667cm 2.0,arrowlength=1.4,arrowinset=0.4]{->}(5.8,0.5207813)(7.2,0.5207813)
\psline[linewidth=0.04cm](8.4,0.5207813)(10.8,0.5207813)
\psdots[dotsize=0.12](9.8,0.5207813)
\psline[linewidth=0.04cm](9.8,-0.27921876)(8.4,-0.87921876)
\psline[linewidth=0.04cm](9.8,-0.27921876)(11.2,-0.87921876)
\psline[linewidth=0.04cm,linestyle=dashed,dash=0.16cm 0.16cm](3.8,0.5207813)(5.2,0.5207813)
\psline[linewidth=0.04cm,linestyle=dashed,dash=0.16cm 0.16cm](1.0,0.5207813)(0.0,0.5207813)
\psline[linewidth=0.04cm,linestyle=dashed,dash=0.16cm 0.16cm](8.8,0.5207813)(7.6,0.5207813)
\psline[linewidth=0.04cm,linestyle=dashed,dash=0.16cm 0.16cm](10.8,0.5207813)(11.4,0.5207813)
\psline[linewidth=0.04cm](9.8,0.5207813)(9.8,-0.27921876)
\psdots[dotsize=0.12](9.8,-0.27921876)
\psdots[dotsize=0.12](1.8,-0.87921876)
\psdots[dotsize=0.12](3.4,-0.87921876)
\usefont{T1}{ptm}{m}{n}
\rput(2.5514061,0.78578126){$v$}
\usefont{T1}{ptm}{m}{n}
\rput(9.791407,0.7657812){$v$}
\usefont{T1}{ptm}{m}{n}
\rput(1.1814063,0.22578125){$\mathcal A_t$}
\usefont{T1}{ptm}{m}{n}
\rput(8.561406,0.24578124){$\mathcal A_t$}
\psdots[dotsize=0.12](8.4,-0.87921876)
\psdots[dotsize=0.12](11.2,-0.87921876)
\end{pspicture} 
}
\end{center}
\caption{The axis $\mathcal A_t$ of a \pbe element $t$ must be "taken away" from the rest of the tree.}
\label{exempleinert}
\end{figure}

These separations are called \textit{expansions (of inert edges)}. We obtain the following theorem.

\begin{theorem}
Let $G$ be a \vGBS group.
Let $T$ be a reduced abelian JSJ tree of $G$.
Call $\mathcal E$ the set of non-ascending slippery, potentially strictly ascending  and \tor $2$-slippery edges.

We define $T_{ab}$ as the $G$-tree obtained by expanding inert edges, blowing up dead ends, then collapsing every edge of $\mathcal E$ and the axes of \pbe elements.

Then $T_{ab}$ is an abelian compatibility JSJ tree of $G$.
\end{theorem}

This gives a description of the abelian compatibility JSJ tree. However this construction is not and cannot be algorithmic.

\begin{theorem}
There is no algorithm that constructs the abelian compatibility JSJ tree of \vGBS groups.
\end{theorem}

This inconstructibility follows from the fact that in \vGBS trees, we may not dectect whether an edge is slippery or not. We describe this problem in the last section.

\section{Preliminaries}\label{prelim}

Let $G$ be a  finitely generated group. A \textit{$G$-tree} is a simplicial tree $T$ equipped with an action of $G$ that we suppose without inversion (the stabilizer of an edge is included in the stabilizer of its endpoints) and minimal.  Given an (oriented) edge $e$, the opposite edge is denoted by $\bar e$. We denote by $G_e$ the stabilizer of $e$. We have $G_e=G_{\bar e}$. The stabilizer of a vertex $v$ is denoted by$G_v$. The orbit of a vertex $v$ and an edge $e$ is denoted by Gothic font $\gr{v}$ and $\gr{e}$. Similarly the opposite orbit  of  ${\gr{e}}$ is denoted by $\bar{\gr{e}}$. The \textit{valency} of a vertex $v$ (or an orbit of vertices $\gr v$) is the number of orbits of edges with initial vertex in $\gr v$.

An edge is \textit{reduced} if its endpoints are in the same orbit under the action of $G$ or its stabilizer is strictly included in both stabilizers of its endpoints. An orbit of edges is \textit{reduced} if one (or equivalently every) representative is reduced.
A tree is reduced if all its edges are reduced. 
The \textit{deformation space} $\mathcal D_T$ of $T$ is the set of $G$-trees having the same set of elliptic subgroups as $T$.
The  \textit{reduced deformation space} of $T$ is the subset of reduced trees of $\mathcal D_T$. 

To \textit{collapse} an orbit of edges $\gr e$ consists in collapsing every connected component of edges in the orbit $\gr e$ to a point. This operation produces a new $G$-tree. The converse of collapsing is called \textit{expanding}.
Given a $G$-tree $T$ and an orbit of edges $\gr e$, the tree obtained by collapsing $\gr e$ is in the deformation space of $T$ if and only if $\gr e$ is not reduced.
We may thus construct a reduced $G$-tree in the deformation space of $T$ by collapsing one by one orbits of non-reduced edges until the tree is reduced (we assume here that number of orbits of edges is finite). Note that if we collapse at the same time all non reduced edges, the obtained tree may be in a different deformation space.

\paragraph{Whitehead moves for $G$-trees}
In \cite{ClayFor09}, Clay and Forester describe three deformation moves on $G$-trees such that any two reduced $G$-trees of a given deformation space are related by a finite sequence of these three moves, and every intermediate tree in this sequence is reduced.

The first deformation move is the \emph {slide} of an edge $e$ along an edge $f$: if two edges $e$ and $f$ are such that 

\begin{itemize}[noitemsep]
\item $e$ is not in the orbit of $f$ or $\bar f$,
\item the terminal vertex of $e$ is equal to the initial vertex $v$ of $f$,
\item the stabilizer of $e$ is included in the stabilizer of $f$,
\end{itemize}

then $e$ may slide along $f$. Call $w$ the terminal vertex of $f$. The new $G$-tree is obtained by changing the terminal vertex of every vertex $g\cdot e$ in the orbit of $e$  from $g\cdot v$ to $g\cdot w$ (see Figure \ref{figureglissement} for the changes on the associated graph of groups). The case were $\gr v=\gr w$ is not excluded. This move does not change the stabilizers of vertices and edges. This move is fully determined by the data of $e$ and $f$. Note that the converse move is also a slide.

The slide of $e$ along $f$ will be denoted by $e/f$.

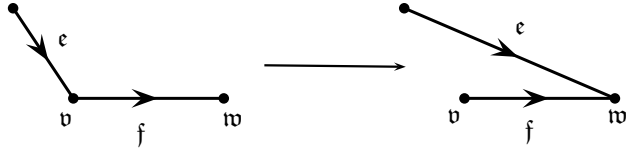
\begin{figure}[!ht]
\begin{center}
\scalebox{1} % Change this value to rescale the drawing.
{
\begin{pspicture}(0,-0.98375)(8.19,0.97375)
\definecolor{color54b}{rgb}{0.00392156862745098,0.00392156862745098,0.00392156862745098}
\psline[ArrowInside=->, ArrowInsidePos=0.5,arrowsize=0.25291667cm,linewidth=0.04cm](0.08,0.86375)(0.88,-0.33625)
\psline[ArrowInside=->, ArrowInsidePos=0.5,arrowsize=0.25291667cm,linewidth=0.04cm](0.88,-0.33625)(2.88,-0.33625)
\psline[ArrowInside=->, ArrowInsidePos=0.5,arrowsize=0.25291667cm,linewidth=0.04cm](5.28,0.86375)(8.08,-0.33625)
\psline[ArrowInside=->, ArrowInsidePos=0.5,arrowsize=0.25291667cm,linewidth=0.04cm](6.08,-0.33625)(8.08,-0.33625)
\psdots[dotsize=0.14](8.08,-0.33625)
\psdots[dotsize=0.14](5.28,0.86375)
\psdots[dotsize=0.14](6.08,-0.33625)
\psdots[dotsize=0.14](2.88,-0.33625)
\psdots[dotsize=0.14](0.88,-0.33625)
\psdots[dotsize=0.14](0.08,0.86375)
\usefont{T1}{ptm}{m}{n}
\rput(0.74140626,0.44875){$\gr e$}
\usefont{T1}{ptm}{m}{n}
\rput(6.8214064,0.60875){$\gr e$}
\usefont{T1}{ptm}{m}{n}
\rput(1.7814063,-0.81125){$\gr f$}
\usefont{T1}{ptm}{m}{n}
\rput(0.7814063,-0.61125){$\gr v$}
\usefont{T1}{ptm}{m}{n}
\rput(3.0014063,-0.61125){$\gr w$}
\usefont{T1}{ptm}{m}{n}
\rput(6.9414062,-0.77125){$\gr f$}
\usefont{T1}{ptm}{m}{n}
\rput(5.9414062,-0.57125){$\gr v$}
\usefont{T1}{ptm}{m}{n}
\rput(8.1214062,-0.57125){$\gr w$}
\psline[linewidth=0.03cm,fillcolor=color54b,arrowsize=0.05291667cm 2.0,arrowlength=1.4,arrowinset=0.4]{->}(3.42,0.10375)(5.3,0.08375)
\end{pspicture} 
}
\end{center}
\caption{Slide of $\gr e$ along $\gr f$}
\label{figureglissement}
\end{figure}

\paragraph{}
The second deformation move is the \emph{induction}. Let $e$ be an edge with initial and terminal vertices $v$ and $w$ and $A$ a subgroup of $G_w$ such that 
\begin{itemize}[noitemsep]
\item $v$ and $w$ are in the same orbit,
\item the stabilizers $G_v$ and $G_e$ are equal,
\item the group $A$ contains  $G_e$.
\end{itemize}

Then we may perform an induction on $e$ with group $A$.  We first add an edge $f$ with terminal vertex $w$,  a new initial vertex $v'$ which is also the new terminal vertex of $e$ and with $G_f=G_{v'}=A$. Every edge of another orbit with initial vertex $w$ keeps $w$ as initial vertex. We then collapse $e$ (see Figure \ref{induction}). This description is given around the edge $e$ but must be made equivariantly.

This move does not change the underlying graph, but the edge $e$ and the vertex $v$ have been replaced by an edge $f$ and a vertex $v'$ with a distinct stabilizer $A$. It is fully determined by the data of $e$ and the group $A$. 

The induction on the edge $e$ with group $A$ will be denoted $i_A(e)$ or $e/e$ if we do not precise the group.

\begin{figure}[!ht]
\begin{center}
% Generated with LaTeXDraw 2.0.5
% Wed Jan 26 17:27:25 CET 2011
% \usepackage[usenames,dvipsnames]{pstricks}
% \usepackage{epsfig}
% \usepackage{pst-grad} % For gradients
% \usepackage{pst-plot} % For axes
\scalebox{1} % Change this value to rescale the drawing.
{
\begin{pspicture}(0,-1.308125)(11.21875,1.308125)
\pscircle[linewidth=0.04,dimen=outer](0.8,0.0096875){0.8}
\psdots[dotsize=0.12](0.8,-0.7903125)
\psline[linewidth=0.03cm,arrowsize=0.05291667cm 2.0,arrowlength=1.4,arrowinset=0.4]{->}(2.2,0.0096875)(4.2,0)
\pscircle[linewidth=0.04,dimen=outer](5.6,0.0096875){0.8}
\psdots[dotsize=0.12](5.6,-0.7903125)
\psline[linewidth=0.03cm,arrowsize=0.05291667cm 2.0,arrowlength=1.4,arrowinset=0.4]{->}(7.0,0.0096875)(9.0,0)
\pscircle[linewidth=0.04,dimen=outer](10.4,0.0096875){0.8}
\psdots[dotsize=0.12](10.4,-0.7903125)
\psdots[dotsize=0.12](5.6,0.8096875)
\usefont{T1}{ptm}{m}{n}
\rput(1.0114063,-1.0853125){$\gr w=\gr v$}
\rput(1.2814063,0.9146875){$\gr e$}
\rput(6.481406,0.3146875){$\gr e$}
\rput(5.7614064,-1.0853125){$\gr w=\gr v$}
\rput(4.6614063,0.3146875){$\gr f$}
\rput(5.581406,1.1146874){$\gr  v'$}
\rput(10.861406,0.9146875){$\gr f$}
\rput(10.531406,-1.0853125){$\gr v'$}
\end{pspicture} 
}
\end{center}
\caption{Induction on $\gr e$}
\label{induction}
\end{figure}

If we have an inclusion of groups $G_e\subset A\subset B\subset G_w$, we may first perform an induction on $A$ and then an induction on $B$. The composition of these two inductions is equal to the induction on $B$: $i_A(e)\cdot i_B(f)=i_B(e)$ (where $f$ is the edge appearing in the induction on $e$). 

The induction with $A=G_e$ has no effect on the $G$-tree. We call this move a \textit{trivial induction}. At the opposite we may take $A=G_w$. In this case performing the induction is the same as sliding along $e$every edge $f$ (not in the orbit of $e$ or $\bar e$) with terminal vertex $v$. 

The converse move of an induction with any group is just an induction with group $G_w$ and slides along $\bar e$. Indeed from the first remark, we have $i_A(e)\cdot i_{G_w}(f)=i_{G_w}(e)$. And by the second remark, the move $i_{G_w}(e)$ is the same as a finite sequence of slides.

By extension, performing the converse of an induction with group $A$ consists in performing an induction with group $A$ and then sliding every edge around $w$ along $\bar e$. This move is possible whenever, for every edge $g$ with initial vertex $w$ not in the orbit of $\gr e$, the group $A$ contains $G_g$.

\paragraph{}The third move is the \emph {$\mathcal A^{\pm1}$-move}.
We first describe the $\mathcal A^{-1}$-move. Suppose that $e$ is an edge with initial and terminal vertices $v$ and $w$, and that $f$ is an edge with terminal vertex $w$ and initial vertex not in $\gr w$, such that
\begin{itemize}[noitemsep]
\item $v$ and $w$ are in the same orbit,
\item the stabilizers $G_v$ and $G_e$ are equal,
\item for every edge $g$ with initial vertex $w$ not in the orbit of $\gr e$, the stabilizer of $f$ contains $G_g$.
%\item the vertex $v$ is of valence $3$ (all adjacent edges belong to the orbit of $e$, $\bar e$ or $\bar f$).
\end{itemize}
We may then perform an {$\mathcal A^{-1}$-move} on $e$ by first performing the converse of an induction on $e$ with group $G_f$ %, then performing a slide of $f$ along $\bar e$, 
and then collapsing $f$ which is now non reduced (see Figure \ref{figureamouvement} for the changes on the underlying graph of groups). We will say that  $v$ (and $\gr{v}$) is a \textit{vanishing vertex} and $f$ (and $\gr{f}$) a \textit{vanishing edge}. %The converse move is called an $\mathcal A$-move.

For some technical reasons, the $\mathcal A^{-1}$-move is not exactly the same as the one described in \cite{ClayFor09}: we allow here the vertex $w$ to be of valence more than $3$.

However the definition of an $\mathcal A$-move remains the same:
Let $e$ be an edge with initial and terminal vertices $v$ and $w$ and $a\in G$ such that

\begin{itemize}[noitemsep]
\item $w=a\cdot v$ (so $v$ and $w$ belong to the same orbit of edges),
\item we  have $G_e \subsetneq aG_ea^{-1}$.
%\item the vertex $v$ is of valence $3$ (all adjacent edges belong to the orbit of $e$, $\bar e$ or $\bar f$).
\end{itemize}

Call $\mathcal F$ the set of the edges with initial vertex $w$ not in the orbit of $e$ or $\bar e$.
Performing a $\mathcal A$-move consists in performing the two following deformations. First expand $w$ in an edge $f$ with initial and terminal vertex $w'$ and $w''$, with $G_f=G_{w''}=G_e$ and $G_{w'}=G_w$ and such that the initial vertex of $e$ and the new terminal vertex of $a\cdot e$ are $w''$ and the initial vertices of the edges of $\mathcal F$ are $w'$. Then perform an induction on $e$ with group $aG_ea^{-1}$.

The $\mathcal A^{-1}$-move is determined by the data of the orbits of $e$ and $f$, however  the $\mathcal A$-move is determined by the data of two edges $e$ and $e'$ in the same orbit such that the terminal vertex of $e$ is the initial edge of $e'$, and $G_e\subsetneq G_{e'}$.

A $\mathcal A^{\pm 1}$-move on $e$ is denoted $\mathcal A^{\pm 1}(e)$ or $e/e$.

% Generated with LaTeXDraw 2.0.8
% Wed Sep 08 11:49:50 CEST 2010
% \usepackage[usenames,dvipsnames]{pstricks}
% \usepackage{epsfig}
% \usepackage{pst-grad} % For gradients
% \usepackage{pst-plot} % For axes
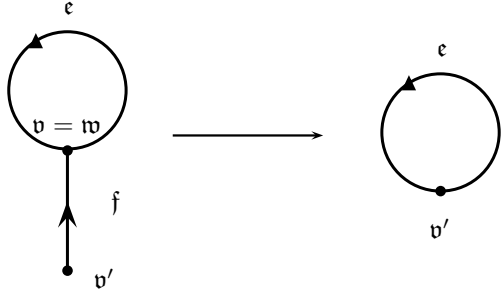
\begin{figure}[!ht]
\begin{center}
\scalebox{1} % Change this value to rescale the drawing.
{
\begin{pspicture}(0,-1.9829688)(6.56,1.9829688)
\definecolor{color477b}{rgb}{0.00392156862745098,0.00392156862745098,0.00392156862745098}
\psdots[dotsize=0.14](0.8,-0.11546875)
\psdots[dotsize=0.14](0.8,-1.7154688)
\psline[ArrowInside=->, ArrowInsidePos=0.5,arrowsize=0.25291667cm,linewidth=0.04cm](0.8,-1.7154688)(0.8,-0.11546875)
\pscircle[linewidth=0.04,dimen=outer](0.8,0.6845313){0.8}
\psline[linewidth=0.03cm,arrowsize=0.05291667cm 2.0,arrowlength=1.4,arrowinset=0.4]{->}(2.2,0.08453125)(4.2,0.08453125)
\usefont{T1}{ptm}{m}{n}
\rput(1.4414062,-0.81046873){$\gr f$}
\usefont{T1}{ptm}{m}{n}
\rput(0.82140625,1.7895312){$\gr e$}
\pstriangle[linewidth=0.04,dimen=outer,fillstyle=solid,fillcolor=color477b](0.38,1.2445313)(0.2,0.2)
\usefont{T1}{ptm}{m}{n}
\rput(0.7914063,0.20953125){$\gr v=\gr w$}
\pscircle[linewidth=0.04,dimen=outer](5.76,0.12453125){0.8}
\usefont{T1}{ptm}{m}{n}
\rput(5.7814064,1.2295313){$\gr e$}
\pstriangle[linewidth=0.04,dimen=outer,fillstyle=solid,fillcolor=color477b](5.34,0.6845313)(0.2,0.2)
\usefont{T1}{ptm}{m}{n}
\rput(5.751406,-1.1504687){$\gr v'$}
\psdots[dotsize=0.14](5.76,-0.65546876)
\usefont{T1}{ptm}{m}{n}
\rput(1.2914063,-1.8104688){$\gr v'$}
\end{pspicture} 
}
\end{center}
\caption{$\mathcal A^{-1}$-move}
\label{figureamouvement}
\end{figure}

\paragraph{}
A move is \textit{admissible} if  we may perform it, if the obtained tree it is reduced, and if it is not a trivial induction. A sequence of moves is \textit{admissible} if after performing the first $n$ moves of the sequence the $n+1$th is admissible. 

If $S$ is an admissible sequence of moves on the $G$-tree $T$, the tree obtained by applying $S$ is denoted by $S\cdot T$.

The number of orbits of vertices and edges in the $G$-trees of a reduced deformation space is not fixed. However, the only way to make a vertex or an edge disappear is to perform an $\mathcal A^{-1}$-move. Thus if two $G$-trees $T$ and $T'$ are related by a single move, we may identify the orbits of vertices and edges of $T$ with the ones of $T'$, except for the one that vanish. Moreover if the move is not an induction or an $\mathcal A^{\pm 1}$-move, the edge and vertex stabilizers does not change, we may then identify not only the orbits of edges but each edge.
To be more precise, let $T$ be a $G$-tree, $e$ an edge of $T$ and $\textbf{m}$ an admissible move on $T$ such that \begin{itemize}[noitemsep]
\item the move $\textbf{m}$ is not an induction on the orbit $\gr e$,
\item the move $\textbf{m}$ is not an $\mathcal A^{\pm 1}$-move on the orbit $\gr e$,
\item if $\textbf{m}$ is a $\mathcal A^{-1}$-move, its vanishing orbit of edges is not $\gr e$.
\end{itemize}
We may then identify $e$ with an edge of $\textbf{m}\cdot T$ and this identification is equivariant.

By extension, an orbit of vertices or edges is \textit{vanishing} if we may perform an admissible sequence of moves ending by a $\mathcal A^{-1}$-move in which it vanishes. A sequence of moves \textit{preserves} an orbit $\gr e$ if no move of the sequence is an $\mathcal A^{-1}$-move in which $\gr e$ vanishes.

As two reduced $G$-trees in the same reduced deformation space are related by a finite admissible sequence of deformation moves, given two reduced $G$-trees $T$ and $T'$ in the same deformation space, we may identify any non-vanishing orbit of vertices or edges of $T$ to one of $T'$.

\paragraph{Edges properties}
Let $T$ be a reduced $G$-tree. An orbit of edges $\gr e$ of $T$ is \textit{slippery} if there exists an admissible sequence of moves on $T$, preserving $\gr e$ and ending by a slide along $\gr e$.  The orbit $\gr e$ is \textit{$2$-slippery} if there exist an admissible sequence $S$ of moves on $T$, preserving $\gr e$ and two distinct orbits of edges $\gr f$ and $\gr f'$ of $S\cdot T$ such that the slides $\gr f/\gr e$ and $\gr f'/\gr e$ are simultaneously admissible in $S\cdot T$. Note that $\gr f'=\bar {\gr f}$ is allowed.

An edge $e$ of $T$ is \textit{ascending} if its endpoints are in the same orbit, and if its stabilizer is equal to the stabilizer of its initial vertex. It is \textit{strictly ascending} if moreover its stabilizer is strictly contained in the stabilizer of its terminal vertex. An edge is \textit{\tor} if its stabilizer is equal to the stabilizer of both its endpoints. The opposite edge of a (strictly) ascending edge is \textit{(strictly) descending}. 

An edge $e$ is \textit{\pa} in $T$  if its initial and terminal vertices $v$ and $v'$ are in the same orbit and if there exists $t\in G$ such that $t\cdot v'=v$ and $tG_et^{-1}\subsetneq G_e$. The opposite edge of a \pa edge is said to be \textit{\pde}. The edge $e$ is \textit{potentially strictly ascending} (or \textit{\psa}) if $\gr e$ is \pa after a finite admissible sequence of moves preserving $\gr e$.
Note that a strictly ascending edge is also \pa. If an edge $e$ is such that $e$ or $\bar e$ is \psa then $e$ is said to be \psad.

An orbit of edges is slippery, $2$-slippery, \pa, \pde, (potentially) strictly ascending (or descending) or \tor if one of its representative is.

The inductions and $\mathcal A^{-1}$-moves may only be performed on strictly ascending (or descending) edges. An $\mathcal A$-move may only be performed on the \pa edges which are not ascending. Note that an $\mathcal A$-move changes a \pa edge into a strictly ascending edge.

Figure \ref{exempledef} represent a Generalized Baumslag-Solitar group and is constructed in the following way: all vertices and edges carry the infinite cyclic group $\Z$, the number $n$ at the end of each edge indicate that the injection of the edge group into the group of its endpoint is $k \mapsto  nk$.

 In such a representation, \pa edges and admissible slides may be seen as divisibility relation. For example in Figure \ref{exempledef}, the edge $\gr e$ is \pde, the edges $\gr f$ and $\gr g$ are slippery since $\gr h$ may first slide along $\gr g$ then along $\gr f$, and $\gr h$ is \psa since it is \pa after sliding along $\gr g$ and $\gr f$.

\begin{figure}[!ht]
\begin{center}
% Generated with LaTeXDraw 2.0.5
% Fri Jan 21 17:08:00 CET 2011
% \usepackage[usenames,dvipsnames]{pstricks}
% \usepackage{epsfig}
% \usepackage{pst-grad} % For gradients
% \usepackage{pst-plot} % For axes
\scalebox{1} % Change this value to rescale the drawing.
{
\begin{pspicture}(0,-1.3829688)(6.8028126,1.3829688)
\psarc[linewidth=0.04,arrowsize=0.15291667cm 2.0,arrowlength=1.4,arrowinset=0.4]{->}(4.7809377,0.08453125){1.0}{0.}{0.0}
\pspolygon[ArrowInside=->, ArrowInsidePos=0.5,arrowsize=0.25291667cm,linewidth=0.04](3.7809374,0.08453125)(1.5809375,1.0845313)(0.5809375,-1.1154687)
\psdots[dotsize=0.12](0.5809375,-1.1154687)
\psdots[dotsize=0.12](1.5809375,1.0845313)
\psdots[dotsize=0.12](3.7809374,0.08453125)
\usefont{T1}{ptm}{m}{n}
\rput(5.562344,0.11046875){$\gr e$}
\usefont{T1}{ptm}{m}{n}
\rput(2.6423438,0.88953123){$\gr f$}
\usefont{T1}{ptm}{m}{n}
\rput(2.3623437,-0.71046873){$\gr g$}
\usefont{T1}{ptm}{m}{n}
\rput(0.87234374,0.13953125){$\gr h$}
\usefont{T1}{ptm}{m}{n}
\rput(1.2323437,0.98953123){$2$}
\usefont{T1}{ptm}{m}{n}
\rput(0.42234374,-0.81046873){$12$}
\usefont{T1}{ptm}{m}{n}
\rput(1.0323437,-1.2104688){$3$}
\usefont{T1}{ptm}{m}{n}
\rput(3.5323437,-0.21046876){$5$}
\usefont{T1}{ptm}{m}{n}
\rput(4.0323437,-0.21046875){$3$}
\usefont{T1}{ptm}{m}{n}
\rput(4.122344,0.38953125){$21$}
\usefont{T1}{ptm}{m}{n}
\rput(3.5323437,0.38953125){$2$}
\usefont{T1}{ptm}{m}{n}
\rput(1.8323439,1.1895312){$3$}
\end{pspicture} 
}
\end{center}

\caption{}
\label{exempledef}
\end{figure}

\paragraph{Expansion}
Let $T$ be a $G$-tree, and $v$ a vertex of $T$. Take $\mathcal F$ a set of edges  with pairwise distinct orbits and initial vertex $v$, and $H$ a subgroup of $G_v$ containing $G_f$ for all $f \in \mathcal F$. Then the \textit{expansion} of $v$ of group $H$ and set $\mathcal F$ is the tree $\tilde T$ obtained from $T$ by expanding $v$ in an edge $e$ of initial and terminal vertex $\tilde v$ and $\tilde v'$, such that $G_e=G_{\tilde v}=H$, $G_{\tilde v'}=G_v$, the edges of $\mathcal F$ have initial vertex $\tilde v$, and every edge of initial vertex $v$ in $T$ and not in the orbit of any edge of $\mathcal F$ has initial vertex $\tilde v'$ (see Figure \ref{expansion} for the changes in the graph of groups). This construction produces a non-reduced tree.

\begin{figure}[!ht]
\begin{center}
% Generated with LaTeXDraw 2.0.5
% Fri Apr 29 16:06:17 CEST 2011
% \usepackage[usenames,dvipsnames]{pstricks}
% \usepackage{epsfig}
% \usepackage{pst-grad} % For gradients
% \usepackage{pst-plot} % For axes
\scalebox{1} % Change this value to rescale the drawing.
{
\begin{pspicture}(0,-1.0789063)(10.82,1.0589062)
\psline[ArrowInside=->, ArrowInsidePos=0.5,arrowsize=0.25291667cm,linewidth=0.04cm](1.2,-0.16109376)(0.0,0.8389062)
\psline[ArrowInside=->, ArrowInsidePos=0.5,arrowsize=0.25291667cm,linewidth=0.04cm](1.2,-0.16109376)(0.0,-0.9610937)
\psdots[dotsize=0.12](1.2,-0.16109376)
\psline[ArrowInside=->, ArrowInsidePos=0.5,arrowsize=0.25291667cm,linewidth=0.04cm](1.2,-0.16109376)(2.6,-0.76109374)
\psline[ArrowInside=->, ArrowInsidePos=0.5,arrowsize=0.25291667cm,linewidth=0.04cm](1.2,-0.16109376)(2.6,0.23890625)
\psline[ArrowInside=->, ArrowInsidePos=0.5,arrowsize=0.25291667cm,linewidth=0.04cm](1.2,-0.16109376)(2.6,0.8389062)
\psline[linewidth=0.04cm,arrowsize=0.05291667cm 2.0,arrowlength=1.4,arrowinset=0.4]{->}(4.0,0.03890625)(6.2,0.03890625)
\psline[ArrowInside=->, ArrowInsidePos=0.5,arrowsize=0.25291667cm,linewidth=0.04cm](9.4,0.03890625)(10.8,-0.56109375)
\psline[ArrowInside=->, ArrowInsidePos=0.5,arrowsize=0.25291667cm,linewidth=0.04cm](9.4,0.03890625)(10.8,0.43890625)
\psline[ArrowInside=->, ArrowInsidePos=0.5,arrowsize=0.25291667cm,linewidth=0.04cm](9.4,0.03890625)(10.8,1.0389062)
\psdots[dotsize=0.12](7.8,0.03890625)
\psdots[dotsize=0.12](9.4,0.03890625)
\psline[ArrowInside=->, ArrowInsidePos=0.5,arrowsize=0.25291667cm,linewidth=0.04cm](7.8,0.03890625)(9.4,0.03890625)
\usefont{T1}{ptm}{m}{n}
\rput(0.6014063,0.56390625){$f$}
\usefont{T1}{ptm}{m}{n}
\rput(0.64140624,-0.85609376){$f'$}
\psline[ArrowInside=->, ArrowInsidePos=0.5,arrowsize=0.25291667cm,linewidth=0.04cm](7.8,0.03890625)(6.6,1.0389062)
\psline[ArrowInside=->, ArrowInsidePos=0.5,arrowsize=0.25291667cm,linewidth=0.04cm](7.8,0.03890625)(6.6,-0.76109374)
\usefont{T1}{ptm}{m}{n}
\rput(7.2814064,0.74390626){$f$}
\usefont{T1}{ptm}{m}{n}
\rput(7.2414064,-0.6560938){$f'$}
\usefont{T1}{ptm}{m}{n}
\rput(1.2114062,-0.41609374){$v$}
\usefont{T1}{ptm}{m}{n}
\rput(7.891406,-0.17609376){$\tilde v$}
\usefont{T1}{ptm}{m}{n}
\rput(9.391406,-0.17609376){$\tilde v'$}
\usefont{T1}{ptm}{m}{n}
\rput(8.581407,0.24390624){$e$}
\end{pspicture} 
}
\end{center}
\caption{Expansion of $v$ of set $\left\{f,f'\right\}$.}
\label{expansion}
\end{figure}
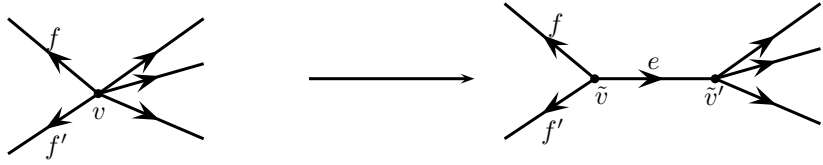

\paragraph{Dead ends}
Let $T$ be a reduced $G$-tree. Let $v$ be a non-vanishing vertex and $f$ an edge with initial vertex $v$. The vertex $v$ is a \textit{dead end} with \textit{wall} $f$ (or $\gr f$) if 
\begin{itemize}[noitemsep]
\item $G_f\subsetneq G_v$,
\item the edge $f$ is slippery or \psad,
\item for every edge $h$ with initial vertex $v$ not in the orbit of $f$, we have the equality $\langle G_{h},G_{f}\rangle=G_{v}$, 
\item there exists an edge $g$ not in the orbit of $f$, with initial vertex $v$, with $G_g\subsetneq G_v$, such that for all edges $h$ not in the orbit of $f$ with initial vertex $v$ there exists $a\in G_v$ such that $G_{a\cdot h}\subset G_g$,

(we may notice that automatically  $\gr g$ and $\gr f$ do not slide one along the other)
\end{itemize}
and if for every $G$-tree $T'$ in the reduced deformation space of $T$ and any representative $v'$ of $\gr v$ in $T'$, there exists an edge with initial edge $v'$ which has the listed properties.

The orbit $\gr v$ of $v$ is a \textit{dead end} if $v$ is a dead end. 

If there exists two edges $f$ and $g$ with distinct orbits such that $v$ is a dead end with wall $f$ and with wall $g$ in $T$, then there are exactly two orbits of edges with initial vertex $\gr v$.
However given another $G$-tree $T'$ in the reduced deformation space of $T$, the orbits of the walls of $v$ in $T$ and $T'$ may differ.% depends on the $G$-tree. 

For example, Figure \ref{figuresensunique} represents the two reduced graphs of groups  of the deformation space of a \GBS decomposition. We may reach the right graph from the left one by a slide of ${\gr h}_2$ along $\bar {\gr h}_1$.

The vertex $\gr v$ is a dead end. In the graph on the left, ${\gr h}_1$ is  the wall $\gr f$ and $\bar {\gr h}_1$ is the edge orbit $\gr g$ as in the fourth point. In the one on the right, $\bar {\gr h}_1$ is the wall $\gr f$ and $\bar {\gr h}_1$ plays the role of $\gr g$. The edge orbit $\gr g$ may be non-unique since another edge may have the same stabilizer as the one of $g$. 

\begin{figure}[!ht]
\begin{center}
% Generated with LaTeXDraw 2.0.8
% Thu Oct 14 22:49:47 CEST 2010
% \usepackage[usenames,dvipsnames]{pstricks}
% \usepackage{epsfig}
% \usepackage{pst-grad} % For gradients
% \usepackage{pst-plot} % For axes
\scalebox{1} % Change this value to rescale the drawing.
{
\begin{pspicture}(0,-1.6592188)(7.681875,1.6792188)
\psline[ArrowInside=-<, ArrowInsidePos=0.5,arrowsize=0.25291667cm,linewidth=0.04cm](1.06,-0.37921876)(0.06,-1.5792187)
\pscircle[ArrowInside=->, ArrowInsidePos=0.5,arrowsize=0.25291667cm,linewidth=0.04,dimen=outer](1.06,0.42078125){0.8}
\psdots[dotsize=0.12](1.06,-0.37921876)
\psdots[dotsize=0.12](0.06,-1.5792187)
\psline[ArrowInside=-<, ArrowInsidePos=0.5,arrowsize=0.25291667cm,linewidth=0.04cm](6.26,-0.37921876)(7.26,-1.5792187)
\pscircle[ArrowInside=->, ArrowInsidePos=0.5,arrowsize=0.25291667cm,linewidth=0.04,dimen=outer](6.26,0.42078125){0.8}
\psdots[dotsize=0.12](6.26,-0.37921876)
\psdots[dotsize=0.12](7.26,-1.5792187)
\usefont{T1}{ptm}{m}{n}
\rput(1.3114063,-0.07421875){$3$}
\usefont{T1}{ptm}{m}{n}
\rput(0.31140625,-0.27421874){$5$}
\usefont{T1}{ptm}{m}{n}
\rput(1.2014062,-0.87421876){$10$}
\usefont{T1}{ptm}{m}{n}
\rput(6.1114063,-0.87421876){$6$}
\usefont{T1}{ptm}{m}{n}
\rput(6.911406,-0.47421876){$3$}
\usefont{T1}{ptm}{m}{n}
\rput(5.911406,-0.07421875){$5$}
\pspolygon[linewidth=0.04,fillstyle=solid,fillcolor=black](6.1,1.1607813)(6.34,1.2807813)(6.34,1.1007812)
\pspolygon[linewidth=0.04,fillstyle=solid,fillcolor=black](0.9,1.1807812)(1.24,1.2807813)(1.2,1.0807812)
\usefont{T1}{ptm}{m}{n}
\rput(1.1214062,1.4857812){${\gr h}_1$}
\usefont{T1}{ptm}{m}{n}
\rput(0.70140624,-1.3342187){${\gr h}_2$}
\usefont{T1}{ptm}{m}{n}
\rput(7.3814063,-1.0742188){${\gr h}_2$}
\usefont{T1}{ptm}{m}{n}
\rput(6.2614064,1.4457812){${\gr h}_1$}
\usefont{T1}{ptm}{m}{n}
\rput(1.3114063,-0.47421876){$\gr v$}
\usefont{T1}{ptm}{m}{n}
\rput(5.911406,-0.47421876){$\gr v$}
\psline[linewidth=0.04cm,fillcolor=black,arrowsize=0.05291667cm 2.0,arrowlength=1.4,arrowinset=0.4]{->}(2.66,0.42078125)(4.66,0.42078125)
\end{pspicture} 
}
\end{center}
\caption{Example of dead end vertex.}
\label{figuresensunique}
\end{figure}
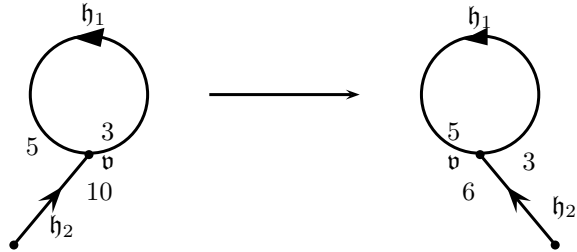

To \textit{blow up} this dead end vertex $v$ in $T$ consists in expanding $v$ with group $G_{v}$ and set $\left\{ f \right\}$ (note that in this case the expansion does not depends of the choice of $f\in \gr f$).  Figure \ref{eclatementsensunique} represents the blow up of $\gr v$ in the left graph of Figure \ref{figuresensunique}. Performing a blow-up does not change the deformation space.
If the dead end $\gr v$ has valence $2$, we may choose equivalently one of the two edges to be the wall. However the blow up does not depend on this choice.

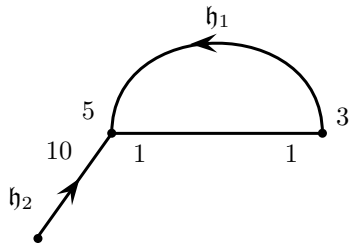
\begin{figure}[!ht]
\begin{center}
\scalebox{1} % Change this value to rescale the drawing.
{
\begin{pspicture}(0,-1.5992187)(5.0028124,1.6192187)
\psline[linewidth=0.04cm,fillcolor=black](1.6209375,-0.11921875)(4.4209375,-0.11921875)
\psdots[dotsize=0.12](1.6209375,-0.11921875)
\psdots[dotsize=0.12](4.4209375,-0.11921875)
\psbezier[ArrowInside=-<, ArrowInsidePos=0.5,arrowsize=0.25291667cm,linewidth=0.04,fillcolor=black](1.6209375,-0.11921875)(1.6209375,1.4807812)(4.4209375,1.4807812)(4.4209375,-0.11921875)
\psline[ArrowInside=->, ArrowInsidePos=0.5,arrowsize=0.25291667cm,linewidth=0.04cm,fillcolor=black](0.6209375,-1.5192188)(1.6209375,-0.11921875)
\usefont{T1}{ptm}{m}{n}
\rput(0.41234374,-0.99421877){${\gr h}_2$}
\usefont{T1}{ptm}{m}{n}
\rput(3.0323439,1.4257812){${\gr h}_1$}
\usefont{T1}{ptm}{m}{n}
\rput(1.3123437,0.18578126){$5$}
\usefont{T1}{ptm}{m}{n}
\rput(4.6923437,0.10578125){$3$}
\usefont{T1}{ptm}{m}{n}
\rput(0.92234373,-0.35421875){$10$}
\usefont{T1}{ptm}{m}{n}
\rput(1.9923438,-0.39421874){$1$}
\usefont{T1}{ptm}{m}{n}
\rput(4.012344,-0.39421874){$1$}
\psdots[dotsize=0.12](0.6409375,-1.5192188)
\end{pspicture} 
}
\end{center}
\caption{Blow up of a dead end vertex}
\label{eclatementsensunique}
\end{figure}

\section{\texorpdfstring{Compatibility JSJ tree of \GBSn{n} groups over $\Z^n$ subgroups}{}}

\subsection{Construction}
Given a group $G$ and any deformation space $ \mathcal D$ of $G$, we may canonically construct a $G$-tree associated to $\mathcal D$, as follows.

\begin{definition}\label{comp}
 Let $T$ be a reduced $G$-tree in the deformation space $\mathcal D$. Define $T_{\mathcal D}$ as the $G$-tree obtained from $T$ by blowing up the dead ends and collapsing the following edges:
\begin{itemize}[noitemsep]
 \item the vanishing edges of $T$
 \item the \psa edges of $T$, 
 \item the non-ascending slippery edges of $T$,
 \item the \tor $2$-slippery edges of $T$.
\end{itemize}
\end{definition}

The non-vanishing edge and vertice orbits  of a reduced $G$-tree $T$ may be identified with another orbit in any $G$-tree $T'$ in the same reduced deformation space. Being a dead end vertex, being \psa, being slippery, and being $2$-slippery does not depend on the chosen tree in the deformation space. But a \tor edge may be changed in a non-ascending edge and vice-versa.

Note that as an edge is either non-ascending, strictly ascending (or descending) or \tor, 	all $2$-slippery edges are collapsed in $T_{\mathcal D}$.

\begin{proposition}\label{construction}
 The tree $T_{\mathcal D}$ does not depend on the choice of the reduced tree in $\mathcal D$ taken for the construction.
\end{proposition}

\begin{proof}
 By Clay-Forester \cite[Corollaire 1.2]{ClayFor09}, two reduced $G$-trees of $\mathcal D$ are related by a finite sequence of deformation moves. We just have to show that we obtain the same tree $T_{\mathcal D}$ starting from two trees $T$ and $T'$ which differ by a single move.
 
 For each move we proceed in three steps. We first prove that the set of orbits we have to collapse is the same before or after performing the move. We then prove that the move commutes with the collapses. Finally we prove that the move commutes with the blow-ups.
 
 For the first step as \psa edges are collapsed, we have to prove that if a non-ascending slippery edge is turned into a \tor edge, then it is $2$-slippery.
 
\begin{enumerate}%[leftmargin=0cm]
 \item Assume that $T$ and $T'$ differ by a slide of an edge $f$ along another edge $e$.
 
Call $v$ the initial vertex of $f$. The only edge which may possibly become \tor is $f$. If $f$ is non-ascending slippery, then as $f$ is reduced we have $G_f\subsetneq G_v$, in $T$, thus also in $T'$. Then $f$ is not \tor in $T'$. Thus no non-ascending slippery edge may be turned into a \tor edge. Thus the edges in $T$ and $T'$ that must be collapsed are the same.

For the second step, it suffices to to check that either $e$ is collapsed in $T_{\mathcal D}$ or $T=T'$.  

As $f$ slides along $e$, the edge $e$ is slippery. If $e$ is not collapsed in $T_{\mathcal D}$, then $e$ must be \tor and not $2$-slippery. Then in the associated graph of group $\gr e$ is a loop and $\gr f$ is the only edge adjacent to $\gr e$ (see Figure \ref{figurecas1-1simplementglissant}).This case is exactly the rigidity case describe in \cite[Theorem 1, case 3]{Levitt05}, that is, the slide leaves $T$ unchanged.

\begin{figure}[!ht]
\begin{center}
 \begin{pspicture}(0,-0.8)(3.8009374,1.0)
  \pscircle[linewidth=0.04,dimen=outer](1.1809375,0.0){0.8}
  \psline[ArrowInside=->, ArrowInsidePos=0.5,arrowsize=0.25291667cm,linewidth=0.04cm,dotsize=0.07055555cm 2.0]{*-*}(3.7809374,0.0)(1.9809375,0.0)
  \usefont{T1}{ptm}{m}{n}
  \rput(0.22234374,0.105){$\gr e$}
  \usefont{T1}{ptm}{m}{n}
  \rput(3.0023437,0.305){$\gr f$}
    \usefont{T1}{ptm}{m}{n}
  \rput(3.6023437,0.205){$\gr v$}
  \psline[linewidth=0.04cm,linestyle=dashed,dash=0.16cm 0.16cm](3.7809374,0.0)(4.7809377,1.2)
  \psline[linewidth=0.04cm,linestyle=dashed,dash=0.16cm 0.16cm](3.7809374,0.0)(4.5809374,0.0)
  \psline[linewidth=0.04cm,linestyle=dashed,dash=0.16cm 0.16cm](3.7809374,0.0)(3.9809375,-0.2)
  \psline[linewidth=0.04cm,linestyle=dashed,dash=0.16cm 0.16cm](3.7809374,0.0)(3.9809375,0.8)
 \end{pspicture}
\end{center}
\caption{The only orbits of edges adjacent to $\gr v$ are $\gr e$, $\bar {\gr e}$ and $\gr f$.}
\label{figurecas1-1simplementglissant}
\end{figure}
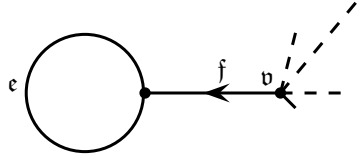

It remains to show that any blow-up of a dead end vertex commutes with the slide. The only non-trivial cases is when an endpoint of $e$ is a dead end. We may assume that the initial vertex $v$ of $e$ is a dead end. As $f$ slides along $e$, neither $e$ nor $\bar f$ may be a wall. Thus if we blow up $v$, the edges $e$ and $\bar f$ will be on the same side of the new edge. Hence this is equivalent to blow up a vertex then slide $f$ along $e$ or slide $f$ along $e$ and then blow up a vertex. 

\item Assume that $T$ and $T'$ differ by an induction on an orbit of edges $\gr e$.

Call $\gr v$ the orbit of the endpoints of $\gr e$.

Here $\gr e$ (or $\bar {\gr e}$) is strictly ascending in both trees. Assume that an edge $f$ was non-ascending slippery and becomes \tor. Then the endpoints of $f$ are in $\gr v$. And as $f$ is \tor in $T'$, the orbit of edges $\gr e$ and $\gr {\bar e}$ may slide along $f$ in $T'$. Then $f$ is $2$-slippery.
Thus the edges in $T$ and $T'$ that must be collapsed are the same. 
 
 As $\gr e$ is strictly ascending in both trees, then it is collapsed in both $T$ and $T'$, and the trees obtained are the same.. 
Moreover as $\gr e$ is ascending the vertex orbit $\gr v$ is not a dead end,  thus blowing up dead end vertices and performing the induction commute.

 \item Assume that $T$ and $T'$ differ by an $\mathcal{A}^{-1}$-move of an orbit of edges $\gr e$ with collapse of an orbit $\gr f$.

When we perform the $\mathcal A^{\pm1}$-move on the orbit of edges $\gr e$, if an orbit of edges $\gr g$ becomes \tor (or stops beeing \tor), then $\gr e$, $\bar {\gr e}$, $\gr g$ ans $\bar {\gr g}$ have same terminal vertex orbit, thus $\gr e$ and $\bar {\gr e}$ may slide along $\gr g$ thus $\gr g$ is two slippery.

In $T$, the orbit $\gr e$ is strictly ascending and $\gr f$ is vanishing, thus both $\gr e$ and $\gr f$ are collapsed in $T_{\mathcal D}$. In $T'$, the orbit $\gr f$ is already collapsed and $\gr e$ is \pa, thus collapsed in $T_{\mathcal D}$. Thus the collapses commute with the move.

It remains to show that the move commutes with the blow up of dead ends. Before performing the $\mathcal A^{-1}$-move the vertex of $\gr e$ is not a dead end, since $\gr e$ is ascending. However the terminal orbit of vertices $\gr v$ of $\gr f$ (which is also the vertex of $\gr e$ after performing the $\mathcal A^{-1}$-move) may be a dead end. This is the only non-trivial case of commutativity. Assume $\gr v$ is a dead end. Call $\gr g$ a wall in $T$. If $\gr g\neq \gr f$, then $\gr g$ remains the wall in $T'$ and it is easy to see that the $\mathcal A^{-1}$-move and the blow up commute.

Two distinct orbits may play the role of the wall if and only if the orbit of vertices $\gr v$ is of valence $2$. Assume that $\gr f$ is the unique wall of $\gr v$, then $\gr v$ is of valence at least $3$. Let $\gr g$ be as in the definition of the dead end and $\gr h$ another orbit of edges with initial vertex $\gr v$ in $T$. Take $v$ a representative of $\gr v$ and $f$, and $g$ representatives of $\gr f$, and $\gr g$ with initial vertex $v$. By assumption $G_g$ is maximal for inclusion and there exists $h$ in the orbit of $\gr h$ with initial vertex $v$ such that $G_{h}\subset G_{g}\not \subset G_{f}$. But in $T'$, as $\gr e$ is \pa, thus  there exists two edges $e'$ and $e$ in the orbit of $\gr e$ such that the initial vertex of $e$ and the terminal vertex of $e'$ are $v$ and $G_{e'}\subset G_{e}$ and we still have $G_{h}\subset G_{g}\not \subset G_{e}(=G_{f})$ with $G_g$ and $G_e$ maximal for the inclusion among groups of edges adjacent to $v$. There is no wall to $v$, this is a contradiction we the fact that $v$ is a dead end.
\end{enumerate}

We obtain the same tree $T_{\mathcal D}$ starting from two reduced tree related by a move. Thus the tree $T_{\mathcal D}$ only depends from the deformation space $\mathcal D$.
\end{proof}

We denote by $T_{comp}$ the $G$-tree associated to Guirardel-Levitt's JSJ deformation space of $G$.

\begin{lemma}\label{compexpand}
Let $T$ be a JSJ tree. Let $\tilde T$ be a reduced Guirardel-Levitt's JSJ tree refined by $T$. Let $\tilde e$ be an edge of $\tilde T$ with initial vertex $\tilde v$ and let $H$ be a subgroup of $G_{\tilde v}$ containing $G_{\tilde e}$ such that for every edge $\tilde f\not \in \tilde {\gr e}$ with initial vertex $\tilde v$ we have $H\subset \langle G_{\tilde e}, G_{\tilde f}\rangle$. Denote by $\tilde T_{\tilde e}$ the tree obtained from $\tilde T$ by performing an expansion of $v$ with group $H$ and set $\{ \tilde e\}$. Then $T$ is compatible with $\tilde T_{\tilde e}$.
\end{lemma}

\begin{proof}
Let $e$ be the lift of $\tilde e$ in $T$ and $v$ a lift of $\tilde v$ such that $G_v= G_{\tilde v}$. Such a $v$ exists since $T$ and $\tilde T$ are both JSJ trees. Call $[w,v]$ the path between the initial vertex $w$ of $e$ and $v$. Call $v_1$ the vertex of $[w,v]$ the closest of $w$ such that $H\subset G_{v_1}$.

If $w= v_1$, then the $G$-tree obtained by the expansion of $w$ with group $H$ and set $\{e\}$  obviously refines $\tilde T_{\tilde e}$. 

Otherwise, denote by $e_1$ the last edge of $[w,v_1]$. 
We have $G_e\subset G_{e_1}\subsetneq H$.
Moreover as for every edge $\tilde f\not \in \tilde {\gr e}$ of initial vertex $\tilde v$ in $\tilde T$, we have $H\subset \langle G_{\tilde f},G_{\tilde e}\rangle$, we obtain $G_{\tilde f}\not \subset G_{e_1}$. Hence for $f$ the lift of $\tilde f$ in $T$, the edges $e$ and $f$ are in distinct components of $T\setminus v_1$. Let $T_e$ be the tree obtained by performing an expansion of $v_1$ with group $H$ and set $\{e_1\}$. Then $T_e$ refines $\tilde T_{\tilde e}$.
\end{proof}

\begin{corollary}\label{Tcomp}
Let $T$ be a JSJ tree.
The $G$-tree $T_{comp}$ is compatible with $T$. Moreover there exists a common refining tree which is a JSJ tree.
\end{corollary}

\begin{proof}
For $f$ an edge of $T_{comp}$ call $T_f$ the tree obtained by collapsing every edge of $T_{comp}$ except $f$.

By the point $1.$ of \cite[Proposition 3.22]{Gl3b}, we have to show that $T$ is compatible with $T_f$ for every edge $f$ of $T_{comp}$. 

Call $\tilde T$ a reduced JSJ tree refined by $T$. By construction, $\tilde T$ refines every tree $T_f$ with $f$ an edge of $T_{comp}$ not obtained by a blow-up. Thus $T$ also refines all these trees.

If $f$ comes from a blow-up, then we may apply Lemma \ref{compexpand}, taking $T$ the JSJ tree, $\tilde T$ the reduced JSJ tree, $v$ the dead end from $f$ comes from, $e$ the wall of $v$ and $H=G_v$. We obtain that $T$ is compatible with  $T_f$.

If follows that $T$ is compatible with $T_{comp}$.
\end{proof}

\subsection{Universal compatibility}

\begin{proposition}\label{univcomp}
Let $G$ be a group that admits a JSJ tree over a set of subgroup $\mathcal A$ and such that the edge groups of every minimal $G$-tree over $\mathcal A$ are elliptic in the JSJ deformation space of $G$. Then $T_{comp}$ is universally compatible. 
\end{proposition}

\begin{proof}
From Lemma 5.3 of  \cite{Gl3a}, every $G$-tree is refined by a JSJ tree. Thus by Corollary \ref{Tcomp}, the tree $T_{comp}$ is universally compatible.
\end{proof}

A \GBSn{n} group $G$ is \textit{generic} if the trivial $G$-tree is not abelian JSJ $G$-tree. The description of the generic \GBSn{n} groups is given in \cite{Moi1}. In particular, if a \GBSn{n} group is not isomorphic to a semi-direct product $\Z^n\rtimes \Z$ then it is generic.

\begin{lemma}\label{GBSsplit}
Let $G$ be a generic \GBSn{n} group. If $T$ is a minimal $G$-tree with an edge group $A\simeq \Z^r$ with $r\leq n$, then $r=n$ and $A$ is universally elliptic over abelian $G$-trees. 
\end{lemma}

\begin{proof}
Since  $G$ is generic, any \GBSn{n} tree is a JSJ $G$-tree (\cite[Theorem 1.2]{Moi1}). Then all edge stabilizers of a given JSJ tree are universally elliptic and commensurable. Let $E$ be one of these stabilizers. Its commensurator is $G$. In every other $G$-tree $T$ over groups $\Z^r$ with $r\leq n$, the group $E$ is elliptic, thus included in the stabilizer of some vertex $v$. As its commensurator is $G$, it is then virtually contained in every stabilizer of vertex in the orbit of $v$. If we assume $T$ is minimal, it implies that $E$ is virtually included in every edge stabilizer. Hence every edge stabilizer is a $\Z^n$ which contains with finite index a universally elliptic group. Thus every edge stabilizer is universally elliptic.
\end{proof} 

\begin{corollary}\label{GGBSncomp}
Let $G$ be a \GBSn{n} group.
The tree $T_{comp}$ associated to the JSJ deformation space of $G$ is compatible with every $G$-tree over the subgroups of $\subset \Z^n$.
\end{corollary}

\begin{proof}
If $G$ is generic, this is a direct consequence of Proposition \ref{univcomp} and Lemma \ref{GBSsplit}.

If $G$ is not generic, then the trivial $G$-tree is a JSJ decomposition. Hence $T_{comp}$ is trivial and universally compatible.
\end{proof}

\subsection{Maximality}

In this section, we prove that in the case of \vGBS groups (and not only \GBSn{n} groups), the tree $T_{comp}$ dominates every other universally compatible trees.

For this we need some technical lemmas, that we divide in three categories. The lemmas of the first category give conditions on trees to be compatible. the lemmas of the second category give  existence of deformation sequences in the deformation space and the lemmas of the third category give some refinement conditions.

\paragraph{Compatibility lemmas}
For $S$ a set of elements of $G$ and $T$ a $G$-tree, call $\mathcal E_S(T)$ the convex hull of all characteristic spaces of elements of $S$ in $T$. Note that if $S$ is reduced to one element then $\mathcal E_S(T)$ is just the characteristic space of this element.

\begin{lemma}\label{incompatibility}\label{constructionnoncomp}
Let $T$ and $T'$ be two $G$-trees and let $a$, $b$, $c$, and $d$ be four elements of $G$. 
\begin{itemize}
\item If the intersection $\mathcal E_{a,b}(T)\cap \mathcal E_{c}(T)$ is empty, and the intersection $\mathcal E_{a,b}(T')\cap \mathcal E_{c}(T')$ contains an edge, then $T$ and $T'$ are not compatible.
\item If $\mathcal E_{\{a,b\}}(T)$ and $\mathcal E_{\{c,d\}}(T)$ are disjoint in $T$ and $\mathcal E_{\{a,c\}}(T')$ and $\mathcal E_{\{b,d\}}(T')$ are disjoint in $T'$, then $T$ and $T'$ are not compatible.
\end{itemize}
\end{lemma}

\begin{proof}
First notice that if $\tilde T$ is a refinement of $T$, then $\mathcal E_{S}(\tilde T)$ surjects onto $\mathcal E_S(T)$ via the natural map.

Thus for the first point, the two assumptions are stable by refinement.  Hence a common refinement should have both properties. This is impossible.

For the second point, assume there exists a common refinement $\tilde T$. As $\mathcal E_{\{a,b\}}(T)$ and $\mathcal E_{\{c,d\}}(T)$ are disjoint in $T$ then $\mathcal E_{\{a,b\}}(\tilde T)$ and $\mathcal E_{\{c,d\}}(\tilde T)$ are also disjoint in $\tilde T$. Call $B$ the bridge in $\tilde T$ between $\mathcal E_{\{a,b\}}(\tilde T)$ and $\mathcal E_{\{c,d\}}(\tilde T)$. Then $B$ is contained in $\mathcal E_{\{a,c\}}(\tilde T)$ and  $\mathcal E_{\{b,d\}}(\tilde T)$ in $\tilde T$. Thus $\mathcal E_{\{a,c\}}(\tilde T)\cap \mathcal E_{\{b,d\}}(\tilde T)$ is not empty in $\tilde T$, hence $\mathcal E_{\{a,c\}}(T')\cap \mathcal E_{\{b,d\}}(T')$ is not empty in $T'$.
\end{proof}

\begin{lemma}\label{normalisation}
Let $G$ be a \vGBS group and $T$ be an abelian JSJ $G$-tree. Let $e$ be a reduced edge with initial vertex $v$ in $T$ such that $G_e \subsetneq G_v$. Then there exists $g\in G_v\setminus G_e$ that centralizes $G_e$.
 \end{lemma} 

\begin{proof}
If $G_v$ is abelian, the lemma is trivial. Otherwise by \cite{Moi1}, the group $G_v$ is a semi-direct product $\Z^n\rtimes_\varphi \langle t \rangle$, with $\varphi$ a non-trivial automorphism of $\Z^n$, and $G_e\subsetneq \Z^n$, thus any element in $\Z^n\setminus G_e$ centralizes $G_e$.
\end{proof}

\begin{corollary}\label{casreduit}
Let $T$ be a $G$-tree. Let $\gr e$ be an orbit of  reduced edges of $T$.
Assume $\gr e$ has one of the following properties:
\begin{enumerate}[noitemsep]
\item $\gr e$ is \pa in $T$ and $G$ is not an ascending HNN-extension,
\item $\gr e$ is not ascending and there exists an orbit of reduced edges $\gr f$ which slides along two consecutive edges of $\gr e\cup \bar {\gr e}$ in $T$,
\item $T$ is a \vGBS-tree, the orbit $\gr e$ is not ascending and there exists an orbit of reduced edges $\gr f$ which slides along $\gr e$ in $T$.
\item $\gr e$ is \tor of endpoint $\gr v$ and there exist two distinct orbits of reduced edges $\gr f$ and $\gr g$ with terminal vertex $\gr v$, distinct from the orbits $\gr e$ or $\bar {\gr e}$.
\end{enumerate}
Let $T'$ be a collapse of $T$ such that $e$ is not collapsed in $T'$. Then $T'$ is not universally compatible.
\end{corollary}

Here we do not need the slides to be admissible.

%\begin{proof}~
\proof~
\begin{enumerate}[noitemsep]
\item If $\gr e$ is \pa in $T$, then $\gr e$ is \pa in $T'$. By \cite[Proposition 7.1]{GL07} the tree $T'$ is not universally compatible.
\item Let $e$, and $e'$ be edges of $\gr e\cup \gr {\bar e}$ and $f$ an edge of $\gr f$ such that $f$ may consecutively slide along $e$ and $\bar e'$ in $T$. Let $v$ and $v'$ be the initial and terminal vertices of $e$, call $v''$ the terminal vertex of $e'$ and $w$ the initial vertex of $f$ (note that the terminal vertex of $f$ is $v$ and the initial vertex of $e'$ is $v'$).

As $e$ is reduced and non ascending, we may find three elements $b\in G_v\setminus G_e$, $c\in G_{v'}\setminus (G_e\cup G_{e'})\neq \emptyset$ and $d\in G_{v''}\setminus G_{e'}$.
If $f$ is not ascending let $a$ be in $G_{w}\setminus G_f$. If $f$ is ascending, let $t$ be such that $t\cdot w=v$, an call $a=t^{-1}ct$.
The sets $\mathcal E_{\{a,b\}}(T)$ and $\mathcal E_{\{c,d\}}(T)$ are separated by an edge in $\gr e$. As $\gr e$ is not collapsed in $T'$, the sets $\mathcal E_{\{a,b\}}(T')$ and $\mathcal E_{\{c,d\}}(T')$ are disjoint in $T'$. But making slide $f$ along $e$ and $\bar e'$, we now have a new tree $T''$ in which $\mathcal E_{\{b,c\}}(T'')\cap \mathcal E_{\{a,d\}}(T'')=\emptyset$ (see Figure \ref{glisssimple}). By Lemma \ref{incompatibility}, the tree $T'$ is not universally compatible.
\begin{figure}[!ht]
\begin{center}
% Generated with LaTeXDraw 2.0.5
% Tue Feb 01 11:02:20 CET 2011
% \usepackage[usenames,dvipsnames]{pstricks}
% \usepackage{epsfig}
% \usepackage{pst-grad} % For gradients
% \usepackage{pst-plot} % For axes
\scalebox{1} % Change this value to rescale the drawing.
{
\begin{pspicture}(0,-0.9429687)(11.839531,0.9429687)
\psline[ArrowInside=->, ArrowInsidePos=0.5,arrowsize=0.25291667cm,linewidth=0.04cm,fillcolor=black](1.6995313,-0.5154688)(3.2995312,-0.5154688)
\psdots[dotsize=0.12](3.2995312,-0.5154688)
\psdots[dotsize=0.12](1.6995313,-0.5154688)
\psline[ArrowInside=->, ArrowInsidePos=0.5,arrowsize=0.25291667cm,linewidth=0.04cm,fillcolor=black](3.2995312,-0.5154688)(4.8995314,-0.5154688)
\psdots[dotsize=0.12](4.8995314,-0.5154688)
\psdots[dotsize=0.12](4.8995314,-0.5154688)
\psline[ArrowInside=-<, ArrowInsidePos=0.5,arrowsize=0.25291667cm,linewidth=0.04cm,fillcolor=black](1.6995313,-0.5154688)(0.29953125,0.48453125)
\psdots[dotsize=0.12](0.29953125,0.48453125)
\usefont{T1}{ptm}{m}{n}
\rput(1.2823437,0.22953124){$f$}
\usefont{T1}{ptm}{m}{n}
\rput(2.4823437,-0.7704688){$e$}
\usefont{T1}{ptm}{m}{n}
\rput(4.1423435,-0.7704688){$e'$}
\usefont{T1}{ptm}{m}{n}
\rput(0.22234374,0.74953127){$a$}
\usefont{T1}{ptm}{m}{n}
\rput(1.7523437,-0.25046876){$b$}
\usefont{T1}{ptm}{m}{n}
\rput(3.2523437,-0.25046876){$c$}
\usefont{T1}{ptm}{m}{n}
\rput(4.862344,-0.23046875){$d$}
\psline[ArrowInside=->, ArrowInsidePos=0.5,arrowsize=0.25291667cm,linewidth=0.04cm,fillcolor=black](6.8995314,-0.5154688)(8.499531,-0.5154688)
\psdots[dotsize=0.12](8.499531,-0.5154688)
\psdots[dotsize=0.12](6.8995314,-0.5154688)
\psline[ArrowInside=->, ArrowInsidePos=0.5,arrowsize=0.25291667cm,linewidth=0.04cm,fillcolor=black](8.499531,-0.5154688)(10.099531,-0.5154688)
\psdots[dotsize=0.12](10.099531,-0.5154688)
\psdots[dotsize=0.12](10.099531,-0.5154688)
\psline[ArrowInside=-<, ArrowInsidePos=0.5,arrowsize=0.25291667cm,linewidth=0.04cm,fillcolor=black](10.1,-0.5154688)(11.5,0.48453125)
\psdots[dotsize=0.12](11.5,0.48453125)
\usefont{T1}{ptm}{m}{n}
\rput(10.479062,0.22953124){$f$}
\usefont{T1}{ptm}{m}{n}
\rput(7.6823435,-0.7704688){$e$}
\usefont{T1}{ptm}{m}{n}
\rput(9.342343,-0.7704688){$e'$}
\usefont{T1}{ptm}{m}{n}
\rput(11.5390625,0.74953127){$a$}
\usefont{T1}{ptm}{m}{n}
\rput(6.952344,-0.25046876){$b$}
\usefont{T1}{ptm}{m}{n}
\rput(8.452344,-0.25046876){$c$}
\usefont{T1}{ptm}{m}{n}
\rput(10.062344,-0.23046875){$d$}
\psline[linewidth=0.03cm,arrowsize=0.05291667cm 2.0,arrowlength=1.4,arrowinset=0.4]{->}(5.56,-0.45703125)(6.36,-0.45703125)
\end{pspicture} 
}
\end{center}
\caption{}
\label{glisssimple}
\end{figure}
\item Assume that $T$ is a \vGBS-tree. Let $f$ be an edge which slides along a reduced non ascending edge $e$. Call $v$ the terminal vertex of $e$. The set $G_v\setminus G_e$ is not empty since $e$ is reduced and non-ascending. Applying Lemma \ref{normalisation}, there exists $\lambda$ in this set such that $G_e=G_{\lambda\cdot e}$, thus $f$ slides consecutively on $e$ and $\lambda \cdot \bar e$. Thus the point $3$ is implied by the point $2$.

\item As $\gr e$ is \tor then $\gr f$ slides along two consecutive edges in $\gr e$. If $\gr f\neq \bar {\gr g}$, collapsing $\gr g(\neq \gr e, \bar{\gr e})$ in $T$, the edge $e$ is no more ascending (and still reduced) in the new tree $T''$, we may apply the point $2$. Note that $T''$ may perhaps not refines $T'$, however it refines the tree obtained from $T'$ by collapsing $\gr g$, this is enough to obtain the result.

If $\gr g$ is in the orbit of $\bar {\gr f}$. Let $f$ be an edge of $\gr f$. Call $v$ its terminal vertex and let $g$ be an edge of $\gr g=\bar{\gr f}$ with terminal vertex $v$. Let $e$ be an edge of $\gr e$ with initial vertex $v$ (on which $f$ may slide). Let $v'$ be the terminal vertex of $e$ and $w$ be the initial vertex of $g$. Call $t\in G$ an element which sends $v$ onto $v'$ and $s\in G$ which sends $v$ onto $w$.

As $e$ is \tor $G_e=G_{t\cdot e}$, and $f$ slides consecutively along $e$ and $t\cdot e$. Define the four elements $a=s^{-1}ts$, $b=sts^{-1}$, $c=t^{-1}sts^{-1}t^{1}$ and $d=t^{-2}sts^{-1}t^{2}$ (see Figure \ref{doubleboucle}). Then the sets $\mathcal E_{\{a,b\}}(T')$ and $\mathcal E_{\{c,d\}}(T')$ are disjoint in $T'$ but after making slide $f$ along $e$, and $t\cdot e$, we have $\mathcal E_{\{b,c\}}\cap \mathcal E_{\{a,d\}}=\emptyset$. Thus by Lemma \ref{incompatibility}, the tree $T'$ cannot be universally compatible.\hfill\qed

\begin{figure}[!ht]
\begin{center}
% Generated with LaTeXDraw 2.0.5
% Tue Feb 01 11:25:56 CET 2011
% \usepackage[usenames,dvipsnames]{pstricks}
% \usepackage{epsfig}
% \usepackage{pst-grad} % For gradients
% \usepackage{pst-plot} % For axes
\scalebox{1} % Change this value to rescale the drawing.
{
\begin{pspicture}(0,-1.9181249)(5.543281,1.9181249)
\psline[ArrowInside=->, ArrowInsidePos=0.5,arrowsize=0.25291667cm,linewidth=0.04cm](1.46,0.07968745)(2.86,0.07968745)
\psdots[dotsize=0.12](1.46,0.07968745)
\psline[ArrowInside=->, ArrowInsidePos=0.5,arrowsize=0.25291667cm,linewidth=0.04cm](2.86,0.07968745)(4.26,0.07968745)
\psdots[dotsize=0.12](2.86,0.07968745)
\psdots[dotsize=0.12](4.26,0.07968745)
\psline[ArrowInside=->, ArrowInsidePos=0.5,arrowsize=0.25291667cm,linewidth=0.04cm](1.46,-1.1203126)(1.46,0.07968745)
\psdots[dotsize=0.12](1.46,-1.1203126)
\psline[ArrowInside=->, ArrowInsidePos=0.5,arrowsize=0.25291667cm,linewidth=0.04cm](0.06,-1.1203126)(1.46,-1.1203126)
\psdots[dotsize=0.12](0.06,-1.1203126)
\psline[ArrowInside=->, ArrowInsidePos=0.5,arrowsize=0.25291667cm,linewidth=0.04cm](2.86,1.2796874)(2.86,0.07968745)
\psline[ArrowInside=->, ArrowInsidePos=0.5,arrowsize=0.25291667cm,linewidth=0.04cm](2.06,1.0796875)(2.86,1.2796874)
\psline[ArrowInside=->, ArrowInsidePos=0.5,arrowsize=0.25291667cm,linewidth=0.04cm](2.86,1.2796874)(3.66,1.4796875)
\psline[ArrowInside=->, ArrowInsidePos=0.5,arrowsize=0.25291667cm,linewidth=0.04cm](4.26,1.0796875)(4.26,0.07968745)
\psline[ArrowInside=->, ArrowInsidePos=0.5,arrowsize=0.25291667cm,linewidth=0.04cm](3.3,0.71968746)(4.26,1.0796875)
\psline[ArrowInside=->, ArrowInsidePos=0.5,arrowsize=0.25291667cm,linewidth=0.04cm](4.26,1.0796875)(5.26,1.4796875)
\psline[ArrowInside=->, ArrowInsidePos=0.5,arrowsize=0.25291667cm,linewidth=0.04cm](1.46,-1.1203126)(2.66,-1.1203126)
\psdots[dotsize=0.12](2.66,-1.1203126)
\psdots[dotsize=0.12](2.86,1.2796874)
\psdots[dotsize=0.12](3.66,1.4796875)
\psdots[dotsize=0.12](5.26,1.4796875)
\psdots[dotsize=0.12](3.28,0.71968746)
\psdots[dotsize=0.12](2.06,1.0796875)
\usefont{T1}{ptm}{m}{n}
\rput(1.1528125,0.04468745){$v$}
\usefont{T1}{ptm}{m}{n}
\rput(2.8328125,-0.15531255){$v'$}
\usefont{T1}{ptm}{m}{n}
\rput(4.3928123,-0.17531255){$t^2\cdot v$}
\usefont{T1}{ptm}{m}{n}
\rput(1.2028126,0.7046875){$g$}
\usefont{T1}{ptm}{m}{n}
\rput(1.2228125,-0.59531254){$f$}
\usefont{T1}{ptm}{m}{n}
\rput(2.1028125,0.28468746){$e$}
\usefont{T1}{ptm}{m}{n}
\rput(1.3328124,-1.6953125){$\mathcal A_a$}
\usefont{T1}{ptm}{m}{n}
\rput(2.7528124,1.7246875){$\mathcal A_c$}
\usefont{T1}{ptm}{m}{n}
\rput(4.362812,1.5846875){$\mathcal A_d$}
\psline[linewidth=0.03cm,arrowsize=0.05291667cm 2.0,arrowlength=1.4,arrowinset=0.4]{->}(0.26,-1.3203125)(2.66,-1.3203125)
\psline[linewidth=0.03cm,arrowsize=0.05291667cm 2.0,arrowlength=1.4,arrowinset=0.4]{->}(2.06,1.2796874)(3.66,1.6796875)
\psline[linewidth=0.03cm,arrowsize=0.05291667cm 2.0,arrowlength=1.4,arrowinset=0.4]{->}(3.26,0.8796874)(5.12,1.6196874)
\psdots[dotsize=0.12](4.26,1.0796875)
\psline[ArrowInside=->, ArrowInsidePos=0.5,arrowsize=0.25291667cm,linewidth=0.04cm](1.46,1.2596874)(1.46,0.05968745)
\psline[ArrowInside=->, ArrowInsidePos=0.5,arrowsize=0.25291667cm,linewidth=0.04cm](0.46,1.0596875)(1.46,1.2596874)
\psline[ArrowInside=->, ArrowInsidePos=0.5,arrowsize=0.25291667cm,linewidth=0.04cm](1.46,1.2596874)(2.46,1.4596875)
\psdots[dotsize=0.12](0.46,1.0596875)
\psdots[dotsize=0.12](2.46,1.4596875)
\usefont{T1}{ptm}{m}{n}
\rput(1.5628128,1.7046875){$\mathcal A_b$}
\psline[linewidth=0.03cm,arrowsize=0.05291667cm 2.0,arrowlength=1.4,arrowinset=0.4]{->}(0.66,1.2596874)(2.26,1.5796875)
\psdots[dotsize=0.12](1.46,1.2596874)
\end{pspicture} 
}
\end{center}
\caption{}
\label{doubleboucle}
\end{figure}
\end{enumerate}
%\end{proof}

\begin{corollary}\label{casfacile} Let $T$ be a $G$-tree. Let $e$ be an edge of $T$ with initial and terminal vertices $v$ and $v'$.
Call $T_e$ the tree obtained from $T$ by collapsing all orbits of edges except $\gr e$.

If there are two reduced non-ascending edges, $f$ with initial vertex $v$ and $g$ with initial vertex $v'$, such that $\gr f\neq \gr {\bar g}$, $G_f\subset G_g$, and $G_e\not \subset G_g$, then $T_e$ is not universally compatible.
\end{corollary}

Here the edge $e$ is not necessarily reduced.

\begin{proof}
As $f$ and $g$ are reduced and non ascending, we may take $a$ an element which fixes the terminal vertex of $f$ but not $f$ and $d$ which fixes the terminal vertex of $g$ but not $g$.
Let $r$ be an element of $G_e\setminus (G_g\cap G_e)$ and define $f'=r\cdot f$ and $g'=r\cdot g$. Define $b=rar^{-1}$ and $c=rdr^{-1}$.
We are as in Figure \ref{cassimple}.
Here $\mathcal E_{\{a,b\}}(T_e)$ and $\mathcal E_{\{c,d\}}(T_e)$ are disjoint. But $\bar f$ may slides along $e$ and $g$. By equivariance, at the same time $\bar f'$ slides along $e$ and $g'$.

After performing these slides, we obtain a tree $T'$ in which $\mathcal E_{\{a,d\}}(T')$ and $\mathcal E_{\{b,c\}}(T')$ are disjoint. By Lemma \ref{incompatibility}, the tree $T_e$ is not universally compatible.
\end{proof}

\begin{figure}[!ht]
\begin{center}
% Generated with LaTeXDraw 2.0.5
% Fri Feb 04 15:07:13 CET 2011
% \usepackage[usenames,dvipsnames]{pstricks}
% \usepackage{epsfig}
% \usepackage{pst-grad} % For gradients
% \usepackage{pst-plot} % For axes
\scalebox{1} % Change this value to rescale the drawing.
{
\begin{pspicture}(0,-1.4929688)(5.6028123,1.4929688)
\psline[ArrowInside=->, ArrowInsidePos=0.5,arrowsize=0.25291667cm,linewidth=0.04cm](1.9809375,-0.20546874)(0.7809375,1.1945312)
\psline[ArrowInside=->, ArrowInsidePos=0.5,arrowsize=0.25291667cm,linewidth=0.04cm](1.9809375,-0.20546874)(0.5809375,-1.2054688)
\psline[ArrowInside=->, ArrowInsidePos=0.5,arrowsize=0.25291667cm,linewidth=0.04cm](1.9809375,-0.20546874)(3.7809374,-0.20546874)
\psline[ArrowInside=->, ArrowInsidePos=0.5,arrowsize=0.25291667cm,linewidth=0.04cm](3.7809374,-0.20546874)(4.9809375,0.9945313)
\psline[ArrowInside=->, ArrowInsidePos=0.5,arrowsize=0.25291667cm,linewidth=0.04cm](3.7809374,-0.20546874)(4.9809375,-1.2054688)
\usefont{T1}{ptm}{m}{n}
\rput(2.7623436,-0.04046875){$e$}
\usefont{T1}{ptm}{m}{n}
\rput(1.6223438,0.65953124){$f$}
\usefont{T1}{ptm}{m}{n}
\rput(1.4423437,-0.98046875){$f'$}
\usefont{T1}{ptm}{m}{n}
\rput(4.0823436,0.49953124){$g$}
\usefont{T1}{ptm}{m}{n}
\rput(4.182344,-0.98046875){$g'$}
\usefont{T1}{ptm}{m}{n}
\rput(0.46234375,1.2995312){$a$}
\usefont{T1}{ptm}{m}{n}
\rput(0.23234375,-1.2204688){$b$}
\usefont{T1}{ptm}{m}{n}
\rput(5.302344,-1.3204688){$c$}
\usefont{T1}{ptm}{m}{n}
\rput(5.2923436,1.1195313){$d$}
\psdots[dotsize=0.12](0.7809375,1.1945312)
\psdots[dotsize=0.12](4.9809375,0.9945313)
\psdots[dotsize=0.12](3.7809374,-0.20546874)
\psdots[dotsize=0.12](1.9809375,-0.20546874)
\psdots[dotsize=0.12](0.5809375,-1.2054688)
\psdots[dotsize=0.12](4.9809375,-1.2054688)
\end{pspicture} 
}
\end{center}
\caption{}
\label{cassimple}
\end{figure}

\begin{lemma}\label{casalmostasc}
Let $T$ be a $G$-tree. Let $e$ be an edge of $T$ whose endpoints are in distinct orbits, call $v$ the initial vertex of $e$ and $v'$ its terminal vertex. Call $T_e$ the tree obtained from $T$ by collapsing all orbits of edges except $\gr e$.
If there exists two reduced edges $f$ and $f'$ in the same orbit $\gr f\neq \gr e$, such that the initial vertex of $f$ is $v$, the terminal vertex of $f'$ is $v'$, and $G_{f}\subsetneq G_{f'}$, then $T_e$ is not universally compatible.
\end{lemma}

\begin{proof}
As $f$ and $f'$ are in the same orbit we may take $t\in G$ such that $t\cdot f=f'$. Let us define $e'=t\cdot e$ and $f''=t\cdot f'$.

First assume $G_v=G_e \subset G_{f'}$. As $G_f\subsetneq G_{f'}$, we have $G_{f'}\subsetneq G_{f''}$ and $G_e\subsetneq G_{e'}$. Thus collapsing $f$ in $T$, the edge $e$ becomes \pa, hence it stays \pa in $T_e$. 
By Corollary \ref{casreduit}, the tree $T_e$ is not universally compatible.

Assume now that $G_v\neq G_e$ or $G_e \not \subset G_{f'}$.
We have $G_{f'}\subsetneq G_{f''}$. Take $a$ in $G_{f''}\setminus G_{f'}$, and call $f'''=a\cdot f'$. As $G_v\neq G_e$ or $G_e \not \subset G_{f'}$, as $f'$ is reduced and $\gr v\neq \gr v'$, we have $G_{f'}\subsetneq G_{v'}$. 

We may now take an element $b$ which fixes the terminal vertex of $f'''$ but not $f'''$ and an element $c$ which fixes the terminal vertex of $f$ but not $f$.

If $G_v\neq G_e$, take $d$ in $G_v\setminus G_e$ (automatically $d\not \in G_f$ and $d\not\in G_{f'}$). If $G_v=G_e$ then $G_e \not \subset G_{f'}$. Take $h\in G_e \setminus (G_{f'}\cap G_e)$. Define $\tilde f=h \cdot f$ and take $d$ which fixes the terminal vertex of $\tilde f$ but not $\tilde f$.

Then $\mathcal E_{\{a,b\}}(T)$ and $\mathcal E_{\{c,d\}}(T)$ are separated at least by $e$. Thus $\mathcal E_{\{a,b\}}(T_e)\cap\mathcal E_{\{c,d\}}(T_e)$ is empty.

Now call $T'$ the tree obtained by collapsing $e$ in $T$ and call $w$ the initial vertex of $f$ in $T'$. Perform the expansion of the vertex $w$ with group $G_f'$ and set $\left\{f,\bar {f'}\right\}$. Call $g$ the edge obtained in this expansion, and $w_1$ the terminal edge of $g$.

We may note that:
 \begin{itemize}[topsep=0cm,noitemsep]\item the element $a$ still stabilizes $f''$ but not $f'$,
\item the element $b$ stabilizes $a\cdot w_1$ but not $a\cdot g$,
\item the element $c$ stabilizes $t^{-1}\cdot w_1$ but not $t^{-1}\cdot g$,
\item  if $G_v\neq G_e$, the element $d$ stabilizes $w_1$ but not $g$,
\item  if $G_v= G_e$, the edge $g$ is between $f$ and the characteristic space of $d$.
\end{itemize}

Then making the edge $a\cdot g$ slide along $\bar f'''$ and $f'$, we obtain a tree in which $\mathcal E_{\{a,d\}}\cap\mathcal E_{\{b,c\}}$ is empty.
By Lemma \ref{incompatibility}, the edge $T_e$ is not universally compatible (see the two slides in Figure \ref{casstrictasc}).
\begin{figure}[!ht]
\begin{center}
% Generated with LaTeXDraw 2.0.5
% Mon Feb 07 20:13:50 CET 2011
% \usepackage[usenames,dvipsnames]{pstricks}
% \usepackage{epsfig}
% \usepackage{pst-grad} % For gradients
% \usepackage{pst-plot} % For axes
\scalebox{0.9} % Change this value to rescale the drawing.
{
\begin{pspicture}(0,-4.4176562)(12.640938,3.4176562)
\psline[ArrowInside=->, ArrowInsidePos=0.5,arrowsize=0.25291667cm,linewidth=0.04cm](2.4609375,1.7792188)(1.0609375,1.7792188)
\psline[ArrowInside=->, ArrowInsidePos=0.5,arrowsize=0.25291667cm,linewidth=0.04cm](3.8609376,1.7792188)(2.4609375,1.7792188)
\psline[ArrowInside=->, ArrowInsidePos=0.5,arrowsize=0.25291667cm,linewidth=0.04cm](5.2609377,1.7792188)(3.8609376,1.7792188)
\psdots[dotsize=0.12](2.4609375,1.7792188)
\psdots[dotsize=0.12](3.8609376,1.7792188)
\psdots[dotsize=0.12](5.2609377,1.7792188)
\psdots[dotsize=0.12](1.0609375,1.7792188)
\psline[ArrowInside=->, ArrowInsidePos=0.5,arrowsize=0.25291667cm,linewidth=0.04cm](3.8609376,1.7792188)(2.4609375,2.5792189)
\psdots[dotsize=0.12](2.4609375,2.5792189)
\psline[ArrowInside=-<, ArrowInsidePos=0.5,arrowsize=0.25291667cm,linewidth=0.04cm,linestyle=dashed,dash=0.16cm 0.16cm](2.4609375,0.77921873)(2.4609375,1.5792187)
\psline[ArrowInside=->, ArrowInsidePos=0.5,arrowsize=0.25291667cm,linewidth=0.04cm,linestyle=dashed,dash=0.16cm 0.16cm](1.0609375,1.7792188)(1.0609375,0.77921873)
\psline[ArrowInside=-<, ArrowInsidePos=0.5,arrowsize=0.25291667cm,linewidth=0.04cm,linestyle=dashed,dash=0.16cm 0.16cm](2.4609375,0.77921873)(3.2609375,-0.22078125)
\psline[ArrowInside=->, ArrowInsidePos=0.5,arrowsize=0.25291667cm,linewidth=0.04cm](3.2609375,-0.22078125)(1.8609375,-0.22078125)
\psline[ArrowInside=-<, ArrowInsidePos=0.5,arrowsize=0.25291667cm,linewidth=0.04cm,linestyle=dashed,dash=0.16cm 0.16cm](3.6609375,3.1792188)(2.4609375,2.5792189)
\psline[ArrowInside=-<, ArrowInsidePos=0.5,arrowsize=0.25291667cm,linewidth=0.04cm,linestyle=dashed,dash=0.16cm 0.16cm](0.4609375,-0.22078125)(1.8609375,-0.22078125)
\psdots[dotsize=0.12](1.8609375,-0.22078125)
\psdots[dotsize=0.12](3.2609375,-0.22078125)
\psdots[dotsize=0.12](2.4609375,0.77921873)
\psdots[dotsize=0.12](1.0609375,0.77921873)
\psdots[dotsize=0.12](3.6609375,3.1792188)
\usefont{T1}{ptm}{m}{n}
\rput(1.7623438,1.9242188){$f$}
\usefont{T1}{ptm}{m}{n}
\rput(3.2023437,1.5842187){$f'$}
\usefont{T1}{ptm}{m}{n}
\rput(4.722344,1.9242188){$f''$}
\usefont{T1}{ptm}{m}{n}
\rput(3.4023438,2.3442187){$f'''$}
\usefont{T1}{ptm}{m}{n}
\rput(2.7223437,1.2242187){$g$}
\usefont{T1}{ptm}{m}{n}
\rput(2.2023437,0.7842187){$w_1$}
\usefont{T1}{ptm}{m}{n}
\rput(2.6260939,-0.46578124){($\tilde f$)}
\usefont{T1}{ptm}{m}{n}
\rput(4.0423436,1.5642188){$a$}
\usefont{T1}{ptm}{m}{n}
\rput(3.9723437,3.2242188){$b$}
\usefont{T1}{ptm}{m}{n}
\rput(1.0423437,0.60421876){$c$}
\usefont{T1}{ptm}{m}{n}
\rput(2.8123438,0.76421875){$d$}
\usefont{T1}{ptm}{m}{n}
\rput(0.35234374,-0.45578125){$(d)$}
\psdots[dotsize=0.12](0.4209375,-0.24078125)
\psline[linewidth=0.03cm,arrowsize=0.05291667cm 2.0,arrowlength=1.4,arrowinset=0.4]{->}(4.1409373,-0.37921873)(4.9609375,-1.04078125)

\psline[ArrowInside=->, ArrowInsidePos=0.5,arrowsize=0.25291667cm,linewidth=0.04cm](6.2409377,-1.9607813)(4.8409376,-1.9607813)
\psline[ArrowInside=->, ArrowInsidePos=0.5,arrowsize=0.25291667cm,linewidth=0.04cm](7.640938,-1.9607813)(6.2409377,-1.9607813)
\psline[ArrowInside=->, ArrowInsidePos=0.5,arrowsize=0.25291667cm,linewidth=0.04cm](9.040937,-1.9607813)(7.640938,-1.9607813)
\psdots[dotsize=0.12](6.2409377,-1.9607813)
\psdots[dotsize=0.12](7.640938,-1.9607813)
\psdots[dotsize=0.12](9.040937,-1.9607813)
\psdots[dotsize=0.12](4.8409376,-1.9607813)
\psline[ArrowInside=->, ArrowInsidePos=0.5,arrowsize=0.25291667cm,linewidth=0.04cm](7.640938,-1.9607813)(6.2409377,-1.16078125)
\psdots[dotsize=0.12](6.2409377,-1.16078125)
\psline[ArrowInside=-<, ArrowInsidePos=0.5,arrowsize=0.25291667cm,linewidth=0.04cm,linestyle=dashed,dash=0.16cm 0.16cm](6.2409377,-2.9607812)(7.640938,-1.9607813)
\psline[ArrowInside=->, ArrowInsidePos=0.5,arrowsize=0.25291667cm,linewidth=0.04cm,linestyle=dashed,dash=0.16cm 0.16cm](6.2409377,-1.9607813)(4.8409376,-2.9607812)
\psline[ArrowInside=-<, ArrowInsidePos=0.5,arrowsize=0.25291667cm,linewidth=0.04cm,linestyle=dashed,dash=0.16cm 0.16cm](6.2409377,-2.9607812)(7.840938,-3.9607813)
\psline[ArrowInside=->, ArrowInsidePos=0.5,arrowsize=0.25291667cm,linewidth=0.04cm](7.840938,-3.9607813)(6.4409375,-3.9607813)
\psline[ArrowInside=-<, ArrowInsidePos=0.5,arrowsize=0.25291667cm,linewidth=0.04cm,linestyle=dashed,dash=0.16cm 0.16cm](7.440937,-0.57921876)(7.640938,-1.9607813)
\psline[ArrowInside=-<, ArrowInsidePos=0.5,arrowsize=0.25291667cm,linewidth=0.04cm,linestyle=dashed,dash=0.16cm 0.16cm](5.0409374,-3.9607813)(6.4409375,-3.9607813)
\psdots[dotsize=0.12](6.4409375,-3.9607813)
\psdots[dotsize=0.12](7.840938,-3.9607813)
\psdots[dotsize=0.12](6.2409377,-2.9607812)
\psdots[dotsize=0.12](4.8409376,-2.9607812)
\psdots[dotsize=0.12](7.440937,-0.57921876)
\usefont{T1}{ptm}{m}{n}
\rput(5.5423436,-1.81578124){$f$}
\usefont{T1}{ptm}{m}{n}
\rput(6.9823437,-2.1757813){$f'$}
\usefont{T1}{ptm}{m}{n}
\rput(8.502344,-1.81578124){$f''$}
\usefont{T1}{ptm}{m}{n}
\rput(7.1823435,-1.39578125){$f'''$}
\usefont{T1}{ptm}{m}{n}
\rput(6.5023437,-2.5157813){$g$}
\usefont{T1}{ptm}{m}{n}
\rput(5.972344,-2.9757812){$w_1$}
\usefont{T1}{ptm}{m}{n}
\rput(7.822344,-2.1757812){$a$}
\usefont{T1}{ptm}{m}{n}
\rput(7.752344,-0.52421875){$b$}
\usefont{T1}{ptm}{m}{n}
\rput(4.842344,-3.2157813){$c$}
\usefont{T1}{ptm}{m}{n}
\rput(6.592344,-2.9757812){$d$}
\usefont{T1}{ptm}{m}{n}
\rput(4.932344,-4.2957812){$(d)$}
\psdots[dotsize=0.12](5.0009375,-3.9807813)
\psline[linewidth=0.03cm,arrowsize=0.05291667cm 2.0,arrowlength=1.4,arrowinset=0.4]{->}(8.140938,-1.04078124)(9.060938,-0.21921875)

\psline[ArrowInside=->, ArrowInsidePos=0.5,arrowsize=0.25291667cm,linewidth=0.04cm](9.760938,1.7792188)(8.360937,1.7792188)
\psline[ArrowInside=->, ArrowInsidePos=0.5,arrowsize=0.25291667cm,linewidth=0.04cm](11.160937,1.7792188)(9.760938,1.7792188)
\psline[ArrowInside=->, ArrowInsidePos=0.5,arrowsize=0.25291667cm,linewidth=0.04cm](12.560938,1.7792188)(11.160937,1.7792188)
\psdots[dotsize=0.12](9.760938,1.7792188)
\psdots[dotsize=0.12](11.160937,1.7792188)
\psdots[dotsize=0.12](12.560938,1.7792188)
\psdots[dotsize=0.12](8.360937,1.7792188)
\psline[ArrowInside=->, ArrowInsidePos=0.5,arrowsize=0.25291667cm,linewidth=0.04cm](11.160937,1.7792188)(9.760938,2.5792189)
\psdots[dotsize=0.12](9.760938,2.5792189)
\psline[ArrowInside=-<, ArrowInsidePos=0.5,arrowsize=0.25291667cm,linewidth=0.04cm,linestyle=dashed,dash=0.16cm 0.16cm](10.160937,0.17921875)(10.9609375,0.77921873)
\psline[ArrowInside=->, ArrowInsidePos=0.5,arrowsize=0.25291667cm,linewidth=0.04cm,linestyle=dashed,dash=0.16cm 0.16cm](8.9609375,1.3792187)(8.360937,0.77921873)
\psline[ArrowInside=-<, ArrowInsidePos=0.5,arrowsize=0.25291667cm,linewidth=0.04cm,linestyle=dashed,dash=0.16cm 0.16cm](10.160937,0.17921875)(11.360937,-0.22078125)
\psline[ArrowInside=->, ArrowInsidePos=0.5,arrowsize=0.25291667cm,linewidth=0.04cm](12.360937,-0.62078124)(10.9609375,-0.62078124)
\psline[ArrowInside=-<, ArrowInsidePos=0.5,arrowsize=0.25291667cm,linewidth=0.04cm,linestyle=dashed,dash=0.16cm 0.16cm](9.560938,-0.62078124)(10.9609375,-0.62078124)
\psdots[dotsize=0.12](10.9609375,-0.62078124)
\psdots[dotsize=0.12](11.360937,-0.22078125)
\psdots[dotsize=0.12](10.160937,0.17921875)
\psdots[dotsize=0.12](8.360937,0.77921873)
\psdots[dotsize=0.12](10.9609375,3.1792188)
\usefont{T1}{ptm}{m}{n}
\rput(9.062344,1.9242188){$f$}
\usefont{T1}{ptm}{m}{n}
\rput(10.502344,1.5842187){$f'$}
\usefont{T1}{ptm}{m}{n}
\rput(12.022344,1.9242188){$f''$}
\usefont{T1}{ptm}{m}{n}
\rput(10.702344,2.3442187){$f'''$}
\usefont{T1}{ptm}{m}{n}
\rput(10.402344,0.76421875){$g$}
\usefont{T1}{ptm}{m}{n}
\rput(9.952344,-0.05578125){$w_1$}
\usefont{T1}{ptm}{m}{n}
\rput(11.342343,1.5642188){$a$}
\usefont{T1}{ptm}{m}{n}
\rput(11.272344,3.2242188){$b$}
\usefont{T1}{ptm}{m}{n}
\rput(8.362344,0.60421876){$c$}
\usefont{T1}{ptm}{m}{n}
\rput(9.852344,0.34421876){$d$}
\usefont{T1}{ptm}{m}{n}
\rput(9.452344,-0.90578126){$(d)$}
\psdots[dotsize=0.12](9.520938,-0.6407812)
\psline[ArrowInside=->, ArrowInsidePos=0.5,arrowsize=0.25291667cm,linewidth=0.04cm](11.160937,1.7792188)(10.9609375,0.77921873)
\psline[ArrowInside=->, ArrowInsidePos=0.5,arrowsize=0.25291667cm,linewidth=0.04cm](9.760938,1.7792188)(8.9609375,1.3792187)
\psdots[dotsize=0.12](8.9609375,1.3792187)
\psdots[dotsize=0.12](10.9609375,0.77921873)
\psbezier[linewidth=0.04,linestyle=dashed,dash=0.16cm 0.16cm](10.9609375,3.1792188)(9.760938,3.1792188)(9.760938,3.1792188)(9.560938,2.9792187)(9.360937,2.7792187)(9.160937,2.3792188)(9.760938,1.7792188)
\psline[ArrowInside=-<, ArrowInsidePos=0.5,arrowsize=0.25291667cm, linewidth=0.04,linestyle=dashed,dash=0.16cm 0.16cm](9.560938,2.9792187)(9.360937,2.6792187)
\psline[ArrowInside=->, ArrowInsidePos=0.5,arrowsize=0.25291667cm,linewidth=0.04cm](12.360937,-0.62078124)(11.360937,-0.22078125)
\psdots[dotsize=0.12](12.360937,-0.62078124)
\end{pspicture} 
}
\end{center}
\caption{The dashed edges are in the orbit $\gr g$ and the plain ones in the orbit $\gr f$.}
\label{casstrictasc}
\end{figure}
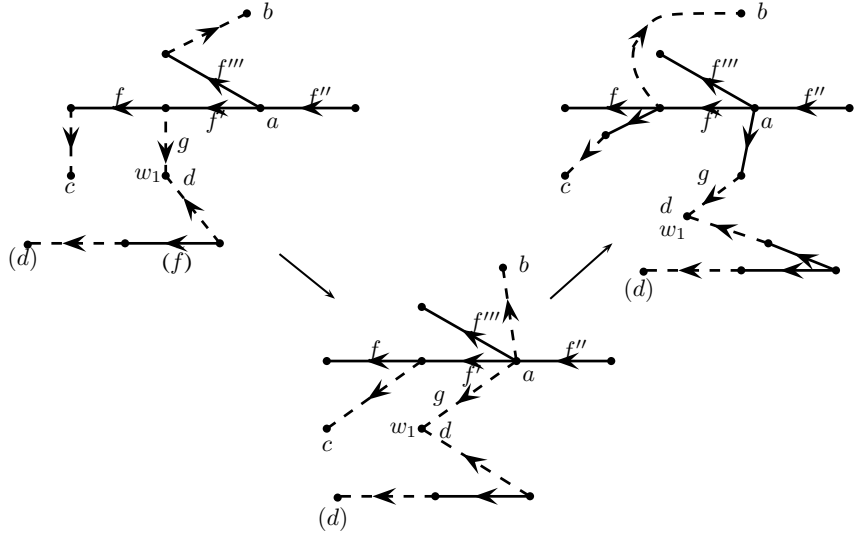  
\end{proof}

\paragraph{Deformation lemmas}

\begin{lemma}\label{suiteglissement}
If $f$ is slippery or \psad, up to exchanging the role of $f$ and $\bar f$, there exists an admissible sequence of moves $$S={g}_1/{h}_1\dots { g}_p/{ h}_p\cdot f/{ i}_1\dots f/{ i}_q$$  which preserves $f$ with ${\gr h}_j,{\gr  i}_j, {\gr g}_j\neq {\gr f},\bar {\gr f}$ for all $j$ such that, after performing the sequence, the edge $F$ is \pa or there exists an admissible slide along $\bar f$.
\end{lemma}

\begin{proof}
As $f$ is slippery or \psad, there exists an admissible sequence of moves $S'=\lambda_1/\mu_1\cdot\dots \cdot\lambda_r/\mu_r)$ such that after performing the sequence either $f$ is \pa (or \pde) or there exists an admissible slide along $f$ or $\bar f$. We may chose $S'$ of minimal length.

Suppose there exists $i$ such that  $\mu_i\in\gr f, \bar {\gr f}$, then after the $i-1$ first moves either $f$ or $\bar f$ is  \pde (if $\lambda_i=f$ or $\bar f$) or there exists an edge $g$ which slides along $f$ (take $g=\lambda_i$ if $\lambda \neq f, \bar f$).
As we chose $S'$ of minimal length, this is impossible, thus $\mu_i\not\in\gr f, \bar {\gr f}$.

Up to replacing $f$ by $\bar f$, we may assume that, in $S'\cdot T$, the edge $f$ is \pa or that the admissible slide is along $\bar f$.

We will first show that $S'$ may be changed into the product of two sequences of move $S_1$ and $S_2$ such that no edge of $\gr f$ or $\bar {\gr f}$ occurs in $S_1$, and $S_2$ is a product of slides of $f$ and $\bar f$. Then it will remain to remove the slides of $\bar f$ in $S_2$.

If $S$ is not already of this form there exists an $i$ such that $\lambda_i= f$ or $\bar {f}$ and $\lambda_{i+1}\neq{\gr f}, \bar {\gr f}$.
We only consider the case $\lambda_i= f$, the following formulas remain true replacing ${f}$ by $\bar f$.
We replace the sequence $f/\mu_i \cdot \lambda_{i+1}/ \mu_{i+1}$ by a sequence of $2$, $3$ or $4$ terms in which $f$ appears in the last terms.

If $\lambda_{i+1}$ and $\mu_{i+1}$ are not in the same orbit (it is not an induction or a $\mathcal A^{\pm 1}$-move), there are tree possibilities:
\begin{itemize}
\item either $\lambda_{i+1}$ is not in the orbit of $\mu_i$ or $\bar \mu_i$, then the two slides commute,%$f/\mu_i \cdot \lambda_{i+1}/ \mu_{i+1}= \lambda_{i+1}/ \mu_{i+1}\cdot f/\mu_i$,
\item or $\lambda_{i+1}$ and $\mu_i$ are in the same orbit (we may assume $\lambda_{i+1}=\mu_i$), then $G_f\subset G_{\mu_{i+1}}$ and we have the equality $f/\mu_i \cdot \lambda_{i+1}/ \mu_{i+1}=\lambda_{i+1}/ \mu_{i+1} \cdot f/\mu_i \cdot f/ \bar \mu_{i+1}$,
\item or $\lambda_{i+1}$ and $\bar \mu_i$ are in the same orbit (we may assume $\lambda_{i+1}=\bar \mu_i$), then $G_f\subset G_{\mu_{i+1}}$ and we have the equality $f/\mu_i \cdot \lambda_{i+1}/ \mu_{i+1}=\lambda_{i+1}/ \mu_{i+1}\cdot f/\mu_{i+1} \cdot f/\mu_i$.
\end{itemize}

If $\lambda_{i+1}$ and $\mu_{i+1}$ are in the same orbit then the move is an induction or an $\mathcal A^{\pm1}$-move.

If $\lambda_{i+1}/\mu_{i+1}$ is an induction then the two moves $f/\mu_i $ and $\lambda_{i+1}/\mu_{i+1}$ commute.

If $\lambda_{i+1}/ \mu_{i+1}$ is an $\mathcal A^{\pm1}$-move we have three cases to consider:
\begin{itemize}[topsep=0cm]
\item If $\lambda_{i+1}$ is not in the orbit of  $\mu_i$ or $\bar \mu_i$ and $\mu_i$ does not vanish in the move $\lambda_{i+1}/ \mu_{i+1}$ then the two moves commute.
\item If $\lambda_{i+1}$ is not in the orbit of  $\mu_i$ or $\bar \mu_i$ but $\mu_i$ vanishes in the move $\lambda_{i+1}/ \mu_{i+1}$ then $f/\mu_i \cdot \lambda_{i+1}/ \mu_{i+1}=\lambda_{i+1}/ \mu_{i+1}$.

\item If $\lambda_{i+1}$ is in the orbit of $\mu_i$ or $ \bar {\mu_i}$.

 If $\lambda_{i+1}/ \mu_{i+1}$ is an $\mathcal A^{-1}$-move, then the two moves commute. 
 Else it is an $\mathcal A$-move, then calling $h$ the new edge we have $f/\mu_i \cdot \lambda_{i+1}/ \mu_{i+1}=\lambda_{i+1}/ \mu_{i+1}\cdot f/h\cdot f/\mu_i \cdot f/ \bar h$.
\end{itemize}
Using these equalities, we may put all slides of $f$ and $\bar f$ at the end of the sequence. 

Let us prove that the process of transformation finishes.

  First note that the number of moves not involving $f$ stays constant. Call $n$ this number. Let $\mathcal S$ be the set of move sequences with exactly $n$ terms not involving $f$ or $\bar f$. Let $\varphi: \mathcal S \rightarrow \N^n$ be defined as follows: for $S$ a sequence $\varphi(S)=(k_1,\dots, k_n)$ where $k_i$ counts the number of terms involving $f$ or $\bar f$ appearing before the $i$th term not involving $f$ or $\bar f$ in $S$. Note that $\varphi(S)=(0,\dots 0)$ if and only if $S=S_1S_2$ with $S_1$ not involving $f$ or $\bar f$ and $S_2$ composed of slides of $f$ or $\bar f$.

Then $\varphi$ decreases strictly for the lexicographic order on $\N^n$ when we apply one of the transformation described previously. Thus the process finishes.

Now, it remains to prove that we may exchange $S_2$ by a sequence in which only slides of $f$ occurs. As a slide of $f$ and a slide of $\bar f$ commute, we may assume that $S_2$ is a product of $S_3$ and $S_4$, with $S_3$ a sequence of slides of $f$ and $S_4$ a sequence of slides of $\bar f$.

After performing the sequence $S'$, either the edge $f$ is \pa or there exists an edge $g$ which slides along $\bar f$.

First assume that after performing the sequence of moves $S'$, the slide $g/\bar f$ is admissible, and that $f$ is not \pa. If a sequence $\bar f/h \cdot g /\bar f$ is admissible in a tree $T$, this implies that $G_g\subset G_f$ and that the terminal vertices of $f$ and $g$ are the same in $T$. Thus $g /\bar f$ is also admissible in $T$ (the admissibility comes from the fact that $f$ is \pa: $g$ is still reduced after the slide).
Thus sequence $S_1\cdot S_3\cdot g /\bar f$ is admissible in the initial tree, and $S=S_1\cdot S_3$ is as required.

Now assume that after performing the sequence of move $S$, the edge $f$ is \pa. 

We now use the following simple fact.
\begin{fait}\label{inter}
Let $T$ be a $G$-tree. Let $f$ and $g\not \in \gr f, \bar {\gr f}$ be two edges of $T$. If the move $\bar f/g$ is admissible in $T$ and $f$ is \pa in $\bar f/g\cdot T$, then there exists $g'\in \gr g$ such that $f/\bar g'$ is admissible in $T$ and $f$ is \pa in $f/\bar g'\cdot T$
\end{fait}

\begin{proof}[Proof of the fact]
Call $v$ and $v'$ the initial and terminal vertices of $g$. As $f$ is \pa in $\bar f / g\cdot T$, there exists $f'\in \gr f$ with terminal vertex $v'$, such that $G_{f'}\subset G_{f} (\subset G_g)$. As $g\not \in \gr f, \bar {\gr f}$, the slide $f'/\bar g$ is admissible in $T$ and $f$ is \pa in $f'/\bar g\cdot T$.
Let $a$ be an element such that $f=a\cdot f'$. Then the edge $g'=a\cdot g$ is as required (the moves $f'/\bar g$ and  $f/\bar g'$ are equal).
\end{proof}

Call $T$ the tree obtained by performing $S_1$ and $S_3$ on the initial tree. Assume $S_4$ is the product $\prod_{i=1}^{p} \bar f/ k_i$. 
We prove the lemma by recurrence. If $p=0$ then the sequence $S_1\cdot S_3$ is as required. 

If $p>0$, applying Fact \ref{inter}, there exists $k'_p\in \gr k_p$ such that we may change $S_4$ by $(\prod_{i=1}^{p-1} \bar f/ k_i)\cdot f /\bar k'_p$. 
Moreover the move $f /\bar k'_p$ commutes with every move $( \bar f/ k_i)$. 
Thus $S_4$ may be changed by $f /\bar k_p\cdot (\prod_{i=1}^{p-1} \bar f/ k_i)$.
We may conclude by applying the recurrence hypothesis.
\end{proof}

\begin{lemma} \label{nonslipvanish}
Let $T$ be a $G$-tree.  Let $f$ be a  vanishing non-slippery edge of $T$. 

There exists a strictly ascending orbit of edges $\gr e$ of $T$ whose endpoint $\gr w$ is an endpoint of $\gr f$ and of valence $3$, and there exists an admissible sequence $S$ of inductions on $\gr e$ and slides of $\gr f$ along $\gr e$ or $\bar {\gr e}$ finishing by a $\mathcal A^{-1}$-move on $\gr e$ in which $\gr f$ and $\gr w$ vanish. 
\end{lemma}

\begin{proof}
As $f$ vanishes, there exists an admissible sequence $S$ such that there exists an admissible $\mathcal A^{-1}$-move on $T'=S\cdot T$ in which $\gr f$ vanishes. Call $w$ the vanishing vertex of this move. As $f$ is non-slippery, there are exactly three orbits $\gr f$, $\gr e$, $\bar {\gr e}$ adjacent to $w$ (by definition every other edge should slide along $\bar f$), and the sequence $S$ does not involve slides along $\gr f$ or $\bar {\gr f}$.

We prove the lemma by recurrence on the length $n$ of the sequence $S$.
 If $n=0$ then $v=w$, and the lemma is trivial.
 
 Assume $n\geq 1$, let $\gr g /\gr g'$ the first move of $S$, and $S'$ the sequence of the $n-1$ last moves of $S$. Call $T''=\gr g /\gr g' \cdot T$. By recurrence hypothesis, we may replace $S'$ by a sequence $S''$ of induction on $\gr e$ and slides along $\gr e$ and $\bar {\gr e}$.
  As $f$ is not slippery we have $\gr g'\neq \gr f, \bar {\gr f}$.
 If $\gr g$ and $\gr g'$ differ from $\gr f$, $\gr e$ and $\bar {\gr e}$, then we may remove the slide from $S$, without changing anything around $\gr v$, thus the sequence $S''$ is admissible in $T$ and the lemma is proven.
If $\gr g=\gr e$ (or $\bar {\gr e}$), as the only edges adjacent to $\gr e$ are in $\gr f$ and $\bar {\gr e}$, and that $\gr f$ is non-slippery, then $\gr g/\gr g'$ is an induction on $e$, thus the sequence $\gr g/\gr g'\cdot S''$ is a sequence of inductions on $e$ and slides of $\gr f$ along $\gr e$ and $\bar {\gr e}$.
Finally, if $\gr g=\gr f$, as $\gr e$, $\bar {\gr e}$ and $\gr f$ are the only three orbits adjacent to $\gr w$, we have $\gr g'=\gr e$ or $\bar {\gr e}$. Thus $\gr g/\gr g'\cdot S''$ is as required. That finishes the proof on the lemma.
\end{proof}

\begin{lemma}\label{lemmeascendant}
Let $T$ be a reduced $G$-tree, $e$ an edge of $T$ and $\mathbf m$ an admissible move. If $e$ is \tor in $T$ but not in $\mathbf m\cdot T$, then there exists a strictly ascending edge $f$ adjacent to $e$ in $T$ and $e$ is \psad.
\end{lemma}

\begin{proof}
Note that a slide of a \tor edge in a reduced tree must be along an ascending edge. As $e$ is not \tor in $\mathbf{m}\cdot T$, if $\mathbf{m}$ is a slide, it is a slide of $e$ along a strictly ascending edge $f$, and $e$ is strictly ascending in $\mathbf{m}\cdot T$. In this case the lemma is proved.

As $e$ is \tor in $T$, the move $\mathbf m$ is not an induction or an $\mathcal A^{\pm1}$-move on $e$.
The move $\mathbf m$ may not be an $\mathcal A$-move on any edge $f$, since an $\mathcal A$-move leaves unchanged the set of \tor edges.%does not change the index of edge groups into vertex groups except for the one of $\gr f$ (an the vanishing edge).

If $\mathbf m$ is an induction or an $\mathcal A^{-1}$-move on an orbit of edges $\gr f$, then $\gr f$ is strictly ascending in $T$ and adjacent to $e$. Thus sliding $e$ along $f$ in $T$, it become strictly ascending.
\end{proof}

\begin{lemma}\label{lemme2slip}
Let $T$ be a reduced $G$-tree, and $e$ be a $2$-slippery \tor edge. Then there exists an admissible sequence $S$ of moves which does not involve $\gr e$ and two edges $f$ and $g$ of distinct orbits in $S\cdot T$ such that the slides $f/e$ and $g/e$ are both admissible in $S\cdot T$.
\end{lemma}

\begin{proof}
Note that if $e$ is \tor in a tree $T'$, then there exists two admissible slides on $e$ if and only if the initial vertex of $e$ is of valence at least four. 

By definition of being  $2$-slippery there exists an admissible sequence $S$ of moves and two edges $f$ and $g$ of distinct orbits in $S\cdot T$ such that $f/e$ and $g/e$ are both admissible in $S\cdot T$. But a move involving $e$ may occur in $S$. We may assume that for every initial subsequence $S_1\neq S$ of $S$, the initial vertex of $e$ in $S_1\cdot T$ is of valence $3$.
Hence $\gr e$ has no adjacent strictly ascending edges, thus by Lemma \ref{lemmeascendant}, the edge $e$ is \tor in every intermediate tree.

Recall that no induction and no $\mathcal A^{\pm1}$-move on a \tor edge are admissible, and note that if a \tor edge slides along an edge $f$ then $f$ is strictly ascending. Hence, no induction, no $\mathcal A^{\pm1}$-moves on $e$ and no slide of $e$ appears in $S$.

If a slide on $e$ occurs in $S$, as the initial vertex of $e$ is of valence $3$, by \cite[Theorem 1, case 3]{Levitt05} this slide may be remove from $S$.

Thus $S$ may be taken as required.
\end{proof}

\paragraph{Refinement lemmas}
\begin{lemma}\label{rafinementcontinu}
Let $T$ be a $G$-tree, let $\gr e$ be an orbit of edges of $T$. Let $T_e$ be the tree obtained from $T$ by collapsing every orbit of edges except $\gr e$. Let $S$ be a sequence of moves on $T$, in which no slide along $\gr e$ or $\bar {\gr e}$ , no induction and no $\mathcal A^{\pm 1}$-move on $\gr e$ occurs, and such that $\gr e$ does not vanish (all moves involving $\gr e$ except slides of $\gr e$ and $\bar {\gr e}$).
Then $S\cdot T$ refines $T_e$.
\end{lemma}

\begin{proof}
In \cite{ClayFor09}, Clay and Forester give a description of the Whitehead moves in terms of expansions and collapses. If a tree $\tilde T$ refines $T_e$ and if we perform an expansion on $\tilde T$, clearly the resulting tree also refines $T_e$. If we perform a collapse of an orbit distinct from $\gr e$ in $\tilde T$ it will still refines $T_e$ (since $T_e$ is obtained from $\tilde T$ by collapsing all orbits of edges except $\gr e$ in $T$). 

A move which is not a slide along $\gr e$ or $\bar {\gr e}$ , not an induction or an $\mathcal A^{\pm 1}$-move on $\gr e$ or in which $\gr e$ vanishes, may be seen as expansions and collapses with no collapse of $\gr e$ (see  \cite{ClayFor09}).
Thus $S\cdot T$ still refines $T_e$.
\end{proof}

\begin{lemma}\label{lift}
Let $T'$ be a $G$-tree. Let $f$ be an edge of $T'$ with initial vertex $w$ and $H$ a group such that $G_f\subset H \subset G_w$. Let $T$ be a refinement of $T'$ obtained by performing an expansion on $w$ with group $H$ and set $\{ f\}$. Call $e$ the new edge and $v$ its initial vertex. Let $T_e$ be the tree obtained from $T$ by collapsing all orbits of edges except $\gr e$.

Let $\mathbf{m}$ be a move of $T'$ which is not a slide of $\bar f$, a slide along $f$, an induction of $f$, an $\mathcal A^{\pm1}$-move of $f$ or in which $f$ vanishes.
Then we may perform $\mathbf m$ in $T$, and $\mathbf m \cdot T$ refines both $T_e$ and $\mathbf m\cdot T'$.

Moreover if $\mathbf{m}$ is not a slide along $\bar f$, then the vertex $v$ in $\mathbf{m}\cdot T$ is still of valence two. 
\end{lemma}

\begin{proof}
Call $v'$ the terminal vertex of $e$ in $T$.
Let $W$ be the subtree of $T$ which collapses onto $w$ in $T'$. As $G_e=G_v$, for each vertex $v''$ in the orbit $\gr v$ in $T$ there is only one edge in $\gr e$ with initial vertex $v''$. This implies that $W$ is of diameter $2$ and that $W\cap \gr v'=\{v'\}$ (and every other vertex of $W$ is in $\gr v$).

Given $g$ and $h$ two edges with a common initial vertex in $T'$, from the previous remark, if $g$ and $h$ are not in the orbit $\gr f$ (but may be in $\bar {\gr f}$), their lifts in $T$ have a common initial vertex. Thus if the move $\mathbf m$ is an induction or an $\mathcal A^{\pm1}$-move on an orbit of edges $\gr g\neq \gr f$, the edge $\gr g$ is still \pa in $T$, and $\mathbf m$ may be performed in $T$. If $\mathbf m$ is a slide, then the two concerned orbits still have a common vertex in $T$ and the slide may be performed.

From Lemma \ref{rafinementcontinu}, the tree $\mathbf m \cdot T$ refines $T_e$.

As the edges of initial vertex $v$ in $T$ are either $e$ or in $\gr f$, if $\mathbf m$ is not a slide along $\bar f$, the vertex stays of valence $2$ in $\mathbf m\cdot T$.
\end{proof}

\begin{corollary}\label{liftsequence}
Let $T'$, $T$ and $T_e$ be as in Lemma \ref{lift}. Let $S$ be a sequence of moves on $T'$ in which no move is  a slide of $\bar f$, a slide along $f$ or $\bar f$, an induction of $f$, an $\mathcal A^{\pm1}$-move of $f$ or in which $f$ vanishes, then  $S$ may be performed in $T$ and $S\cdot T$ refines $T_e$ and $S\cdot T'$.
\end{corollary}

\begin{proof}
This is just a recurrence on the length of $S$, noticing that as in $S$ no slide along $\bar f$ occurs, for all initial subsequence $S_1$ of $S$, the valence of $v$ in $S_1\cdot T$ remains $2$, and we may apply Lemma \ref{lift}.
\end{proof}

\paragraph{Maximality}

\begin{lemma}\label{rafinementdeT}
 Let $T_e$ be a universally compatible $G$-tree with exactly one orbit of edges $\gr e$ (we choose for $\gr e$ an arbitrary orientation). 
There exists a JSJ tree $T$ which refines $T_e$, such that every edge collapsed in $T_e$ is reduced.
\end{lemma}

\begin{proof}
Let $T'$ be a reduced JSJ $G$-tree. As $T_e$ is universally compatible, there exists a common refinement tree between $T_e$ and $T'$. Call $T$ their least common refinement (for the existence of such a tree see \cite{Gl3b}). As  $T'$ is a JSJ tree, it dominates $T_e$ (which is universally elliptic), and thus $T$ is also a JSJ tree (see \cite{Gl3b}).

Two cases may appear. Either $\gr e$ is reduced in $T$, then $T'=T$, and $T$ is reduced, or $\gr e$ is not reduced, and then $T$ has one orbit of (unoriented) edges more than $T'$.

It remains to prove that we may assume that the only non reduced edges orbit of $T$ is $\gr e$.

If there exists another non reduced orbit $\gr f\neq \gr e, \bar{\gr e}$ in $T$, then collapsing $\gr f$ we obtain a JSJ tree which still refines $T_e$. We may thus collapse one by one non-reduced edge which are not in the orbit of  $\gr e$ or $\bar{\gr e}$ until every edge except maybe $\gr e$ is reduced.
\end{proof}

\begin{proposition}\label{reduced}
Let $G$ be a \vGBS group. Assume that $G$ is not an ascending HNN-extension. Let $T_e$ be a universally compatible $G$-tree with exactly one orbit of edges $\gr e$. If $T_e$ is refined by a reduced JSJ tree, then $T_{comp}$ dominates $T_e$
\end{proposition}

%\begin{proof}
\proof
Let $T$ be a reduced JSJ $G$-tree that refines $T_e$.
Let $\gr e$ be the orbit of edge of $T'$ not collapsed in $T_e$.

If $\gr e$ is not collapsed in $T_{comp}$, then $T_{comp}$ refines $T_e$ thus dominates it. If $\gr e$ is collapsed in $T_{comp}$, by construction of $T_{comp}$ the orbit $\gr e$ is non-ascending slippery, \psad, \tor $2$-slippery, or vanishing.

\begin{enumerate}%[leftmargin=0cm]
\item \label{casglissant} If $\gr e$ non-ascending and slippery or \psad in $T$.

By Lemma \ref{suiteglissement} there exists an admissible sequence of moves $S$ without slide along $\gr e$ or $\bar {\gr e}$ and no induction or $\mathcal A^{\pm 1}$-move of $\gr e$, at the end of which either $\gr e$ is \pa or \pde, or  there exists an admissible slide along $\gr e$ or $\bar {\gr e}$.

By Lemma \ref{rafinementcontinu}, the tree $S\cdot T$ refines $T_e$. By Lemma \ref{lemmeascendant}, we may assume that $e$ is not \tor. Then by Corollary \ref{casreduit} cases 1. and 3, the tree $T_e$ is not universally compatible. This is a contradiction.

 \item If $\gr e$ is $2$-slippery and \tor in $T$. 

By Lemmas \ref{lemme2slip} and \ref{rafinementcontinu}, we may assume that there exists two edges $f$ and $g$ in $T$ of distinct orbits with terminal vertices $v$ which may (directly) slide along $e$ . Then by Corollary \ref{casreduit} case 3, the tree $T_e$ is not universally compatible. This is a contradiction.

\item If $\gr e$ is vanishing non-slippery in $T$.

Take $e$ an edge in the orbit $\gr e$ and call $v$ the terminal vertex of $e$. By Lemma \ref{nonslipvanish}, there exists a strictly ascending edge $f$ with initial vertex $v$ such that $v$ has exactly $3$ orbits of adjacent edges $\gr e$, $\gr f$ and $\bar{\gr f}$. Moreover up to perform inductions on $\gr f$ and slides of $e$ along $\gr f$, we may assume that there exists an admissible $\mathcal A^{-1}$-move on $\gr f$ in $T$ in which $e$ vanishes. It implies that $G_f\subset G_e$, and $f$ slides along $e$ (this slide is not admissible, since $f$ is ascending). By Corollary \ref{casreduit} case 3. the tree $T_e$ is not universally compatible. This is a contradiction.

\item  If $\gr e$ is vanishing, slippery and \tor in $T$.

 %If $\gr e$ is not ascending or strictly ascending, by case \ref{casglissant}, the tree $T$ is not universally compatible. We may thus assume $\gr e$ is \tor in $T$.

As $\gr e$ is vanishing there exists a tree in the deformation space of $T$ in which $\gr e$ is not ascending. Then by Lemma \ref{lemmeascendant}, the orbit $\gr e$ is \psad, and the case is already done in Case \ref{casglissant}.\hfill\qed

\end{enumerate}
%\end{proof}

\begin{proposition}\label{maxcasnonreduit} Let $G$ be a \vGBS group.
Let $T_e$ be a universally compatible tree with one orbit of edges $\gr e$ (whose orientation is fixed arbitrarily). 
If there is a  JSJ tree $T$ which refines $T_e$, in which the unique non-reduced orbit of edges is the one not collapsed in $T_e$, then $T_{comp}$ dominates $T_e$.
 \end{proposition}
 
\proof
The orbit of edges of $T$ not collapsed in $T_e$ is also called $\gr e$.
Let $T'$ be the tree obtained by collapsing $\gr e$ in $T$. Then $T'$ is a reduced JSJ tree.

If $\gr e$ is obtained by blowing up a dead end vertex in $T'$, then $\gr e$ is not collapsed in $T_{comp}$, thus $T_{comp}$ refines (hence dominates) $T_e$ . We may  assume that $\gr e$ is not obtained by blowing up a dead end vertex in $T'$.

Fix $e$ a representative of $\gr e$,
and let $v$ and $v'$ be the initial and terminal vertices of $e$. Since $e$ is not reduced, the orbits of $v$ and $v'$ are distinct. We assume $G_v=G_e$. Then there exists $f$ not in the orbit $\gr e$ with initial vertex $v$ (otherwise the tree is not minimal). 

\begin{lemma}\label{casnonecrase}
If none of the edges with initial vertex $v$ except those in $\gr e$ is collapsed in $T_{comp}$ then $T_{comp}$ dominates $T_e$.
\end{lemma}
\begin{proof}[Proof of Lemma \ref{casnonecrase}]
Given a hyperbolic $h$ in $T_e$, its axis contains an edge in the orbit $\gr e$ (we may assume it is $e$), thus in $T$ its axis contains $v$. As $G_e=G_v$ (and $e$ is not reduced), the edge $e$ is the only edge of $\gr e$ with initial or terminal vertex $v$ and thus the axis of $h$ also contains an edge which is not collapsed in $T_{comp}$. Then $h$ is also hyperbolic in $T_{comp}$. Then all hyperbolic elements of $T_e$ are hyperbolic in $T_{comp}$, as $G$ is finitely generated this implies $T_{comp}$ dominates $T_e$.
\end{proof}

We may thus assume there always exists an edge of $T$ distinct from $e$ with initial vertex $v$ which is collapsed in $T_{comp}$.

\begin{enumerate}%[leftmargin=0.5cm,label=\Alph*.]
\item We first assume that $G_{v'}\neq G_e$. 

This hypothesis will not be used in the cases \ref{cas1-2strictasc}, \ref{cas1-2strictdes} and \ref{vannonslip}. 
\begin{enumerate}%[leftmargin=0.5cm,label=\alph*.]
\item \label{troisaretes} Assume that $v$ has at least three orbits of adjacent edges $\gr e$, $\gr f$, and $\gr g$.

\begin{enumerate}%[leftmargin=0.2cm,label=\roman*.]
      \item \label{strictlyasc} Assume that $\gr f$ (or $\gr g$) is strictly ascending/descending in $T$.
     Up to taking $\bar {\gr f}$, we may assume that $\gr f$ is ascending.
     Then $\bar e$ may slide along $\gr f$ and becomes a reduced edge. We obtain a reduced JSJ tree which refines $T_e$. This case is done in Proposition \ref{reduced}.
     
 \item Assume that $\gr f$ and ${\gr g}$ are not ascending/descending in $T$.    
     
     Let $f$ and $g$ be representatives of $\gr f$ and $\gr g$ with initial vertices $v$.
     
     We may take $a$, $b$, $c$ and $d$ as described in Figure \ref{cas13}: call $w$ and $w'$ the terminal vertices of $f$ and $g$. Take $a$ in $G_w\setminus G_f$, $b$ in $G_{w'}\setminus G_g$, and $c$ in $G_{v'}\setminus G_e$ that centralizes $G_f$ (it exists by Lemma \ref{normalisation}). Call $e'=c\cdot e$, $g'=c\cdot g$ and $w''=c\cdot w'$. Take $d$ in $G_{w''}\setminus G_{g'}$. The sets  $\mathcal E_{\{a,b\}}(T)$ and $\mathcal E_{\{c,d\}}(T)$ are separated by the edge $e$. In $T_e$ every edge is collapsed except the orbit $\gr e$, thus $\mathcal E_{\{a,b\}}(T_e)\cap \mathcal E_{\{c,d\}}(T_e)=\emptyset$. 

As $c$ centralizes $G_f$, the group $G_f$ is contained in $G_{e'}$, the edge $\bar f$ may slide along $e$ and $\bar e'$. We obtain a new tree $T''$ in which $\mathcal E_{\{b,c\}}(T'')\cap \mathcal E_{\{a,d\}}(T'')=\emptyset$. By Lemma \ref{incompatibility}, the tree $T_e$ is not universally compatible.

\begin{figure}[!ht]
\begin{center}
% Generated with LaTeXDraw 2.0.5
% Thu Jan 27 14:29:36 CET 2011
% \usepackage[usenames,dvipsnames]{pstricks}
% \usepackage{epsfig}
% \usepackage{pst-grad} % For gradients
% \usepackage{pst-plot} % For axes
\scalebox{1} % Change this value to rescale the drawing.
{
\begin{pspicture}(0,-1.3929688)(7.3228126,0.529688)
\psline[ArrowInside=->, ArrowInsidePos=0.5,arrowsize=0.25291667cm,linewidth=0.04cm](1.7609375,-0.22546875)(0.5609375,0.37453124)
\psline[ArrowInside=->, ArrowInsidePos=0.5,arrowsize=0.25291667cm,linewidth=0.04cm](1.7609375,-0.22546875)(3.3609376,-0.22546875)
\psline[ArrowInside=->, ArrowInsidePos=0.5,arrowsize=0.25291667cm,linewidth=0.04cm](1.7609375,-0.22546875)(0.5609375,-1.0254687)
\psline[ArrowInside=->, ArrowInsidePos=0.5,arrowsize=0.25291667cm,linewidth=0.04cm](4.9609375,-0.22546875)(3.3609376,-0.22546875)
\psline[ArrowInside=->, ArrowInsidePos=0.5,arrowsize=0.25291667cm,linewidth=0.04cm](4.9609375,-0.22546875)(6.5609374,-1.0254687)
\psdots[dotsize=0.12](6.5609374,-1.0254687)
\psdots[dotsize=0.12](3.3609376,-0.22546875)
\psdots[dotsize=0.12](4.9609375,-0.22546875)
\psdots[dotsize=0.12](1.7609375,-0.22546875)
\psdots[dotsize=0.12](0.5609375,-1.0254687)
\psdots[dotsize=0.12](0.5609375,0.37453124)
\usefont{T1}{ptm}{m}{n}
\rput(2.4623437,0.07953125){$e$}
\rput(4.282344,0.07046875){$e'$}
\rput(0.23234375,0.43453124){$a$}
\rput(0.88234377,-0.12046875){$f$}
\rput(1.3623438,-0.8204687){$g$}
\rput(1.8523438,-0.02046875){$v$}
\rput(3.3923438,0.05046875){$v'$}
\rput(0.23234375,-1.2204688){$b$}
\rput(3.3623438,-0.5204688){$c$}
\rput(6.0423436,-0.5204688){$g'$}
\rput(7.012344,-1.1204687){$d$}
\end{pspicture} 
}

\end{center}
\caption{}
\label{cas13}
\end{figure}

\item \label{castoric} Assume that $\gr f$ is \tor in $T$. 

We may assume $\gr g=\bar {\gr f}$. We proceed as in Figure \ref{figurecastoric}: by Lemma \ref{normalisation}, we may take $a$ in $G_{v'}\setminus G_e$ that centralizes $G_f$, and $e'=a\cdot e$ (note that $G_e=G_{e'}$ thus $\bar f$ may slide consecutively along $e$ and $\bar e'$). Call $v''$ the initial vertex of $e'$, call  $f'=a\cdot f$ and $g'=a\cdot g$ the edges in the orbit of $f$ and $g$ with initial vertex $v''$, and $e''=a\cdot e'$ such that when $\bar f$ slides along $e$ and $\bar e'$, then $f'$ slide along $e'$ and $\bar e''$ (note that if $G_e$ is of index $2$ in $G_{v'}$, we have $e''=e$). Call $w$ the terminal vertex of $f$, $w'$ the terminal vertex of $f'$ and $w''$ the terminal vertex of $g'$. Let  $t$ be such that $t\cdot v= w$, and $t'$ such that $t'\cdot v'=w'$.
Define $b=tat^{-1}$,  $c=t'^{-1}at'$ and $d=t'at'^{-1}$.

Then in $T$, the sets  $\mathcal E_{\{a,b\}}(T)$ and $\mathcal E_{\{c,d\}}(T)$ are separated by the edge $e$. In $T_e$ every edge is collapsed except the orbit $\gr e$, thus $\mathcal E_{\{a,b\}}(T_e)\cap \mathcal E_{\{c,d\}}(T_e)=\emptyset$. 
     
As $a$ centralizes $G_f$, we have $G_f \subset G_{e'}$.
Making slide $\bar f$ along $e$ and $\bar e'$, simultaneously $\bar f'$ slides along $e'$ and $\bar e''$ and $g'$ slides along $t^{-1}\cdot e'$ and $t^{-1}\cdot \bar e$. We obtain $\mathcal E_{\{b,c\}}\cap \mathcal E_{\{a,d\}}=\emptyset$. By Lemma \ref{incompatibility}, the tree $T_e$ is not universally compatible.

\begin{figure}[!ht]
\begin{center}
% Generated with LaTeXDraw 2.0.5
% Sun Feb 06 15:03:11 CET 2011
% \usepackage[usenames,dvipsnames]{pstricks}
% \usepackage{epsfig}
% \usepackage{pst-grad} % For gradients
% \usepackage{pst-plot} % For axes
\scalebox{1} % Change this value to rescale the drawing.
{
\begin{pspicture}(0,-1.8629688)(11.722813,1.8629688)
\psline[ArrowInside=->, ArrowInsidePos=0.5,arrowsize=0.25291667cm,linewidth=0.04cm](0.4809375,0.48453125)(0.4809375,-0.71546876)
\psline[ArrowInside=->, ArrowInsidePos=0.5,arrowsize=0.25291667cm,linewidth=0.04cm,linestyle=dashed,dash=0.16cm 0.16cm](0.4809375,-0.71546876)(1.6809375,-0.71546876)
\psline[ArrowInside=->, ArrowInsidePos=0.5,arrowsize=0.25291667cm,linewidth=0.04cm,linestyle=dashed,dash=0.16cm 0.16cm](0.4809375,0.48453125)(1.6809375,0.48453125)
\psdots[dotsize=0.12](1.6809375,0.48453125)
\psdots[dotsize=0.12](0.4809375,0.48453125)
\psdots[dotsize=0.12](0.4809375,-0.71546876)
\psdots[dotsize=0.12](1.6809375,-0.71546876)
\psline[ArrowInside=->, ArrowInsidePos=0.5,arrowsize=0.25291667cm,linewidth=0.04cm,linestyle=dashed,dash=0.16cm 0.16cm](3.0809374,0.48453125)(1.6809375,0.48453125)
\psline[ArrowInside=->, ArrowInsidePos=0.5,arrowsize=0.25291667cm,linewidth=0.04cm](3.0809374,0.48453125)(3.0809374,-0.71546876)
\psline[ArrowInside=-<, ArrowInsidePos=0.5,arrowsize=0.25291667cm,linewidth=0.04cm](3.0809374,1.4845313)(3.0809374,0.48453125)
\psline[ArrowInside=->, ArrowInsidePos=0.5,arrowsize=0.25291667cm,linewidth=0.04cm,linestyle=dashed,dash=0.16cm 0.16cm](3.0809374,1.4845313)(2.0809374,1.4845313)
\psdots[dotsize=0.12](2.0809374,1.4845313)
\psdots[dotsize=0.12](3.0809374,0.48453125)
\psdots[dotsize=0.12](3.0809374,1.4845313)
\psdots[dotsize=0.12](3.0809374,-0.71546876)
\psline[ArrowInside=->, ArrowInsidePos=0.5,arrowsize=0.25291667cm,linewidth=0.04cm,linestyle=dashed,dash=0.16cm 0.16cm](3.0809374,-0.71546876)(2.0809374,-1.5154687)
\psdots[dotsize=0.12](2.0809374,-1.5154687)
\usefont{T1}{ptm}{m}{n}
\rput(1.7023437,0.7295312){$a$}
\usefont{T1}{ptm}{m}{n}
\rput(1.0623437,0.62953126){$e$}
\usefont{T1}{ptm}{m}{n}
\rput(0.20234375,-0.11046875){$f$}
\usefont{T1}{ptm}{m}{n}
\rput(1.7523438,-0.5104687){$b$}
\usefont{T1}{ptm}{m}{n}
\rput(3.3823438,-0.19046874){$f'$}
\usefont{T1}{ptm}{m}{n}
\rput(3.3823438,0.96953124){$g'$}
\usefont{T1}{ptm}{m}{n}
\rput(1.9723438,-1.6904688){$d$}
\usefont{T1}{ptm}{m}{n}
\rput(1.7423438,1.5295312){$c$}
\usefont{T1}{ptm}{m}{n}
\rput(2.3623438,0.6695312){$e'$}
\psline[linewidth=0.03cm,arrowsize=0.05291667cm 2.0,arrowlength=1.4,arrowinset=0.4]{->}(4.6209373,0.50453126)(6.7009373,0.50453126)
\psline[ArrowInside=->, ArrowInsidePos=0.5,arrowsize=0.25291667cm,linewidth=0.04cm](7.9609375,0.88453126)(7.9609375,-0.31546876)
\psline[ArrowInside=->, ArrowInsidePos=0.5,arrowsize=0.25291667cm,linewidth=0.04cm,linestyle=dashed,dash=0.16cm 0.16cm](11.180938,-0.6954687)(9.980938,-0.6954687)
\psline[ArrowInside=->, ArrowInsidePos=0.5,arrowsize=0.25291667cm,linewidth=0.04cm,linestyle=dashed,dash=0.16cm 0.16cm](7.9609375,0.88453126)(9.880938,0.48453125)
\psdots[dotsize=0.12](9.780937,0.50453126)
\psdots[dotsize=0.12](7.9609375,0.88453126)
\psdots[dotsize=0.12](7.9609375,-0.31546876)
\psdots[dotsize=0.12](9.980938,-0.6954687)
\psline[ArrowInside=->, ArrowInsidePos=0.5,arrowsize=0.25291667cm,linewidth=0.04cm,linestyle=dashed,dash=0.16cm 0.16cm](11.180938,0.50453126)(9.780937,0.50453126)
\psline[ArrowInside=->, ArrowInsidePos=0.5,arrowsize=0.25291667cm,linewidth=0.04cm](11.180938,0.50453126)(11.180938,-0.6954687)
\psline[ArrowInside=-<, ArrowInsidePos=0.5,arrowsize=0.25291667cm,linewidth=0.04cm](9.280937,1.4845313)(11.180938,0.50453126)
\psline[ArrowInside=->, ArrowInsidePos=0.5,arrowsize=0.25291667cm,linewidth=0.04cm,linestyle=dashed,dash=0.16cm 0.16cm](11.220938,1.5045313)(10.220938,1.5045313)
\psdots[dotsize=0.12](10.180938,1.5045313)
\psdots[dotsize=0.12](11.180938,0.50453126)
\psdots[dotsize=0.12](11.200937,1.5045313)
\psdots[dotsize=0.12](11.180938,-0.6954687)
\psline[ArrowInside=->, ArrowInsidePos=0.5,arrowsize=0.25291667cm,linewidth=0.04cm,linestyle=dashed,dash=0.16cm 0.16cm](7.9609375,-0.31546876)(8.9609375,-1.1154687)
\psdots[dotsize=0.12](8.9609375,-1.1154687)
\usefont{T1}{ptm}{m}{n}
\rput(9.802343,0.74953127){$a$}
\usefont{T1}{ptm}{m}{n}
\rput(8.622344,1.0295312){$e''$}
\usefont{T1}{ptm}{m}{n}
\rput(7.722344,0.28953126){$f'$}
\usefont{T1}{ptm}{m}{n}
\rput(9.871407,-0.49046874){$b$}
\usefont{T1}{ptm}{m}{n}
\rput(11.442344,-0.17046875){$f$}
\usefont{T1}{ptm}{m}{n}
\rput(10.562344,1.1295313){$g'$}
\usefont{T1}{ptm}{m}{n}
\rput(9.031406,-1.2904687){$d$}
\usefont{T1}{ptm}{m}{n}
\rput(10.122344,1.6695312){$c$}
\usefont{T1}{ptm}{m}{n}
\rput(10.462344,0.28953126){$e'$}
\psline[ArrowInside=-<, ArrowInsidePos=0.5,arrowsize=0.25291667cm,linewidth=0.04cm,linestyle=dashed,dash=0.16cm 0.16cm](10.180938,1.5045313)(9.180938,1.5045313)
\psdots[dotsize=0.12](9.240937,1.5045313)
\end{pspicture} 
}
\end{center}
\caption{The plain lines are in the orbit $\gr f$, the dashed ones in the orbit $\gr e$}
\label{figurecastoric}
\end{figure}
\end{enumerate} 

\item \label{semireduit1}  \label{cas121}  Assume that $v$ has only two orbits of adjacent edges $\gr e$ and $\gr f$.

As $\gr f$ is the only orbit with initial orbit of vertices $\gr v$ distinct from $\gr e$, we may assume that $\gr f$ is collapsed in $T_{comp}$ (see Lemma \ref{casnonecrase}).
Fix $f$ a representative of $\gr f$ with initial vertex $v$.

As every edge of $T$ except the ones in the orbit $\gr e$ is reduced and the terminal vertex of $\gr f$ is not adjacent to $\gr v$, we have $G_f\subsetneq G_e$.

As $f$ is reduced in $T$,  the edge $f$ is not \tor in $T'$, thus it is slippery, \psad, or vanishing.
\begin{enumerate}
\item \label{semireduitslip} Assume that $f$ is slippery or \psad.

 From Lemma  \ref{suiteglissement}, there exists a sequence of moves $S$ admissible in  $T'$ such that either $f$ is \pa or \pde in $S\cdot T'$, or there exists an admissible slide $g/f$ or $g/\bar f$ in $S\cdot T'$. We thus have four cases to treat here.
 
 Moreover this sequence may be split in two subsequence $S_1$ and $S_2$, such that no move involving $f$ occurs in $S_1$, and $S_2$ is composed either of slides of $f$ or of slides of $\bar f$.
       By Corollary \ref{liftsequence}, the sequence $S_1$ may be performed in $T$ and  the tree $S_1\cdot T$ still refines $T_e$. Up to replacing  $T$ by $S_1\cdot T$, we may thus assume that $S=S_2$. 
  
\begin{enumerate}
\item Assume that $S$ is composed of slides of $f$ and there exists an admissible slide $g/\bar f$ in $S\cdot T'$, then by Corollary \ref{liftsequence} and Lemma \ref{lift}, we may apply the sequence $S\cdot g/\bar f$ to  $T$, and obtain a tree  $(S\cdot g/\bar f) \cdot T$ which refines $T_e$. Moreover $v$ is now of valence three, we may thus apply Case \ref{troisaretes}.

\item \label{cas1-2strictasc} Assume that $S$ is composed of slides of $f$ and $f$ is \pa in $S\cdot T'$. 

By Corollary \ref{liftsequence}, we may perform $S$ in $T$, the tree $S\cdot T$ refines $T_e$, and the vertex $v$ in $S\cdot T$ has valence $2$. Thus, up to replacing $T$ by $S\cdot T$ we may assume that the sequence $S$ is empty. Now $f$ is \pa in $T'$. As in $T$, the terminal vertex of $f$ is not in $\gr v$ it must be in $\gr {v'}$, that is there exists an edge $f'$ in the orbit of $f$, with terminal vertex $v'$ and $G_{f'}\subsetneq G_f$. We may apply Lemma \ref{casalmostasc} (interchanging $f$ and $f'$). Thus the tree $T_e$ is not universally compatible, which is a contradiction.
\item \label{suiteglissementdroite} Assume here that there exists a sequence $ S\cdot g/\bar f  =\bar f/i_1\dots \bar f/i_q\cdot g/f$ admissible in $T'$.

In $T$ we may perform the sequence $\bar f /e\cdot S\cdot g/\bar f$. Up to collapsing all orbits of edges different from $\gr e$, $\gr f$ and $\gr g$, we may assume that $S$  is composed of slides of $\bar f$ along edge in $\gr g$ and $\bar {\gr g}$ (that is every $i_j$ is in $\gr g\cup \bar {\gr g}$).

\begin{lemma}\label{cas1213}
Let $T$ be a $G$-tree (not necessarily a \vGBS tree). Let $f$ and $g$ be two edges of $T'$ verifying the following properties:
\begin{itemize}
\item for $w'$ the terminal vertex of $f$, there exists an edge $f'\neq f$ in $\gr f$ with terminal vertex $w'$, such that $G_f=G_{f'}$,
\item there exists an admissible sequence $S$ of slides in $T'$, composed of slides of $\bar f$ along edges in $\gr g$ or $\bar{\gr g}$, ending with a slide of $g$ along $f$.  
\end{itemize}
Let $w$ be the initial edge of $f$ and let $H$ be a subgroup such $G_f\subsetneq H \subsetneq G_w$, call $T$ the tree obtained from $T'$ by performing an expansion of $w$ with group $H$ and set $\{f\}$ and $T_H$ the tree obtained by collapses every orbit of edges of $T$ except the one created in the expansion.
Then $T_H$ is not universally compatible.
\end{lemma}

This lemma proves that in Case \ref{suiteglissementdroite}, the tree $T_e$ is not universally compatible. Indeed taking $H=G_e$ we obtain $T_e=T_H$.

\begin{proof}[Proof of Lemma \ref{cas1213}]
We may assume that for every slide $\bar f/g'$ in $S$ we have $G_f \subsetneq G_{g'}$, otherwise $\bar g'$ may slide along $f$, we may take $g=\bar g'$ and truncate the sequence $S$.

Call $e$ the edge of $T'$ obtained in the expansion of $w$ and call $v$ and $v'$ the initial and terminal vertices of $e$ (the vertex $v$ is the initial vertex of $f$ in $T$).

These two cases appear, depending on whether $g$ is strictly ascending (or descending) or not (note that $g$ cannot be \tor, otherwise $g$ may never slide along $f$).

First assume that $g$ is not ascending. Let $c$ such that $f'=c\cdot f$. 
Let $a$, $b$ and $d$ be three elements such that $a$ fixes the initial vertex of $g$ but not $g$ (this is possible since $g$ is not ascending), the element $b$ fixes $v'$ but not $e$ and $d$ fixes $c\cdot v'$ but not $e'=c\cdot e'$. 
  
  Then  $\mathcal E_{\{a,b\}}(T_H)\cap \mathcal E_{\{c,d\}}(T_H)=\emptyset$ (since $\mathcal E_{\{a,b\}}(T)$ and $\mathcal E_{\{c,d\}}(T)$ are separated by $e$). Now perform the sequence of deformations $f/e\cdot S\cdot g/f'$ in $T$.  Let $g'$ be an edge on which $\bar f$ slides in $S$. When $g$ slides along $\bar f$ and $f'$, $g'$ slides along an edge in $\gr f\cup\bar{\gr f}$. However since $G_f\subsetneq G_{g'}$, neither $g'$ nor $\bar g'$ slides along $f$  when $g$ slides along $f$. This ensure us that $f$ and $f'$ are between $g'$ and $g$ after performing the slides. Thus the sets $\mathcal E_{\{b,c\}}$ and $\mathcal E_{\{a,d\}}=\emptyset$ are separated by $f'$ in the resulting tree. Hence $T_H$ is not universally compatible.

\begin{figure}[!ht]
\begin{center}
\include{casglissementnonreduit2}
\end{center}
\caption{}
\label{casglissementnonreduit}
\end{figure}

If $g$ is strictly ascending, as the sequence $S$ is admissible in $T$, this implies that $f$ is a loop in $T'$ (otherwise $g$ becomes non-reduced after sliding along $f$). Thus in $T$ the terminal vertex of $f$ is in $\gr v$.

Take $f''$ in the orbit of $f$ such that its terminal vertex is the initial vertex of $g$. We may take $a'$ which fixes the initial vertex of $f''$ but not $f''$.  We may then proceed exactly as in the case where $g$ is not ascending, replacing $a$ by $a'$. Note that $a'$ does not fix $g$ since $f''$ is between $g$ and the characteristic space of $a'$.
\end{proof}
\item \label{cas1-2strictdes} The last possibility is when $S$ is composed of slides of $\bar f$ and after performing the sequence $ f/e\cdot S$ in $T$ , the edge $f$ is \pde.

   Collapsing all edges along which $f$ slides in $S$, we obtain a tree $\tilde T$ in which $e$ is not collapsed, and $T_e$ is tree obtained from $\tilde T$ by collapsing all orbits of edges except $\gr e$, the edge $f$ has initial vertex $v$, there exists an edge $f'$ in the orbit $\gr f$ with terminal vertex $v'$ and $G_f\subset G_{f'}$. By Lemma \ref{casalmostasc}, the tree $T$ is not universally compatible.

\end{enumerate}
 \item \label{vannonslip}Assume that $f$ is vanishing non-slippery.

 By Lemma \ref{nonslipvanish}, the edge $f$ is not a loop and there exists in $T'$ a strictly ascending edge $g$ adjacent to $f$ and we may perform an admissible sequence $S$ of inductions on $g$ and slides of $f$ (or $\bar f$ depending on whether the vertices of $g$ are in the orbit of the initial or terminal vertex of $f$) along $\gr g$ or $\bar {\gr g}$ followed by a $\mathcal A^{-1}$-move in which $f$ vanishes. 
 
 If the endpoints of $g$ are in the same orbit as the terminal vertex of $f$ then $S$ is also admissible in $T$ and by Lemma \ref{rafinementcontinu} the resulting tree still refines $T_e$ but there is now three distinct orbits $\gr e$, $\gr g$, and $\bar {\gr g}$ adjacent to $v$, thus the case is already treated in Case \ref{troisaretes}.
 
 If $g$ is adjacent to (the image of) $v$ in $T'$,  then $g$ and $\bar {\gr g}$ are adjacent to $v'$ in $T$. Call $\bar g'$ the edge in the orbit of $\bar {\gr g}$ adjacent to $v'$, we have $G_{g'}\subsetneq G_{g}=G_{v'}$.
 
 We prove that $T_e$ is not universally compatible by induction on the length of $S$.
 
 If the sequence $S$ is empty, we have $G_{g'}\subset G_f\subset G_e\subset G_g=G_{v'}$, thus $g'$ may slide consecutively along $\bar e$ and $f$. As $f$ is reduced, $g'$ may then slide along an edge $\bar f'$ in the orbit $\bar {\gr f}$ different from $\bar f$, and then along an edge $e'$ in the orbit $\gr e$ (see Figure \ref{vanishnonreduced}).
 Collapsing $f$, the obtained tree $\tilde T$ still refines $T_e$, but then $e$ is reduced and non ascending in $\tilde T$ and $g'$ slides consecutively along two distinct edges in the orbits $\gr e\cup \bar {\gr e}$. By Corollary \ref{casreduit} case 2. the tree $T_e$ is not universally compatible.
 
 If the sequence $S$ is not empty, let $\mathbf m$ be the first move of $S$. Then $\mathbf m$ is either an induction on $\gr g$ or a slide of $f$ on $\gr g$ or $\gr g'$.
 
 If $\mathbf m$ is an induction, then by Lemma \ref{lift}, we may perform the induction in $T$, and the obtained tree still refines $T_e$, thus we may reduce the length of $S$ by one.
If $\mathbf m$ is a slide of $f$ along an edge $\tilde g$ in the orbit of $\gr g$ or $\bar {\gr g}$, then if $G_e\not \subset G_{\tilde g}$, then by Lemma \ref{casalmostasc}, the tree $T_e$ is not universally compatible. 
Else $G_e \subset G_{\tilde g}$ and $\mathbf m\cdot T'$ is refined by $e/\tilde g\cdot T$ which also refines $T_e$. Thus the sequence $S$ may be shortened.
 
 \begin{figure}[!ht]
 \begin{center}
\begin{pspicture}(0,-1.08375)(5.401875,2.06375)
\psline[ArrowInside=->, ArrowInsidePos=0.5,arrowsize=0.25291667cm,linewidth=0.04cm](2.06,0.98375)(0.06,0.18375)
\psline[ArrowInside=->, ArrowInsidePos=0.5,arrowsize=0.25291667cm,linewidth=0.04cm](2.06,-0.61625)(0.06,0.18375)
\psline[ArrowInside=->, ArrowInsidePos=0.5,arrowsize=0.25291667cm,linewidth=0.04cm](2.06,0.98375)(4.06,0.98375)
\psline[ArrowInside=->, ArrowInsidePos=0.5,arrowsize=0.25291667cm,linewidth=0.04cm](2.06,-0.61625)(4.06,-0.61625)
\psline[ArrowInside=->, ArrowInsidePos=0.5,arrowsize=0.25291667cm,linewidth=0.04cm](5.26,-0.01625)(4.06,0.98375)
\psline[ArrowInside=->, ArrowInsidePos=0.5,arrowsize=0.25291667cm,linewidth=0.04cm](4.06,0.98375)(5.26,1.91625)
\psdots[dotsize=0.12](5.26,-0.01625)
\psdots[dotsize=0.12](5.26,1.91625)
\psdots[dotsize=0.12](4.06,0.98375)
\psdots[dotsize=0.12](2.06,0.98375)
\psdots[dotsize=0.12](0.06,0.18375)
\psdots[dotsize=0.12](2.06,-0.61625)
\psdots[dotsize=0.12](4.06,-0.61625)
\usefont{T1}{ptm}{m}{n}
\rput(5.061406,0.68875){$g'$}
\rput(5.061406,1.48875){$g$}
\usefont{T1}{ptm}{m}{n}
\rput(3.0814064,0.70875){$e$}
\usefont{T1}{ptm}{m}{n}
\rput(0.80140626,0.76875){$f$}
\usefont{T1}{ptm}{m}{n}
\rput(0.68140626,-0.41125){$f'$}
\usefont{T1}{ptm}{m}{n}
\rput(3.1414063,-0.91125){$e'$}
\end{pspicture} 
 \end{center}
 \caption{}
 \label{vanishnonreduced}
 \end{figure}
      
\end{enumerate}
\end{enumerate}
\item Assume now that $G_{v'}=G_e$. 

Let  $g$ be an edge  with initial vertex $v'$, $g\neq \bar e$ (then $g$ is not in the orbit $\bar {\gr e}$).

\begin{enumerate}
 \item \label{222} Assume that there exist at least three orbits of edges $g$, $i$ and $\bar e$ adjacent to $v'$ and at least three orbits of edges $f$, $h$ and $e$ adjacent to $v$.

If one of the edges $f$, $g$, $h$ or $i$ is strictly ascending, with the same argument of \ref{strictlyasc} the edge $e$ (or $\bar e$) may slide along it and we are now in Case \ref{semireduit1}.

\begin{enumerate}
\item First assume none of the edges $f$, $g$, $h$ or $i$ is \tor.
We will find a contradiction using Lemma \ref{incompatibility}.

\begin{figure}[!ht]
\begin{center}
% Generated with LaTeXDraw 2.0.5
% Fri Feb 04 10:52:23 CET 2011
% \usepackage[usenames,dvipsnames]{pstricks}
% \usepackage{epsfig}
% \usepackage{pst-grad} % For gradients
% \usepackage{pst-plot} % For axes
\scalebox{1} % Change this value to rescale the drawing.
{
\begin{pspicture}(0,-1.3829688)(5.4228125,1.3829688)
\psline[ArrowInside=->, ArrowInsidePos=0.5,arrowsize=0.25291667cm,linewidth=0.04cm](1.9009376,-0.05546875)(0.7009375,1.1445312)
\psline[ArrowInside=->, ArrowInsidePos=0.5,arrowsize=0.25291667cm,linewidth=0.04cm](1.9009376,-0.05546875)(0.7009375,-1.0554688)
\psline[ArrowInside=->, ArrowInsidePos=0.5,arrowsize=0.25291667cm,linewidth=0.04cm](1.9009376,-0.05546875)(3.5009375,-0.05546875)
\psline[ArrowInside=->, ArrowInsidePos=0.5,arrowsize=0.25291667cm,linewidth=0.04cm](3.5009375,-0.05546875)(4.7009373,1.1445312)
\psline[ArrowInside=->, ArrowInsidePos=0.5,arrowsize=0.25291667cm,linewidth=0.04cm](3.5009375,-0.05546875)(4.7009373,-1.0554688)
\psdots[dotsize=0.12](3.5009375,-0.05546875)
\psdots[dotsize=0.12](1.9009376,-0.05546875)
\psdots[dotsize=0.12](0.7009375,1.1445312)
\psdots[dotsize=0.12](0.7009375,-1.0554688)
\psdots[dotsize=0.12](4.7009373,1.1445312)
\psdots[dotsize=0.12](4.7009373,-1.0554688)
\usefont{T1}{ptm}{m}{n}
\rput(0.26234376,1.1895312){$a$}
\rput(0.23234375,-1.2104688){$b$}
\rput(5.0423436,1.1495312){$c$}
\rput(5.112344,-1.0704688){$d$}
\rput(1.5623437,0.7295312){$f$}
\rput(3.8623438,0.70953125){$g$}
\rput(2.7423437,0.24953125){$e$}
\rput(0.89234376,-0.49046874){$h$}
\rput(4.282344,-0.37046874){$i$}
\end{pspicture} 
}
\end{center}
\caption{}
\label{egal222}
\end{figure}

We take $a$, $b$, $c$ and $d$ stabilizing respectively the terminal vertices of $f$, $h$, $g$ and $i$ but not stabilizing the edges (see Figure \ref{egal222}). In $T$ the spaces $\mathcal E_{\{a,b\}}(T)$ and $\mathcal E_{\{c,d\}}(T)$ are separated by the edge $e$ which is not collapsed in $T_e$, thus $\mathcal E_{\{a,b\}}(T_e)\cap \mathcal E_{\{c,d\}}(T_e)=\emptyset$. But making $\bar f$ and $\bar g$ slide respectively along  $e$ and $\bar e$, we obtain $\mathcal E_{\{b,c\}}\cap \mathcal E_{\{a,d\}}=\emptyset$. Thus $T_e$ is not universally compatible.

\item \label{toricpastoric} Assume that $g$ (or $i$) is \tor and $\gr h$ and $\gr f$ are not \tor. Then $\bar h$ may slide along $e$, then along $g$ and  finally along another edge $\bar e'$ in $\bar {\gr e}$. Call $w$ the terminal vertex of $h$ and $w'$ the terminal vertex of $f$. Then we may take $a$ in $G_w\setminus G_h$, $b$  in $G_{w'}\setminus G_f$, $c$ which send the initial vertex of $g$ onto its terminal vertex and $d=cbc^{-1}$ (see Figure \ref{casasc}). 

\begin{figure}[!ht]
\begin{center}
% Generated with LaTeXDraw 2.0.5
% Sun Feb 06 21:44:51 CET 2011
% \usepackage[usenames,dvipsnames]{pstricks}
% \usepackage{epsfig}
% \usepackage{pst-grad} % For gradients
% \usepackage{pst-plot} % For axes
\scalebox{1} % Change this value to rescale the drawing.
{
\begin{pspicture}(0,-1.6829687)(6.1228123,1.6829687)
\psline[ArrowInside=->, ArrowInsidePos=0.5,arrowsize=0.25291667cm,linewidth=0.04cm](1.2409375,-0.05546875)(0.2409375,1.1445312)
\psline[ArrowInside=->, ArrowInsidePos=0.5,arrowsize=0.25291667cm,linewidth=0.04cm](1.2409375,-0.05546875)(0.2409375,-1.2554687)
\psline[ArrowInside=->, ArrowInsidePos=0.5,arrowsize=0.25291667cm,linewidth=0.04cm](1.2409375,-0.05546875)(2.8409376,-0.05546875)
\psline[ArrowInside=->, ArrowInsidePos=0.5,arrowsize=0.25291667cm,linewidth=0.04cm](2.8409376,-0.05546875)(2.8409376,1.1445312)
\psline[ArrowInside=->, ArrowInsidePos=0.5,arrowsize=0.25291667cm,linewidth=0.04cm](4.4409375,1.1445312)(2.8409376,1.1445312)
\psline[ArrowInside=->, ArrowInsidePos=0.5,arrowsize=0.25291667cm,linewidth=0.04cm](4.4409375,1.1445312)(5.8409376,0.14453125)
\psdots[dotsize=0.12](0.2409375,1.1445312)
\psdots[dotsize=0.12](1.2409375,-0.05546875)
\psdots[dotsize=0.12](0.2409375,-1.2554687)
\psdots[dotsize=0.12](2.8409376,-0.05546875)
\psdots[dotsize=0.12](2.8409376,1.1445312)
\psdots[dotsize=0.12](4.4409375,1.1445312)
\psdots[dotsize=0.12](5.8409376,0.14453125)
\usefont{T1}{ptm}{m}{n}
\rput(3.7023437,1.3095312){$e'$}
\usefont{T1}{ptm}{m}{n}
\rput(1.9023438,0.12953125){$e$}
\usefont{T1}{ptm}{m}{n}
\rput(2.5423439,0.58953124){$g$}
\usefont{T1}{ptm}{m}{n}
\rput(5.4423437,0.78953123){$f'$}
\usefont{T1}{ptm}{m}{n}
\rput(0.95234376,0.76953125){$h$}
\usefont{T1}{ptm}{m}{n}
\rput(0.44234374,-0.6304687){$f$}
\usefont{T1}{ptm}{m}{n}
\rput(0.24234375,1.4895313){$a$}
\usefont{T1}{ptm}{m}{n}
\rput(0.23234375,-1.5104687){$b$}
\usefont{T1}{ptm}{m}{n}
\rput(5.8123436,-0.11046875){$d$}
\usefont{T1}{ptm}{m}{n}
\rput(3.2223437,0.48953125){$c$}
\psline[linewidth=0.03cm,arrowsize=0.05291667cm 2.0,arrowlength=1.4,arrowinset=0.4]{->}(3.0009375,-0.05546875)(3.0009375,1.0445312)
\end{pspicture} 
}
\end{center}
\caption{}
\label{casasc}
\end{figure}

In $T$ the spaces $\mathcal E_{\{a,b\}}$ and $\mathcal E_{\{c,d\}}$ are separated by the edge $e$ which is not collapsed in $T$, thus $\mathcal E_{\{a,b\}}\cap \mathcal E_{\{c,d\}}=\emptyset$  in $T_e$. But making $\bar h$ sliding consecutively along  $e$, $g$ and $\bar e'$, the spaces $\mathcal E_{\{b,c\}}$ and $\mathcal E_{\{a,d\}}$ are separated by $e'$. Thus $T_e$ is not universally compatible.

\item Up to symmetries the only case that remains is when $f$ and $g$ are both \tor.

Define $e'$ as in Case \ref{toricpastoric}. The edge $\bar f$ may slide along $e$ and $g$ and $\bar {e'}$. Collapsing $g$, we are reduced to the case where $G_e\subsetneq G_{v'}$. The difference here is that $G_{v'}$ is not abelian, however the argument of Case \ref{castoric} only use the fact that $\bar f$ slides consecutively along two edges in the orbits $\gr e$ or $\bar {\gr e}$. This is also the case here, we thus have a contradiction with the universal compatibility of $T$.
\end{enumerate}

\item \label{cas221}Assume that there are exactly two orbits of edges  $\gr e$ and $\gr f$ adjacent to $\gr v$. (Up to symmetries, this is the last remaining case).

As $f$ is collapsed in $T_{comp}$ but is not \tor in $T'$, it is slippery, \psad or vanishing non slippery.  

If $f$ is vanishing non-slippery, the case is already treated in \ref{vannonslip} (which does not use the fact that $G_e\subsetneq G_{v'}$). 
If $f$ is slippery or \psad,  by Lemma  \ref{suiteglissement}, there exists a sequence of moves $S$ admissible in  $T'$ such that in $S\cdot T'$, either $f$ is \pa or \pde, or there exists an admissible slide $g/f$ or $g/\bar f$. Moreover this sequence may be split in two subsequences $S_1$ and $S_2$, such that no move involving $f$ occurs in $S_1$, and such that %$S_2$ is composed either of slides of $f$ or of slides of $\bar f$.
%By lemma \ref{liftsequence}, the sequence $S_1$ may be performed in $T$ and  the tree $S_1\cdot T$ still refines $T_e$. We may thus assume that $S=S_2$. 
%We have $4$ cases here:
\begin{itemize}[topsep=0cm,noitemsep]
\item either $S_2$ is composed of slides of $f$, and there exists a slide along $\bar f$ in $S\cdot T'$,
\item or $S_2$ is composed of slides of $\bar f$, and there exists a slide along $f$ in $S\cdot T'$,
\item or $S_2$ is composed of slides of $f$, and $f$ is strictly ascending in $S\cdot T'$,
\item or $S_2$ is composed of slides of $\bar f$, and $\bar f$ is strictly ascending in $S\cdot T'$.
\end{itemize}

The last two cases are done in \ref{semireduitslip}.
It remains the first and second cases.

As before, we may assume that there is no strictly ascending edge with initial vertex $v'$. 

We now have to use the fact the vertex $w$ obtained by collapsing $e$ is not a dead end of wall $f$. 

Recall that in $T'$, the non-vanishing vertex $w$ is a dead of wall $f$ if:
\begin{itemize}
\item the vertex group of $G_f$ is strictly include in $G_w$,
\item for every edge $h$ with initial vertex $w$ and not in $\gr f$ we have $\langle G_f,G_h\rangle=G_w$,
\item for every edge $h$ with initial vertex $w$, we have $G_h\neq G_w$,
\item there exists an edge $g$ with initial vertex $w$ and not in $\gr f$ such that for every edge $h$ with initial vertex $w$ and not in $\gr f$ we have $G_h\subset G_g$,
\end{itemize}
and if for every tree $T''$ in the deformation space there exists an edge $f'$ (not necessarily in the orbit $\gr f$) with initial vertex $w$ having these properties.

Note that we may assume that $G_f$ is strictly include in $G_w$ (because $f$ is reduced in $T$).

We distinguish the cases depending on whether the fact that $w$ is not a dead end may be seen in $T$ or not.
\begin{enumerate}
\item \label{danslarbre} Assume that one of the listed properties does not hold in $T'$.

There are three possibilities: 
\begin{enumerate}
\item There exists an edge $h$ of initial vertex $w$ with $\langle G_f,G_h\rangle\subsetneq G_w$.

Then $G_h\subsetneq G_w$, thus in $T$ the edge is not ascending. Call $\tilde v$ the terminal vertex of $f$ and $\tilde v'$ the terminal vertex of $h$.
Define $\tilde T$ the tree obtained from $T'$ by performing an expansion on the vertex $w$ with group $G_{\tilde w'}=\langle G_f, G_h\rangle$ and set $\left\{f,h\right\}$.

Call $\tilde e$ the edge of the expansion. We prove that $\tilde T$ is not compatible with $T_e$, contradiction the universal compatibility of $T_e$.

Take $a\in G_{\tilde v}\setminus G_f$, $b\in G_{\tilde v'}\setminus G_h$ and $c$ in $G_w\setminus \langle G_f,G_h\rangle$. 
In $T_e$, the sets $\mathcal E_{a,b}(T_e)$ and $\mathcal E_{c}(T_e)$ both contain $e$, but in $\tilde T$, the sets $\mathcal E_{a,b}(\tilde T)$ and $\mathcal E_{c}(\tilde T)$ are separated be $\tilde e$. By Lemma \ref{constructionnoncomp}, the trees $T_e$ and $\tilde T$ are not compatible. We have a contradiction.

\item \label{deadendcastoric} There exists $g$ of initial vertex $w$ with $G_g=G_w$.

By Lemma \ref{lemmeascendant}, as no edge of endpoint $w$ is strictly ascending, the edge $g$ is \tor and remains \tor with initial vertex $w$ in any tree of the reduced deformation space.

We use now the fact that there exists a sequence $S=S_1\cdot S_2$  such that $S_2$ which is composed either of slide of $\bar f$, that finishes by a slide of an edge $g'$ along $f$, or of slide of $f$, that finishes by a slide of an edge $g'$  along $\bar f$.

In both cases we may assume that $g'\neq g$. Indeed, after preforming $S$ in $T'$, the edge $g$ is still \tor. Thus if $g$ slides along $f$ then $f$ is ascending. Let $S_1$ be the minimal initial subsequence of $S$ such that $f$ is ascending in $S_1\cdot T'$. Either $f$ is strictly ascending, or it is \tor, and then by Lemma \ref{lemmeascendant}, $f$ may be change in a strictly ascending edge in one slide. Thus up to change $S$, we may assume that $f$ is \pa, and the case is already done.

%We may thus assume that $g'\neq g$.
If $S_2$ is composed of slides of $f$, then by Corollary \ref{liftsequence} and Lemma \ref{lift}, the tree $(S\cdot g'/f)\cdot T$ refines $T_e$, but now there $v$ is of valence $3$ (with orbits $\gr f$, $\gr {\bar g'}$ and $\gr e$) and $v'$ is of valence at least $3$ (with orbits $\gr g$, $\gr {\bar g}$ and $\gr {\bar e}$). We are reduced to Case \ref{222}.

If $S_2$ is composed of slides of $\bar f$, as $g'\neq g$, first call $T_1=S_1\cdot T$. Then collapsing in $T_1$ the orbit $\gr g$ and every edge not in $\gr f$, $\gr e$ and $\gr g'$ in $T$, we obtain a tree $\tilde T$ in which there still exists an admissible sequence of slides of $\bar f$ along edges in $\gr g'$ and $\bar {\gr g'}$, followed by a slide $g'/f$, but now $G_e\subsetneq G_{v'}$. We may now apply Lemma \ref{cas1213} in which $T = \tilde T$ and $T_e = T_H$ with $H=G_e$.

\item There are two edges $g$ and $g'$ such that $G_{g'}\not\subset G_g$ and $G_g\not\subset G_{g'}$, and there is no edge $h$ such that $G_{g'}\subset G_h$ and $G_g\subset G_h$.
If there exists an edge $h$ with initial vertex $v'$ such that $G_f\subset G_h$, or $G_f\subset G_h$, by Lemma \ref{casfacile}, the tree $T_e$ is not universally elliptic.

Thus the case where $S_2$ is composed of slide of $\bar f$ is not possible.
It remains the case where $S_2$ is composed of slide of $f$, and then there exists a slide $h/\bar f$.

By Corollary \ref{liftsequence} and Lemma \ref{lift} we may perform $(S\cdot h/\bar f)$ in $T$, and the obtained tree still refines $T_e$. 
The conditions on $f$, $g$ and $g'$ implies, that $G_f$, $G_g$ and $G_{g'}$ are maximal for the inclusion between groups of edge with an endpoint $w$. Applying \cite[Proposition 4.10]{GL07}, there are at least $3$ edges with initial vertex $w$ and distinct maximal edge groups  in $(S\cdot h/\bar f)\cdot T'$. As $G_h\subset G_f$, the vertex $w$ is of valence at least $4$. As in $(S\cdot h/\bar f)\cdot T$ the vertex $v$ is of valence $3$ (including the one added by $e$), this implies that $v'$ is of valences at least $3$, we are thus reduced to Case \ref{222}.

\end{enumerate}
\item Assume that there exists another tree $T''$ in which $w$ vanishes or in which the listed properties does not hold.

Let $S'$ be an admissible sequence such that at the end $\gr w$ vanishes, or the listed properties does not hold in $T''=S'\cdot T'$.
\begin{enumerate}
\item \label{suitecool} Assume that no moves involving $\gr f$ or $\bar {\gr f}$ (except slides of $f$), occurs in $S'$.

Then by Corollary \ref{liftsequence}, we may apply $S'$ to $T$ and the obtained tree is still refines $T_e$. 

If one of the listed properties does not hold in $T''$, we are reduced to Case \ref{danslarbre}.

If $w$ vanishes,  it is now of valence $3$ with a strictly ascending vertex $g$. If $g$ is still strictly ascending in $S'\cdot T$ (that is the endpoints of $g$ are in $\gr v'$, then $e$ may slide along $g$. We obtained a new tree which still refines $T_e$ in which $G_e\subsetneq G_{v'}$. We may thus apply Case \ref{semireduit1}. 

If $g$ is no more ascending in $S\cdot T$, then its endpoints are in the orbit $\gr v$ and $\gr v'$, and by Lemma \ref{casalmostasc} the tree $T_e$ is not universally elliptic.
\item If a move involving $\gr f$ or $\bar {\gr f}$ different from a slide of $f$ occurs in $S'$, call $S'_1$ the maximal initial subsequence of $S$ with no moves involving $f$ and let $\mathbf m$ be the next move.
Then $S'_1\cdot T$ refines $T_e$, and the vertex $v$ is still of valence $2$. If $\mathbf m$ is a slide on $\bar f$, in $(S'_1\cdot \mathbf m)\cdot T$ the vertex $v$ is now of valence $3$. If $v'$ is of valence greater than $3$ then we may apply Case \ref{222}, if $v'$ is of valence $2$ in $(S'_1\cdot \mathbf m)\cdot T$, then we may exchange the role of $v$ and $v'$ and continue to preform the sequence $S$ (which is strictly shorter).

If $\mathbf m$ is an induction or an $\mathcal A^{\pm1}$-move on $f$, then $f$ is \pa in $S'_1\cdot T'$, and the terminal vertex of $f$ is in the orbit $v'$. By Lemma \ref{casalmostasc} the tree $T_e$ is not universally compatible.

If  $\mathbf m$ is an $\mathcal A^{-1}$-move in which $f$ vanishes, if the move occurs at $w$, then $w$ is the vanishing vertex and the case is already treated in \ref{suitecool}. If the move occurs at the terminal vertex of $f$, then $(S'_1\cdot \mathbf m)\cdot T$ still refines $T_e$ and now the valence of $v$ is $3$, as previously, if the valence of $v'$ is greater than $3$, we may apply Case \ref{222},  and if $v'$ is of valence $2$ in $(S'_1\cdot \mathbf m)\cdot T$, then we may exchange the role of $v$ and $v'$ and continue to preform the sequence $S'$ (which is strictly shorter).

If $\mathbf m$ is a slide of $\bar f$ along an edge $g$. If $G_g\neq G_w$ then Corollary \ref{casfacile}, the tree $T_e$ is not universally elliptic.

If $G_g=G_w$ then we are in Case \ref{deadendcastoric}.
\hfill\qed
\end{enumerate}
\end{enumerate}
\end{enumerate}
\end{enumerate}
%\end{proof}

\begin{theorem}
 Let $G$ be  a \GBSn{n} group which is not an ascending HNN-extension, then $T_{comp}$ is the compatibility JSJ tree for decomposition over $\Z^n$ groups.
\end{theorem}

\begin{proof}\label{GGBSn} By Proposition \ref{GGBSncomp},  the tree $T_{comp}$ is universally compatible.
By Lemma \ref{rafinementdeT}, Propositions \ref{reduced} and \ref{maxcasnonreduit}, the tree $T_{comp}$ dominates every universally compatible tree with one edge.
By \cite[Proposition 3.22]{Gl3b}, this imply that $T_{comp}$ dominates every universally compatible tree.
\end{proof}

\section{\texorpdfstring{The Case of \vGBS groups}{}}\label{sectioninert}
In the previous section, we showed that $T_{comp}$ is the compatibility JSJ tree of a \GBSn{n} group over $\Z^n$ groups.
However some \GBSn{n} groups may split over  $\Z^{n+1}$. Thus the abelian compatibility JSJ tree may be different from $T_{comp}$. This is actually the case even for very simple \GBSn{1} groups such as  $BS(2,2)=\langle a, t | ta^2t^{-1}=a^2\rangle$.

Indeed, the reduced (abelian and cyclic) JSJ deformation space of $BS(2,2)$ is reduced to one tree with one orbit of (non-oriented) non-ascending edges. Thus $T_{comp}$ is also this JSJ tree.

However $BS(2,2)$ may split over $\Z^2$ is the following way: define $$A=\langle a,s | sa^2s^{-1}=a^2\rangle\simeq BS(2,2)$$ and $$B=\langle b, t | bt=tb\rangle\simeq \Z^2.$$
Define $G=A\ast_{a^2=b,s=t^2}B$, then $$G=\langle a,t | ta^2t^{-1}=a^2\rangle\simeq BS(2,2).$$
The element $t^2$ is hyperbolic in the JSJ decomposition, but belongs to the edge group of the decomposition $A\ast_{a^2=b,s=t^2}B$. Thus the tree associated to the decomposition $A\ast_{a^2=b,s=t^2}B$ is not compatible with the JSJ tree (which is $T_{comp}$). This implies that $T_{comp}$ is not the abelian compatible JSJ tree.
In fact in this case the compatibility JSJ tree is the trivial $G$-tree.

%In this section, we assume that $G$-trees have free abelian edge groups.
\subsection{Definitions and construction tools}

An element $t\in G$ is \textit{bi-elliptic} in a $G$-tree $T$ if it stabilizes an edge of the tree, \textit{\pbe}if there exists a reduced $G$-tree over abelian groups in which it is bi-elliptic. 
An edge is \textit{\bea}if it is contained in the axis of some hyperbolic element  which is \pbe.

Let $T$ be a reduced $G$-tree. Let $e$ be an edge of $T$ with initial vertex $v$. Let $w$ be the vertex obtained by collapsing the ascending edges adjacent to $v$. Assume that $G_v$ is abelian.

\paragraph{Inert edges} Let $e$ be an edge of $T$ verifying the following four properties in every tree of the deformation space of $G$ :
\begin{itemize}
\item the initial vertex $v$ of $e$ is of valence at least $2$,
\item the edge $e$ is \bea,
\item the edge $e$ is not vanishing, not slippery and not \psad,
\item the centralizer of $G_e$ in $G_w$ is $G_v$.
\end{itemize}

The edge $e$ is \textit{inert} if furthermore it has one of the following properties:
\begin{itemize}
\item the group $G_v/G_e$ is not isomorphic to $\Z/p^k\Z$ for any prime number $p$, and for all $f\not\in \gr{e}$ with initial vertex $v$, we have $\langle{G_f,G_e}\rangle=G_v$,
 \item there exit a prime number $p$ and an integer $k$ such that $G_v/G_e\simeq \Z/p^k\Z$, and for all $f\not\in \gr{e}$ of initial vertex $v$, we have the implication if $G_e \subset G_f$  then $ f$ is \tor,
\end{itemize} and if $e$ has these properties in every tree of the reduced deformation space (note that as $e$ is not \psad, we may follow $e$ in every tree of the reduced deformation space).

The edge $e$ is said to be \textit{inert of type 1} in the first case and \textit{inert of type 2} in the second.

An orbit of edge $\gr e$ is \textit{inert} if one of its representative is inert.

The reason why we have to distinguish inert edges of type $1$ and $2$ comes from the following fact:

\begin{fact}
Let $G_e$ be a subgroup of $G_v\simeq \Z^n$.
The subgroups of $G_v$ containing $G_e$ are pairwise comparable for inclusion if and only if $G_v/G_e\simeq  \Z/p^k\Z$ for some prime number $p$ and integer $k$.
\end{fact}

\begin{lemma}\label{forcementtoric}
If $e$ is inert then $G_v$ is maximal abelian and if there exists an edge $f\not\in \gr{e}$ of initial vertex $v$ such that $G_e\subset G_f$ then $f$ is \tor.
\end{lemma}

\begin{proof}
Let us first prove the implication $G_e\subset G_f \Rightarrow f$ is \tor.

If $e$ is of type 2, this point is in the definition. If $e$ is of type 1, call $v'$ the terminal vertex of $f$. By definition $\langle{G_f,G_e}\rangle=G_v$ thus the edge $f$ is ascending and $e$ may slide along $f$, thus after sliding as $e$ is still inert we must have $\langle G_e, G_f\rangle=G_{v'}$, as $G_e\subset G_f$ we have $G_f=G_{v'}$ and $f$ is \tor.

A consequence is that an ascending edge adjacent to $v$ is \tor. Call $\mathcal V$ the subtree fixed by $G_v$. Then every vertex of $\mathcal V$ is in $\gr v$ and every edge of $\mathcal V$ is \tor, thus $\mathcal V$ collapses onto $w$. Any element that commute to $G_v$ must either fix a point of $\mathcal V$ if it is elliptic, or have its axis contained in $\mathcal V$ if it is hyperbolic. In both cases it belongs to $G_w$. %$\mathcal Thus an element which commute with $G_v$ belongs to $G_w$, thus by the fifth point it must belong to $G_v$, that is $G_v$ is maximal abelian.
\end{proof}

\subsection{Compatibility JSJ tree over abelian groups}\label{compunivvgbs}

\paragraph{Expansion of inert edges}

If an edge $e$ is inert of type $1$, to \textit{expand} $e$ consists in performing an expansion on $v$ with group $G_{v}$ and set $\left\{e\right\}$  (see definition in Section \ref{prelim}).

If an edge $e$ is inert of type 2, we define $H_e$ as the largest subgroup such that for every tree in the reduced deformation space of $T$, and every edge $f$ with same initial vertex as $e$, we have $H_e \subset \langle G_e, G_f\rangle$.

The edge $e$ is also \bea, and  $G_v/G_e\simeq \Z/p^k\Z$. Let $a$ be an element of $G_v$ whose projection into $G_v/G_e$ is a generator. Let $i$ be the minimal integer such that there exist $q$ with $p\not | q$ and a \pbe element whose axis contains both $e$ and $a^{p^iq}\cdot e$. Note that as $a^{p^k}\cdot e = e$, we have $i\leq k$.  Define $F_e=\langle G_e, a^{p^i}\rangle$

If $F_e \subset H_e$, to \textit{expand} $e$ consists in performing an expansion on $v$ with group $F_e$ and set $\left\{e\right\}$.
The expansion of $e$ is not defined when $F_e \not \subset H_e$

\paragraph{Construction of the abelian compatibility JSJ tree}
\begin{definition}\label{Tab}
Let $G$ be a \vGBS group.
Let $T$ be a reduced abelian JSJ tree of $G$.
Call $\mathcal E$ the set of non-ascending slippery,  strictly ascending  and \tor $2$-slippery edges.

We define $T_{ab}$ as the $G$-tree obtained by expanding inert edges, blowing up dead ends, then collapsing simultaneously every edge of $\mathcal E$ and every \bea edges.
\end{definition}

Note that the wall of a dead end cannot be inert. There may be several blow up and expansion around a vertex, but the blowing up of dead ends, and the expansion of inert edges are made on distinct edges. 

\begin{lemma}
The tree $T_{ab}$ does not depend on the choice of $T$ in the reduced JSJ deformation space of $G$.
\end{lemma}

%\begin{proof}
\proof
We show that the tree obtained before collapsing \bea edges does not depend on $T$.

Using Proposition \ref{construction}, we just have to prove that the expansion of an inert edge commutes with the Whitehead moves and collapses.

Let  $T$ be a reduced JSJ tree and $\mathbf m$ be an admissible move on $T$. 
\begin{itemize}
\item Assume that $\mathbf m$ is a slide $e/f$. Let $g$ be an inert edge. If $g\neq f, \bar f, \bar e$, then obviously the slide commutes with the expansion. As $f$ and $\bar f$ are slippery, they are not inert. If $\bar e$ is inert, by Lemma \ref{forcementtoric}, the edge $f$ is \tor. Then either $f$ is $2$-slippery and must be collapsed, and thus it is the same to expand $\bar e$ and collapse $f$ than sliding $e$, expand $\bar e$ and collapse $f$, or $f$ is not $2$-slippery, thus the slide has no effect on the tree \cite[Theorem 1, case 3]{Levitt05}.

%Now if $g$ is a \bea edge, then if $g\neq f, \bar f$ then it stays \bea after the slide. It remains to prove that if $g=f$ or $\bar f$ then $g$ is also collapsed in $\mathbf m \cdot T$. But $f$ is slippery thus, either $f$ is not \tor, then $f$ is collapsed in $\mathbf m \cdot T$, or $f$ is \tor $2$-slippery, then again $f$ is collapsed in $\mathbf m \cdot T$, or $f$ is \tor but not two slippery, then by \cite[Theorem 1, case 3]{Levitt05} the trees $T$ and $\mathbf m \cdot T$ are equal and thus $f$ is \bea in$\mathbf m \cdot T$ hence collapsed.

\item If $\mathbf m$ is an induction, or an $\mathcal A^{-1}$-move on an edge $e$ of initial vertex $v$, as $e$ is strictly ascending, the group $G_v$ is not maximal abelian, thus no edge adjacent to $v$ is inert. \hfill\qed
\end{itemize}
%\end{proof}

Call $\tilde T_{ab}$ the tree obtained from $T_{ab}$ by collapsing edges coming from an expansion or a blow-up.

Note that $\tilde T_{ab}$ is refined by every JSJ tree.

\begin{lemma}\label{univell}
The tree $T_{ab}$ is universally elliptic.
\end{lemma}

\begin{proof}
As every edge of $\tilde T_{ab}$ comes from an edge of a reduced JSJ tree, the tree $\tilde T_{ab}$ is universally elliptic.

We have to show that the edges obtained by a blow up or an expansion are universally elliptic.
Let $T_J$ be an abelian JSJ tree, and let $e$ be an edge of $T_J$ of initial vertex $v$.

If the vertex $v$ is a dead end with wall $e$ or if $e$ is an inert edge of type $1$, let $f\not \in \gr e$ be an edge of $T_J$ with initial vertex $v$.  Then $G_v=\langle G_e, G_f\rangle$, from \cite[Theorem 7.3]{Moi1} this implies that $G_v$ is abelian. The group $G_v$ is abelian and generated by universally elliptic elements thus universally elliptic. Thus the edge comming from a blow up of $v$ or an expansion of $e$ is universally elliptic

If $e$ is inert of type $2$, call $f$ the edge coming from the expansion. Then $G_e$ is universally elliptic and of finite index in $G_f$, thus $G_f$ is also universally elliptic.
\end{proof}
\subsection{Universal compatibility}

\begin{lemma}\label{univcomploc}
The tree $T_{ab}$ is compatible with every abelian JSJ tree $T$.
\end{lemma}

\begin{proof}
For $f$ an edge of $T_{ab}$ call $T_f$ the tree obtained by collapsing every edge of $T_{ab}$ except $f$.

By the point $1.$ of \cite[Proposition 3.22]{Gl3b}, we have to show that $T$ is compatible with $T_f$ for every edge $f$ of $T_{ab}$. 

Call $\tilde T$ a reduced JSJ tree refined by $T$. By construction, $\tilde T$ refines every tree $T_f$ with $f$ an edge of $T_{ab}$ not obtained by a blow up or an expansion, thus refines $\tilde T_{ab}$. Hence $T$ also refines $\tilde T_{ab}$.

If $f$ comes from a blow up or an expansion, then we may apply Lemma \ref{compexpand}. We obtain that $T$ is compatible with  $T_f$.

If follows that $T$ is compatible with $T_{ab}$.
\end{proof}

\begin{lemma}\label{elliptic}
Let $G$ be a \vGBS group. Any $G$-tree over abelian groups is elliptic with respect to $T_{ab}$ and $\tilde T_{ab}$
\end{lemma}

\begin{proof}
Let $E$ be an abelian subgroup of $G$ and assume that $G$ splits non-trivially over $E$ as an amalgam or an HNN extension.
By definition every element of $E$ is \pbe. As every \bea edges are collapsed, every element of $E$ is elliptic in $T_{ab}$, but as $E$ is  finitely generated, this implies that $E$ is elliptic.
\end{proof}

\begin{lemma}\label{dominerafine}
 Let $T_1$ and $T_2$ be two reduced universally elliptic $G$-trees. If every JSJ tree refines $T_2$ and if $T_1$ dominates $T_2$ then $T_1$ refines $T_2$.
\end{lemma}

\begin{proof}
 By \cite[Lemma 5.3]{Gl3a}, there exists a JSJ tree $T_{JSJ}$ which refines $T_1$. As any JSJ refines $T_2$, the tree $T_{JSJ}$ also refines $T_2$. We prove that every edge of $T_{JSJ}$ collapsed in $T_1$ is also collapsed in $T_2$. 

We may collapse one by one the non-reduced edges of $T_{JSJ}$ which must be collapsed in $T_1$. The remaining tree is still a JSJ tree, hence still refines $T_2$. Let $e$ be a reduced edge that is collapsed in $T_1$. As $e$ is reduced, there are more elliptic elements after collapsing $e$. Yet $T_1$ dominates $T_2$, thus elliptic of $T_1$ are elliptic of $T_2$. Necessarily, the edge $e$ is collapsed in $T_2$.
\end{proof}

\begin{lemma}\label{estcompatible}
Let $T_1$ and $T_2$ be $G$-trees. If every JSJ tree refines $T_2$ and if $T_1$ is elliptic with respect to $T_2$, then $T_1$ and $T_2$ are compatible.
\end{lemma}

\begin{proof}
Note that as $T_2$ is refined by every JSJ tree it is reduced.

 As $T_1$ is elliptic with respect to $T_2$, by \cite[Lemma 3.2]{Gl3a} there exists $T'$ a refining tree of $T_1$ which dominates $T_2$. As $T_2$ is  universally elliptic, the tree $T''$ obtained from $T'$ by collapsing the non universally elliptic edges still dominates $T_2$. We may also assume that $T''$ is reduced.
The trees $T''$ and $T_2$ are universally elliptic and reduced. By Lemma \ref{dominerafine} the tree $T''$ refines $T_2$. Thus $T'$ is a common refinement of $T_1$ and $T_2$.
\end{proof}

\begin{corollary}\label{tildetab}
Let $G$ be a \vGBS group. Then $\tilde T_{ab}$ is universally compatible.
\end{corollary}

\begin{proof}
That is an immediate consequence of Lemmas \ref{elliptic} and \ref{estcompatible}.
\end{proof}

\begin{lemma}\label{commut}
Let $T_J$ be a reduced JSJ tree. Let $e$ be an inert edge. Call $T'_J$ the tree obtained by expending $e$ in $T_J$. Call $f$ the new edge of $T'_J$ adjacent to $e$. Call $v$ and $v'$ the initial and terminal vertices of $f$, and $T_v$ and $T_{v'}$ the connected components of $v$ and $v'$ in $T'_J\setminus f$. Then\begin{itemize}
\item every $t$ whose characteristic space is contained in $T_{v'}$ and which belongs to the centralizer of $G_e$ belongs to $G_{v'}$,
\item for every edge $e'\neq f$ adjacent to $v'$ in $T_{v'}$ and every element $t$ whose characteristic space is contained in $T_{v}$, the group $G_{e'}$ is not contained in the centralizer of $t$.
\end{itemize}
Moreover, in both point, we may replace $G_e$ and $G_{e'}$ by $G_f$.

\end{lemma}

\begin{proof}
Call $w$ the vertex obtained by collapsing in $T_J$ every \tor edge adjacent to the initial vertex of $e$.
Let $t$ be an element whose characteristic space is contained in $T_{v'}$. Assume that it commutes with every element of $G_e$. If $t$ is hyperbolic then $G_e$ is contained in every edge group of it axis. But applying Lemma \ref{forcementtoric}, we easily see that every edge of the axis is \tor, and that the whole axis of $t$ in $T_J$ collapses onto $w$. Thus by the fourth point of the definition of an inert edge, the element $t$ belongs to $G_{v'}$.

If $t$ is elliptic, then again applying Lemma \ref{forcementtoric}, then path between $v'$ and the characteristic space of $t$ contains only \tor edges. Thus $t$ belongs to $G_w$, thus to $G_{v'}$. 

As $G_e \subset G_f$ the result is also true for $G_f$.

Let us treat now the second point of the lemma. As $e$ is inert, it is neither slippery nor \pa. Thus $G_{e'}$ is not include in $G_e$. As every edge of $T_v$ adjacent to $v$ has edge group $G_e$, if $G_{e'}$ centralizes an element $t$ with characteristic space intersecting $T_v$, then $t$ belongs to $G_v$. As $G_v\subset G_{v'}$, the characteristic space of $t$ is not contained in $T_v$.

As $G_f$ is not include in $G_e$, the result is also true replacing $G_{e'}$ by $G_f$.
\end{proof}

Let $e$ be an edge of $T_{ab}$ coming from the expansion of an inert edge or the blow up of a dead end. Call $T_e$ the tree obtained by collapsing every edge in $T_{ab}$ except the ones in $\gr e$.

\begin{lemma}\label{localcomp}
Let $T$ be an abelian $G$-tree which dominates $T_e$, then $T$ is compatible with $T_e$. 
\end{lemma}

\begin{proof}The method consists in recreating the edge $e$ in $T$.

Let $T'$ be the $G$-tree obtained from $T$ by collapsing every non-universally elliptic edge. By Lemma \ref{univell} l'arbre $T_e$ is universally elliptic, thus the tree $T'$ still dominates $T_e$.

Now, $T'$ is universally elliptic, thus refined by a JSJ tree. By Lemma \ref{univcomploc}, the tree $T_e$ is compatible with any JSJ tree, therefore $T'$ and $T_e$ are compatible. Call $T'_e$ the least common refinement. The tree $T'_e$ is also universally elliptic. As $T'$ dominates $T_e$, the trees $T'$ and $T'_e$ belong to the same deformation space.

If $T'_e=T'$ then $T$ refines $T_e$ thus is compatible with it.
Otherwise, $T'_e$ has one orbit of edges $\gr e$ more than $T'$ and this orbit is not reduced. Call $v$ and $v'$ the endpoints of $e$ in $T'_e$ such that $G_{v'}=G_{e}\subset G_v$. 

Fix $\tilde T_e$ a JSJ tree refining $T'_e$. We may assume that if an edge of $\tilde T_e$ is not reduced, then it is not collapsed in $T'_e$. Call also $e$ the unique lift of $e$ in $\tilde T_e$. As $G_e=G_{v'}$, the vertex  $v'$ as a unique preimage in $\tilde T_e$. Indeed the terminal vertex $\tilde v$ of $e$ in $\tilde T_e$ has group $G_{v'}$ and collapses onto $v'$, if an edge adjacent to $\tilde v$ collapse onto $v'$, as $G_{\tilde v}=G_{v'}$, this edge may not be reduced. This is in contradiction with the construction of $\tilde T_e$.

Let $\widehat T_e$ be a reduced JSJ tree refined by $\tilde T_e$. As $e$ in $T_{ab}$ comes from an expansion or a blow up, call $\hat v$ the vertex on which $e$ collapses in $\widehat T_e$. Let $\mathcal G$ be the set of edges of $\widehat T_e$ adjacent to $\hat v$. 

 Let $f\neq e, \bar e$ be an edge of $T'_e$ with initial vertex $v'$, call $\tilde f$ its lift in $\tilde T_e$. Then $\tilde f$ has initial vertex $\tilde v$. If $\tilde f$ is reduced then it is the lift of an edge of $\mathcal G$. If it is not reduced, as $\tilde T_e$ is minimal, it is in the path of between $\tilde v$ and the lift of some edge $g$ of $\mathcal G$. But $G_g\subset G_{\tilde v'}$, thus $G_g\subset G_{\tilde f}$. Thus in both cases there exists an edge $g\in \mathcal G$ such that $G_g\subset G_f$.

Let $w'$ be the vertex on which $e$ collapses in $T'$. We have $G_w=G_v$. Call $T_w$ the subtree of $T$ which collapses on $w'$.
As $G_e\subset G_w$ is universally elliptic, there exists a vertex $w$ of $T_w$ fixed by $G_e$. We claim that $w$ is uniquely determined.

Indeed, the tree $T'$ is obtained from $T$ by collapsing non-universally elliptic edges. Thus $T_w$ has no universally elliptic edges. But the centralizer of $G_e$ in $G_{w'}$ is $G_{\hat v}$ (see Lemma \ref{commut}) which is universally elliptic. Thus if an edge of $T_w$ is fixed by $G_e$, its stabilizer is contained in $G_{\hat v}$, and thus universally elliptic. Which is absurd.

Call $\mathcal F$ the set of edges of $T'_e$ with initial vertex $v'$. Take $f\in \mathcal F$ with $f\neq \bar e$ an edge of $T'_e$ of initial vertex $v'$, as $f$ is not collapsed in $T'$, it has a unique lift in $T$ (that we also call $f$), let us show that $f$ has initial vertex $w$ in $T$.

 There exists $g\in \mathcal G$ such that  $G_g\subset G_f$, now by Lemma \ref{commut}, the element which commutes with $G_f$ and fixes a point in $T_w$ belongs to $G_{\hat v}$. But again, as $G_{\hat v}$ is universally elliptic, no edge of $T_w$ has its stabilizer contained in $G_{\hat w}$, thus the initial edge of $f$ must be $w$. We may thus perform and expansion on $w$ with group $G_e$, and set $\mathcal F$. 
 
 The obtained tree refines $T'_e$, thus $T_e$.
\end{proof}

\begin{proposition}
The tree $T_{ab}$ is universally compatible.
\end{proposition}

\begin{proof}
Let $\tilde T$ be an abelian $G$-tree. By Lemma \ref{elliptic}, the tree $\tilde T$ is elliptic with respect to $T_{ab}$, applying \cite[Lemma 3.2]{Gl3a}, there exists $T$ a refining tree of $\tilde T$ which dominates $T_{ab}$. Let $e$ be an edge of $T_{ab}$, and $T_e$ the tree obtained by collapsing every edge orbit except $\gr e$.

If $e$ is not collapsed in $\tilde T_{ab}$, applying Corollary \ref{tildetab}, the tree $T$ is compatible with $T_e$.
If $e$ is collapsed in $\tilde T_{ab}$, applying  Lemma \ref{localcomp}, the tree $T$ is compatible with $T_e$. From \cite[Proposition 3.14]{Gl3b}, the $T$ is compatible with $\tilde T$, and consequently $\tilde T$ is compatible with $T_{ab}$. 
\end{proof}

\subsection{Maximality}
\begin{proposition} Let $G$ be a \GGBS group. The tree $T_{ab}$ dominates every universally compatible $G$-tree.
\end{proposition}

\begin{proof}

Let $T'$ be a universally compatible $G$-tree with a unique orbit of edges $\gr e$. Let $\tilde T_J$ be a JSJ tree which refines $T'$. We may assume that $\tilde T_J$ has a most one orbit of non-reduced edges and if so this orbit is $\gr e$.

If $\tilde T_J$ is reduced, the only case not treated in Proposition \ref{reduced}, is when $\gr e$ is \bea. But if $\gr e$ is \bea, this exactly means that there exists an hyperbolic element of $T'$ which belongs to an edge group of another abelian $G$-tree. This immediately contradict the universal compatibility of $T'$.

If $\tilde T_J$ is not reduced, recall the notation of Proposition \ref{maxcasnonreduit}, fix $e$ a representative of $\gr e$, call $v$ the initial vertex of $e$ and $v'$ its terminal vertex. We assume that $G_v=G_e$. Call $w$ the vertex obtained by collapsing every \tor edge adjacent to $v'$. The vertex $v$ is of valence at least $2$, and $f\not \in \gr e$ is an edge of initial vertex $v$. The cases not treated in Proposition \ref{maxcasnonreduit} are the ones where $f$ is an inert edge and $e$ is not the expansion on $f$. To be more specific, after excluding the subcases already treated, the remaining cases are

the valence of $v$ is $2$, the edge $f$ is \bea but not \pa, not \pde, not slippery, not vanishing (and not \tor) and 
\begin{enumerate}
\item either $G_e=G_v\subsetneq G_{v'}$ and there is no prime $p$ such that $G_{v'}/G_f\simeq \Z/p^k\Z$,
\item or the centralizer of $G_e$ in $G_w$ strictly contains $G_v$.
\item or $G_e=G_v\subsetneq G_{v'}$, the edge $f$ is inert of type 2, and $G_e\subsetneq F_f$,
\item or $G_e=G_v\subsetneq G_{v'}$ and $f$ is  inert of type 2, and $F_f \subsetneq G_e$,%  \subsetneq H_f$.
\item or one of the property of inert edges does not hold for $f$ in a reduced  JSJ tree  $T'_J$.
\end{enumerate}

Remember that if $f$ is inert of type 2, as $G_{v'}/G_f \simeq \Z/p^k\Z$ with $p$ a prime number, we must have $G_e\subset H_f$ or $H_f\subset G_e$. The definition of $F_f$ is given in Section \ref{compunivvgbs}.

In the two first cases there exists an abelian  subgroup $E$ of $G_w$ containing $G_f$ and not comparable with $G_e$. Let $a$ be an element in $G_e$ but not in $E$ and $b$ an element in $E$ but not in $G_e$. 

Let $\widehat T$ be the tree obtained from $T_J$ by performing an expansion on $v'$ with group $E$, and set $\left\{f\right\}$.

Let $c$ be an element fixing the terminal edge of $f$ but not $f$ (such an element exists since $f$ is not ascending).
Then $a$ fixes the edge $e$ in $\mathcal E_{\left\{c,bcb^{-1}\right\}}(T_e)$, but no vertex of $\mathcal E_{\left\{c,bcb^{-1}\right\}}(\widehat T)$, by Lemma \ref{constructionnoncomp}, these $T_e$ and $\widehat T$ are not compatible. Thus $T_e$ is not universally compatible.

In the third case, by construction of $F_f$, there exist an element $a\in G_{v'}\subsetneq F_f$ and a \pbe element whose axis contains $f$ and $a\cdot f$, as $e$ is between $f$ and $a\cdot f$ this implies that $e$ is \bea, thus $T_e$ cannot be universally compatible.

In the fourth case, the tree $T_{ab}$ obviously dominates $G_e$.

In the last case, call $T_J$ the reduced JSJ obtained by collapsing $e$ in $\tilde T_J$. There exists a sequence $S$ of admissible moves such that $S\cdot T_J=T'_J$. 
We prove that we may apply a sequence $S'$ at $\tilde T_J$ such that $S\cdot \tilde T_J$ still refines $T'$. As $f$ is not slippery, and not \pa, the only way $f$ may appear in $S$ is in a slide $f/g$ or $\bar f/g$. From Lemma \ref{lift}, if no slide of $\bar f$ occurs, then $S$ is admissible in $\tilde T_J$ and $S\cdot \tilde T_J$ refines $T'$. If a slide of $\bar f$ along an edge $g$ occurs, then  $g$ is \tor. Thus in $\tilde T_J$ it suffices to replace $f/g$ by $e/g$.
%\subsection{Particular case}
%
%An edge $e$ is called \textit{mobile} if there exists an element $g\in G$ such that $G_{g\dot e} \subsetneq G_e$.
%
%\begin{proposition}
%Let $T$ be a \vGBS tree with no mobile edge. Then a 
%\end{proposition}
\end{proof}

\section{Algorithmic construction}
\subsection{\texorpdfstring{Inconstructibility of  the abelian compatibility JSJ of \vGBS groups}{}}

In \cite{BoMaVe}, Bogopolski, Martino and Ventura show that there exist four automorphisms $(\varphi_i)_{1\leq i\leq4}$ of $\Z^4$ such that the semi-direct product $G=\Z^4\rtimes_{(\varphi_i)} F_4$, where $F_4$ is the free group in four generators, has an undecidable conjugacy problem. From this example, we give a family of \vGBS groups for which it is impossible to describe algorithmically the abelian compatibility JSJ tree.

\begin{proposition}\label{JSJconj}
 Let  $G$ the group described in the previous paragraph and $T$ be the $G$-tree associated to the presentation with one orbit of vertices (whose stabilizer are all equal to $\Z^4$) and four orbits of edges.
Let $x_1$ and $x_2$ be two primitive elements of $\Z^4$. We build $G(x_1,x_2)$ the fundamental group of the following graph of groups $\Gamma(x_1,x_2)$: we add to $\Gamma$ two loops $\ell_1$ and $\ell_2$ whose edge groups an both $\Z$. The injections of $G_{\ell_i}$ are defined by $n\mapsto x_i^n$. 

The group $G(x_1,x_2)$ is then isomorphic to $\langle G, t_1, t_2 | t_ix_it_i^{-1}=x_i\rangle$.
Then the compatibility JSJ tree of $G(x_1,x_2)$ is trivial if and only if $x_1$ et $x_2$ are conjugated.
\end{proposition}

\begin{proof}
Call $T(x_1,x_2)$ the Bass-Serre tree associated to $\Gamma(x_1,x_2)$. The $G$-tree $T(x_1,x_2)$ is an abelian JSJ tree of $G(x_1,x_2)$. The four \tor edges are collapsed in $T_{ab}$ since they are $2$-slippery, and the vertex is not a dead end. As $\langle x_1, x_2 \rangle$ is not of finite index in the vertex group, no edge is inert. As $x_1$ and $x_2$ are primitive and the $\varphi$ are in $Gl_n(\Z)$, the edges $l_i$ are not \psa. Thus $l_1$ and $l_2$ are collapsed if and only if they are slippery. But the only edge which may slide along $l_1$ is $l_2$. Thus $l_1$ is collapsed if and only if $x_2$ is conjugate to a power of $x_1$ in $G$, but again as the elements $x_1$ and $x_2$ are primitive elements,  the edge $l_1$ is collapsed if and only if $x_1$ and $x_2$ are conjugated in $G$.
\end{proof}

\begin{corollary}
 There is no algorithm that permit to decide whether $G(x_1,x_2)$ has a trivial compatibility JSJ tree.
\end{corollary}

\begin{proof}
The conjugacy problem between hyperbolic elements in \GGBS groups is decidable. From  \cite{BoMaVe} we may deduce that the conjugacy problem between elliptic elements of $G$ is not solvable. Proposition \ref{JSJconj} allows to conclude.
\end{proof}

\bibliographystyle{plain}
\bibliography{beeker-compatibilityJSJ}
\vspace{0.3cm}
\begin{flushleft}

\noindent \textsc{Benjamin Beeker}\\
MAPMO,\\
Universit\'e d'Orl\'eans, UFR Sciences\\
B\^atiment de math\'ematiques - Rue de Chartres\\
B.P. 6759\\
45067 Orl\'eans Cedex 2\\
France\\
{\tt benjamin.beeker@univ-orleans.fr}
\end{flushleft}
\end{document}